\documentclass[english,linesnumbered,ruled,vlined]{article}
\usepackage[letterpaper, margin=1.2in]{geometry}
\usepackage{babel}
\usepackage{verbatim}
\usepackage{mathtools}
\usepackage{booktabs}
\usepackage{enumitem}
\usepackage{bm}
\usepackage{algorithm, algpseudocode, multirow}
\usepackage{amsmath}
\numberwithin{equation}{section}
\usepackage{amsthm}
\usepackage{amssymb}
\usepackage{minitoc}
\usepackage{color,xcolor}
\usepackage[numbers]{natbib}
\usepackage[pdfusetitle, colorlinks=false]{hyperref} 
\usepackage[all]{hypcap}
\usepackage{cleveref}
\crefname{ALG@line}{line}{lines}
\Crefname{ALG@line}{Line}{Lines}
\makeatletter
\@ifpackageloaded{hyperref}{
    \DeclareRobustCommand{\theHALG@line}{line.\thealgorithm.\arabic{ALG@line}}
}{
}
\makeatother
\makeatletter
\theoremstyle{definition}
\newtheorem{defn}{\protect\definitionname}[]
\theoremstyle{plain}
\newtheorem{thm}{\protect\theoremname}[section]
\theoremstyle{plain}
\newtheorem{proposition}{\protect\propositionname}[section]
\theoremstyle{plain}

\theoremstyle{plain}
\newtheorem{lem}{\protect\lemmaname}[section]
\theoremstyle{remark}

\theoremstyle{definition}

\providecommand{\examplename}{Example}
\theoremstyle{plain}

\newcommand{\qbf}{\mathbf{q}}

\providecommand{\assumptionname}{Assumption}
\providecommand{\definitionname}{Definition}
\providecommand{\lemmaname}{Lemma}
\providecommand{\propositionname}{Proposition}
\providecommand{\remarkname}{Remark}
\providecommand{\theoremname}{Theorem}
\providecommand{\corollaryname}{Corollary}

\crefdefaultlabelformat{#2\textbf{#1}#3} %
\crefname{section}{\textbf{section}}{\textbf{sections}}
\Crefname{section}{\textbf{Section}}{\textbf{Sections}}
\crefname{thm}{\textbf{Theorem}}{\textbf{theorems}}
\Crefname{thm}{\textbf{Theorem}}{\textbf{Theorems}}
\crefname{lem}{\textbf{Lemma}}{\textbf{lemmas}}
\Crefname{lem}{\textbf{Lemma}}{\textbf{Lemmas}}
\crefname{prop}{\textbf{proposition}}{\textbf{propositions}}
\Crefname{prop}{\textbf{Proposition}}{\textbf{Propositions}}
\crefname{algorithm}{\textbf{algorithm}}{\textbf{algorithms}}
\Crefname{algorithm}{\textbf{Algorithm}}{\textbf{Algorithms}}
\crefname{coro}{\textbf{Corollary}}{\textbf{corollaries}}
\Crefname{coro}{\textbf{Corollary}}{\textbf{corollaries}}
\crefname{defn}{\textbf{Definition}}{\textbf{definitions}}
\Crefname{defn}{\textbf{Definition}}{\textbf{definitions}}
\crefname{table}{\textbf{Table}}{\textbf{tables}}
\Crefname{table}{\textbf{Table}}{\textbf{tables}}
\crefname{figure}{\textbf{Figure}}{\textbf{figures}}
\Crefname{figure}{\textbf{Figure}}{\textbf{figures}}
\crefname{exple}{\textbf{Example}}{\textbf{examples}}
\Crefname{exple}{\textbf{Example}}{\textbf{examples}}
\Crefname{assumption}{\textbf{Assumption}}{\textbf{Assumptions}}
\crefname{assumption}{\textbf{Assumption}}{\textbf{Assumptions}}
\Crefname{rem}{\textbf{Remark}}{\textbf{Remarks}}
\crefname{rem}{\textbf{Remark}}{\textbf{Remarks}}

\usepackage[most]{tcolorbox}
\usepackage{fancybox,ulem,subfig}
\tcbuselibrary{breakable,theorems,skins}

\newtcbtheorem[number within=section]{exbox}{Example}{
  enhanced,
  colback=gray!3,           
  colframe=gray!70!black,   
  boxrule=0.6pt,
  arc=2mm,                  
  left=8pt,right=8pt,top=8pt,bottom=8pt,
  fonttitle=\bfseries,       
  attach boxed title to top left={yshift=-1mm, xshift=2mm},
}{ex}

\newif\ifcomments
\commentstrue 
\ifcomments
  \newcommand{\comm}[2][]{\textcolor{red}{\textbf{[#1:} #2\textbf{]}}}
  \newcommand{\qd}[1]{\comm[QD]{#1}}
  \newcommand{\zl}[1]{\comm[ZL]{#1}}
\else
  \newcommand{\comm}[2][]{}
  \newcommand{\qd}[1]{}
  \newcommand{\zl}[1]{}
  \newcommand{\xx}[1]{}
\fi

\providecommand{\corollaryname}{Corollary}

\newcommand{\isarxiv}{}

\makeatother

\begin{document}
\global\long\def\inprod#1#2{\left\langle #1,#2\right\rangle }%

\global\long\def\inner#1#2{\left\langle #1,#2\right\rangle }%

\global\long\def\binner#1#2{\big\langle#1,#2\big\rangle}%

\global\long\def\norm#1{\left\Vert #1\right\Vert }%

\global\long\def\bnorm#1{\big\Vert#1\big\Vert}%

\global\long\def\Bnorm#1{\Big\Vert#1\Big\Vert}%

\global\long\def\red#1{\textcolor{red}{#1}}%

\global\long\def\blue#1{\textcolor{blue}{#1}}%

\global\long\def\brbra#1{\left(#1\right)}%

\global\long\def\Brbra#1{\left(#1\right)}%

\global\long\def\rbra#1{(#1)}%

\global\long\def\lrbra#1{\left(#1\right)}%

\global\long\def\sbra#1{[#1]}%

\global\long\def\bsbra#1{\left[#1\right]}%

\global\long\def\Bsbra#1{\Big[#1\Big]}%

\global\long\def\abs#1{\vert#1\vert}%

\global\long\def\babs#1{\big\vert#1\big\vert}%

\global\long\def\lrabs#1{\left|#1\right|}%

\global\long\def\cbra#1{\{#1\}}%

\global\long\def\bcbra#1{\left\{  #1\right\}  }%

\global\long\def\Bcbra#1{\left\{  #1\right\}  }%

\global\long\def\matr#1{\bm{#1}}%

\global\long\def\til#1{\tilde{#1}}%

\global\long\def\wtil#1{\widetilde{#1}}%

\global\long\def\wh#1{\widehat{#1}}%

\global\long\def\mcal#1{\mathcal{#1}}%

\global\long\def\mbb#1{\mathbb{#1}}%

\global\long\def\mtt#1{\mathtt{#1}}%

\global\long\def\ttt#1{\texttt{#1}}%

\global\long\def\dtxt{\textrm{d}}%

\global\long\def\aeq{\overset{(a)}{=}}%

\global\long\def\bignorm#1{\bigl\Vert#1\bigr\Vert}%

\global\long\def\Bignorm#1{\Bigl\Vert#1\Bigr\Vert}%

\global\long\def\rmn#1#2{\mathbb{R}^{#1\times#2}}%

\global\long\def\deri#1#2{\frac{d#1}{d#2}}%

\global\long\def\pderi#1#2{\frac{\partial#1}{\partial#2}}%

\global\long\def\limk{\lim_{k\rightarrow\infty}}%

\global\long\def\trans{\textrm{T}}%

\global\long\def\onebf{\mathbf{1}}%

\global\long\def\zerobf{\mathbf{0}}%

\global\long\def\zero{\bm{0}}%


\global\long\def\Euc{\mathrm{E}}%

\global\long\def\Expe{\mathbb{E}}%

\global\long\def\rank{\mathrm{rank}}%

\global\long\def\range{\mathrm{range}}%

\global\long\def\diam{\mathrm{diam}}%

\global\long\def\epi{\mathrm{epi} }%

\global\long\def\inte{\operatornamewithlimits{int}}%

\global\long\def\dist{\operatornamewithlimits{dist}}%

\global\long\def\proj{\operatornamewithlimits{Proj}}%

\global\long\def\cov{\mathrm{Cov}}%

\global\long\def\argmin{\operatornamewithlimits{argmin}}%

\global\long\def\argmax{\operatornamewithlimits{argmax}}%

\global\long\def\where{\operatornamewithlimits{where}}%

\global\long\def\conv{\operatornamewithlimits{conv}}%

\global\long\def\tr{\operatornamewithlimits{tr}}%

\global\long\def\dist{\operatorname{dist}}%

\global\long\def\sign{\operatornamewithlimits{sign}}%

\global\long\def\prob{\mathrm{Prob}}%

\global\long\def\st{\operatornamewithlimits{s.t.}}%

\global\long\def\dom{\mathrm{dom}}%

\global\long\def\prox{\mathrm{prox}}%

\global\long\def\diag{\mathrm{diag}}%

\global\long\def\and{\mathrm{and}}%

\global\long\def\as{\textup{a.s.}}%

\global\long\def\ae{\textup{a.e.}}%

\global\long\def\Var{\operatornamewithlimits{Var}}%

\global\long\def\Cov{\operatornamewithlimits{Cov}}%

\global\long\def\raw{\rightarrow}%

\global\long\def\law{\leftarrow}%

\global\long\def\Raw{\Rightarrow}%

\global\long\def\Law{\Leftarrow}%

\global\long\def\vep{\varepsilon}%

\global\long\def\dom{\operatornamewithlimits{dom}}%

\global\long\def\err{\operatorname{err}}%

\global\long\def\soc{\operatorname{soc}}%

\global\long\def\rsoc{\operatorname{rsoc}}%

\global\long\def\tsum{{\textstyle {\sum}}}%

\global\long\def\Cbb{\mathbb{C}}%

\global\long\def\Ebb{\mathbb{E}}%

\global\long\def\Fbb{\mathbb{F}}%

\global\long\def\Nbb{\mathbb{N}}%

\global\long\def\Rbb{\mathbb{R}}%

\global\long\def\extR{\widebar{\mathbb{R}}}%

\global\long\def\Pbb{\mathbb{P}}%

\global\long\def\Mrm{\mathrm{M}}%

\global\long\def\Acal{\mathcal{A}}%

\global\long\def\Bcal{\mathcal{B}}%

\global\long\def\Ccal{\mathcal{C}}%

\global\long\def\Dcal{\mathcal{D}}%

\global\long\def\Ecal{\mathcal{E}}%

\global\long\def\Fcal{\mathcal{F}}%

\global\long\def\Gcal{\mathcal{G}}%

\global\long\def\Hcal{\mathcal{H}}%

\global\long\def\Ical{\mathcal{I}}%

\global\long\def\Kcal{\mathcal{K}}%

\global\long\def\Lcal{\mathcal{L}}%

\global\long\def\Mcal{\mathcal{M}}%

\global\long\def\Ncal{\mathcal{N}}%

\global\long\def\Ocal{\mathcal{O}}%

\global\long\def\Pcal{\mathcal{P}}%

\global\long\def\Scal{\mathcal{S}}%

\global\long\def\Tcal{\mathcal{T}}%

\global\long\def\Xcal{\mathcal{X}}%

\global\long\def\Ycal{\mathcal{Y}}%

\global\long\def\Zcal{\mathcal{Z}}%

\global\long\def\i{i}%

\global\long\def\abf{\mathbf{a}}%

\global\long\def\bbf{\mathbf{b}}%

\global\long\def\cbf{\mathbf{c}}%

\global\long\def\ebf{\mathbf{e}}%

\global\long\def\fbf{\mathbf{f}}%

\global\long\def\hbf{\mathbf{h}}%

\global\long\def\qbf{\mathbf{q}}%

\global\long\def\gbf{\mathbf{g}}%

\global\long\def\lambf{\bm{\lambda}}%

\global\long\def\alphabf{\bm{\alpha}}%

\global\long\def\sigmabf{\bm{\sigma}}%

\global\long\def\thetabf{\bm{\theta}}%

\global\long\def\deltabf{\bm{\delta}}%

\global\long\def\lbf{\mathbf{l}}%

\global\long\def\ubf{\mathbf{u}}%

\global\long\def\pbf{\mathbf{\mathbf{p}}}%

\global\long\def\vbf{\mathbf{v}}%

\global\long\def\wbf{\mathbf{w}}%

\global\long\def\xbf{\mathbf{x}}%

\global\long\def\ybf{\mathbf{y}}%

\global\long\def\zbf{\mathbf{z}}%

\global\long\def\dbf{\mathbf{d}}%

\global\long\def\Wbf{\mathbf{W}}%

\global\long\def\Abf{\mathbf{A}}%

\global\long\def\Gbf{\mathbf{G}}%

\global\long\def\Ubf{\mathbf{U}}%

\global\long\def\Pbf{\mathbf{P}}%

\global\long\def\Ibf{\mathbf{I}}%

\global\long\def\Ebf{\mathbf{E}}%

\global\long\def\Mbf{\mathbf{M}}%

\global\long\def\Dbf{\mathbf{D}}%

\global\long\def\Qbf{\mathbf{Q}}%

\global\long\def\Lbf{\mathbf{L}}%

\global\long\def\Pbf{\mathbf{P}}%

\global\long\def\Xbf{\mathbf{X}}%

\global\long\def\abm{\bm{a}}%

\global\long\def\bbm{\bm{b}}%

\global\long\def\cbm{\bm{c}}%

\global\long\def\dbm{\bm{d}}%

\global\long\def\ebm{\bm{e}}%

\global\long\def\fbm{\bm{f}}%

\global\long\def\gbm{\bm{g}}%

\global\long\def\hbm{\bm{h}}%

\global\long\def\pbm{\bm{p}}%

\global\long\def\qbm{\bm{q}}%

\global\long\def\rbm{\bm{r}}%

\global\long\def\sbm{\bm{s}}%

\global\long\def\tbm{\bm{t}}%

\global\long\def\ubm{\bm{u}}%

\global\long\def\vbm{\bm{v}}%

\global\long\def\wbm{\bm{w}}%

\global\long\def\xbm{\bm{x}}%

\global\long\def\ybm{\bm{y}}%

\global\long\def\zbm{\bm{z}}%

\global\long\def\Abm{\bm{A}}%

\global\long\def\Bbm{\bm{B}}%

\global\long\def\Cbm{\bm{C}}%

\global\long\def\Dbm{\bm{D}}%

\global\long\def\Ebm{\bm{E}}%

\global\long\def\Fbm{\bm{F}}%

\global\long\def\Gbm{\bm{G}}%

\global\long\def\Hbm{\bm{H}}%

\global\long\def\Ibm{\bm{I}}%

\global\long\def\Jbm{\bm{J}}%

\global\long\def\Lbm{\bm{L}}%

\global\long\def\Obm{\bm{O}}%

\global\long\def\Pbm{\bm{P}}%

\global\long\def\Qbm{\bm{Q}}%

\global\long\def\Rbm{\bm{R}}%

\global\long\def\Ubm{\bm{U}}%

\global\long\def\Vbm{\bm{V}}%

\global\long\def\Wbm{\bm{W}}%

\global\long\def\Xbm{\bm{X}}%

\global\long\def\Ybm{\bm{Y}}%

\global\long\def\Zbm{\bm{Z}}%

\global\long\def\lambm{\bm{\lambda}}%

\global\long\def\alphabm{\bm{\alpha}}%

\global\long\def\albm{\bm{\alpha}}%

\global\long\def\taubm{\bm{\tau}}%

\global\long\def\mubm{\bm{\mu}}%

\global\long\def\yrm{\mathrm{y}}%

\global\long\def\rone{\text{\ensuremath{\brbra{\textrm{I}}}}}%

\global\long\def\rtwo{\brbra{\text{II}}}%

\global\long\def\rthree{\brbra{\text{\textrm{III}}}}%

\global\long\def\rfour{\brbra{\text{\textrm{IV}}}}%

\global\long\def\rfive{\brbra{\text{V}}}%

\global\long\def\rsix{\brbra{\text{\textrm{VI}}}}%

\global\long\def\rseven{\brbra{\text{VI\textrm{I}}}}%

\global\long\def\reight{\brbra{\text{VI\textrm{I}I}}}%

\global\long\def\aleq{\overset{(a)}{\leq}}%

\global\long\def\bleq{\overset{(b)}{\leq}}%

\global\long\def\cleq{\overset{(c)}{\leq}}%

\global\long\def\dleq{\overset{(d)}{\leq}}%

\global\long\def\ageq{\overset{(a)}{\geq}}%

\global\long\def\bgeq{\overset{(b)}{\geq}}%

\global\long\def\cgeq{\overset{(c)}{\geq}}%

\global\long\def\beq{\overset{(b)}{=}}%

\global\long\def\ceq{\overset{(c)}{=}}%

\global\long\def\deq{\overset{(d)}{=}}%

\global\long\def\vbfp{\vbf_{\text{p}}}%

\global\long\def\vbfd{\vbf_{\text{d}}}%

\global\long\def\tp{t_{\text{p}}}%

\global\long\def\td{t_{\text{d}}}%

\global\long\def\tab{\qquad}%

\global\long\def\btab{\hspace{1.2cm}}%

\global\long\def\bbtab{\hspace{1.8cm}}%

\global\long\def\Lin{\operatorname{Lin}}%

\global\long\def\Span{\operatorname{Span}}%

\global\long\def\supp{\operatorname{supp}}%

\global\long\def\holder{\text{Hölder}}%
\global\long\def\apex{\text{APEX}}%
\global\long\def\pws{\text{PWS}}%
\global\long\def\rapex{\text{r}\apex}%
\global\long\def\flag{\text{\textbf{Flag}}}%
\global\long\def\false{\text{\textbf{\text{False}}}}%
\global\long\def\true{\text{\textbf{True}}}%
\global\long\def\Wcer{\text{W-certificate}}%
\global\long\def\lowerflag{\text{\textbf{Lower-Flag}}}%
\global\long\def\Adet{\mbb A_{\text{det}}}%
\global\long\def\Azr{\mbb A_{\text{zr}}}%
\global\long\def\calZA{\mcal Z_{\mcal A}}%
\global\long\def\onestep{\text{One-Step}}%
\global\long\def\quarflag{\text{\textbf{Cert-Flag}}}%
\global\long\def\wrapex{\texttt{wrAPEX}}%
\global\long\def\maxquad{\texttt{MAXQUAD}}%

\newcommand{\jim}[1]{\textcolor{red}{\textbf{#1}}}
\newcommand{\lzw}[1]{\textcolor{blue}{\textbf{#1}}}

\def \fstar {f^*}
\def \Lavg {L_\text{avg}}

\global\long\def\bundleSize{B}%

\title{Accelerated Prox-Level Methods for Unknown Piecewise-Smooth Optimization I: Convex Optimization}

\author{Zhenwei Lin\thanks{lin2193@purdue.edu, School of Industrial Engineering, Purdue University} \qquad\qquad\quad Zhe Zhang \thanks{zhan5111@purdue.edu, School of Industrial Engineering, Purdue University}  }

\maketitle

\begin{abstract}
    We introduce a nearly parameter-free algorithm for minimizing piecewise smooth (PWS) convex functions under the quadratic-growth (QG) condition, where the locations and structure of the smooth regions are entirely \textit{unknown}. Our algorithm, \apex{} (Accelerated Prox-Level method for Exploring Piecewise Smoothness), is an accelerated bundle-level method designed to adaptively exploit the underlying PWS structure. For this setting, APEX achieves the best-known oracle-complexity result among existing first-order methods, improving the dependence on the condition number relative to prior bundle-level guarantees. Furthermore, APEX generates a verifiable and accurate termination certificate, enabling a robust, nearly parameter-free implementation. To the best of our knowledge, APEX is the first
algorithm to simultaneously achieve the best-known first-order oracle complexity for PWS optimization and provide
certificate guarantees.
\end{abstract}


\section{Introduction}


Nonsmooth objective functions are central to modern optimization, particularly in machine learning and operations research. Yet, nonsmoothness remains a major hurdle for first-order methods required in large-scale settings; the linear convergence rates of first-order methods for smooth, strongly convex problems \cite{lan2020first,nesterov2013introductory} deteriorate to sublinear rates for their nonsmooth counterparts. This fundamental gap is well-documented in lower complexity theory \cite{nemirovskij1983problem}. To overcome this limitation, algorithms must exploit the underlying structure of the objective. In this work, we investigate how to exploit piecewise-smooth structures to bridge the divide between smooth and nonsmooth optimization.

Specifically, we consider the problem of minimizing a nonsmooth convex function $f(x)$ over a convex set $X$:
\begin{equation}
f^{*}:=\min_{x\in X} f(x). \label{eq:problem}
\end{equation}
Under the assumption that the objective $f$ satisfies the quadratic-growth (QG) condition (see Definition~\ref{def:QG}), we investigate the oracle complexity of first-order methods for finding an $\vep$-optimal solution, measured by the number of (sub)gradient evaluations required to reach a point $x$ such that $f(x) - \fstar \leq \vep$.
We focus on \textit{unknown} piecewise smooth (PWS) functions: the domain $X$ admits a finite partition into regions on each of which the function is smooth, but the partition is not known a priori (Definition~\ref{def:piecewise_smooth}).
Such functions arise, for example, in quadratic Model Predictive Control (MPC)~\cite{kouvaritakis2016model} and two-stage stochastic linear programming (SLP)~\cite{dantzig1955linear}, where the regions are determined by the active constraints and the optimal dual solution in the second-stage problem, respectively, and hence are not available a priori. They also arise in ReLU networks~\cite{nair2010rectified}, where regions are determined by the composition of activation functions. As network depth increases, these pieces become more difficult to characterize a priori.



The challenge is to design algorithms that automatically exploit this $\pws$ structure to surpass the performance of generic nonsmooth methods~\cite{drusvyatskiy2018error,drusvyatskiy2019efficiency} and, ideally, to match the performance of smooth optimization. While existing work has predominantly addressed the special case where the objective is a maximum of smooth components~\cite{carmon2021thinking,nesterov2005smooth}, the general setting of unknown PWS structure remains largely unexplored. In particular, three major gaps remain in the current literature: the lack of accelerated global convergence guarantees, sensitivity to unavailable problem-specific parameters, and the lack of efficient, accurate, and verifiable certificate generation.

\paragraph{\uline{Lack of accelerated global convergence rates for unknown PWS optimization.}}
Existing work provides several approaches to PWS optimization, frequently leveraging the fact that convex PWS functions can be represented as the maximum of smooth component functions. One  line of work~\cite{carmon2021thinking,nesterov2013introductory} directly targets this representation; when gradients for all underlying components are available, prox-linear methods~\cite{nesterov2013introductory,zhang2022solving,lan2023optimal,srd} can achieve global linear convergence. In the more general "unknown" setting, where only the gradient of the currently active component is accessible, the survey descent method by Han and Lewis~\cite{osti_10410938} utilizes gradients at multiple "survey points" to capture the local max-of-smooth structure. While this approach achieves local linear convergence, global convergence guarantees remain unavailable.

Another research direction with significant empirical success is the bundle method~\cite{de2014level,oliveira2014bundle,lemarechal1995new,kiwiel1995proximal,mifflin2012science}. It aggregates past subgradients to build a cutting-plane model of the objective, and generates new iterates by projecting a candidate point onto a level set defined by this model. Theoretical justification for its superior empirical performance was established in the recent work of Zhang and Sra~\cite{zhang2025linearly}, which proved that the bundle-level (BL) method~\cite{lemarechal1995new} solves unknown PWS problems with the same oracle complexity as smooth optimization.

However, as an unaccelerated method, the BL method is inherently suboptimal for smooth problems. It remains unclear whether its accelerated counterpart, the APL method \cite{lan2015bundle}, can effectively leverage PWS structures. Our preliminary empirical evidence, illustrated in Figure~\ref{fig:The-illustration-demonstrates-2}, suggests a negative answer: the theoretically slower BL method consistently outperforms both APL and its restarted variants on the $\texttt{MAXQUAD}$ problem, a prototypical example of PWS functions.
This performance discrepancy highlights the need for a new algorithmic framework that bridges the gap between acceleration and structural exploitation.
\textit{We aim to design an accelerated method that effectively exploits unknown piecewise structures.}


\paragraph{\uline{Sensitivity to unavailable problem-specific parameters.}}
Efficient accelerated algorithms typically require knowledge of the QG modulus $\mu$ to achieve fast convergence. However, because the QG condition is defined globally, $\mu$ is often difficult to estimate in practice~\cite{lan2023optimal}. Although Lan et al. recently proposed parameter-free methods for smooth optimization that operate without knowing $\mu$~\cite{lan2023optimal}, it remains an open question whether similar guarantees can be achieved in the more challenging PWS setting.

Beyond the QG modulus, the convergence of first-order methods is also governed by the empirical Lipschitz smoothness constant, defined at iteration $k$ as:
\begin{equation}\label{eq:empirical_Lipschitz_constant}
    \hat{L}_k = \frac{2 \left[ f(x^{k+1}) - f(x^k) - \inner{f^{\prime}(x^k)}{x^{k+1} - x^k} \right]_+}{\norm{x^{k+1} - x^k}^2}.
\end{equation}
In smooth optimization, the literature generally distinguishes between two types of complexity guarantees: those based on the worst-case empirical smoothness constant, $L_{\max} = \max_{k} \hat{L}_k$, and those based on an averaged smoothness constant $L_{\text{avg}}$, typically a weighted average of $\{\hat{L}_k\}$. While analysis in terms of $L_{\max}$ is often more straightforward~\cite{li2025simple,lan2023optimal}, bounds based on $L_{\text{avg}}$ are typically tighter and more reflective of superior empirical performance~\cite{nesterov2015universal}. In PWS optimization, the gap between $L_{\max}$ and $L_{\text{avg}}$ can be even more substantial due to spikes in $\hat{L}_k$ that occur when iterates cross into different smooth regions (see Example~\ref{ex:intro-core}).
\textit{It is therefore natural to seek a parameter-free algorithm that, without knowledge of $\mu$, attains an accelerated complexity bound associated with the average empirical Lipschitz smoothness constant, $\Lavg$.}

\begin{exbox}{Average versus worst-case empirical Lipschitz smoothness constants}{intro-core}
    In the PWS setting, the gradient is continuous within each smooth region but may be discontinuous across their boundaries. This discontinuity can cause the empirical Lipschitz smoothness constant $\hat{L}_k$ \eqref{eq:empirical_Lipschitz_constant} to spike when iterates $x^k$ and $x^{k+1}$ belong to different regions. Since the worst-case empirical Lipschitz smoothness constant $L_{\max}$ is determined by these spikes, it can be significantly larger than the average $L_{\text{avg}}$.

    To illustrate this, consider the iterates generated by applying the Polyak-step to the following one-dimensional objective function with two smooth regions, $\{x \geq 0\}$ and $\{x < 0\}$:
    $$\min_{x \in \bsbra{-M, +\infty}} f(x;i) = (ix + 1) \cdot \onebf_{\bcbra{x \geq 0}} + e^x \cdot \onebf_{\bcbra{x < 0}}.$$
    The objective is parameterized by $i \in \mathbb{N}_+$ and $M \in \mathbb{R}_+$. Starting from an initial point $x^0 = 1/i$, we obtain $L_{\max} = i^2 e^{-1/i}/2$, while $L_{\text{avg}} \rightarrow e^{-M}$ as the method progresses. Consequently, the ratio $L_{\max} / L_{\text{avg}}$ can be arbitrarily large for certain problem instances. Figure~\ref{fig:worst_average_L} illustrates how these two Lipschitz constants evolve over iterations.\end{exbox}

\paragraph{\uline{Lack of efficient, accurate, and verifiable certificates.}}
In convex optimization, beyond identifying an accurate candidate solution, it is equally important to provide a termination certificate that verifies the solution's quality, specifically the optimality gap. Such certificates are essential for early termination and for developing adaptive methods that perform in-situ adjustments to unknown problem parameters. In smooth optimization, the gradient norm is the standard choice for this purpose. However, in nonsmooth optimization, despite significant effort, the existing approaches suffer from either computational intractability or limited accuracy. For instance, the gradient norm is no longer a reliable indicator of optimality, computing a Moreau-stationary certificate is  expensive~\cite{moreau1965proximite,davis2018stochasticsubgradientmethodconverges}, and Goldstein stationarity~\cite{goldstein1977optimization,davis2022gradient,lin2022gradient} cannot readily accommodate the set constraint $X$. For PWS optimization, a more recent approach called the Normalized Wolfe Certificate ($\Wcer$)~\cite{zhang2025linearly} partially addresses these issues and has been used to design algorithms adaptive to an unknown QG modulus $\mu$. Nevertheless, this certificate remains relatively inaccurate, as the guaranteed optimality gap can be off by a factor of $O(L/\mu)$. This inaccuracy leads to suboptimal oracle complexity bounds for the resulting adaptive methods~\cite{zhang2025linearly}.
\textit{This gap in the literature prompts us to ask whether it is possible to efficiently generate a termination certificate that is accurate to within a constant $O(1)$ factor. }

\begin{figure}
\raggedright{}%
\begin{minipage}[t]{0.45\columnwidth}%
\begin{center}
\includegraphics[width=\textwidth]{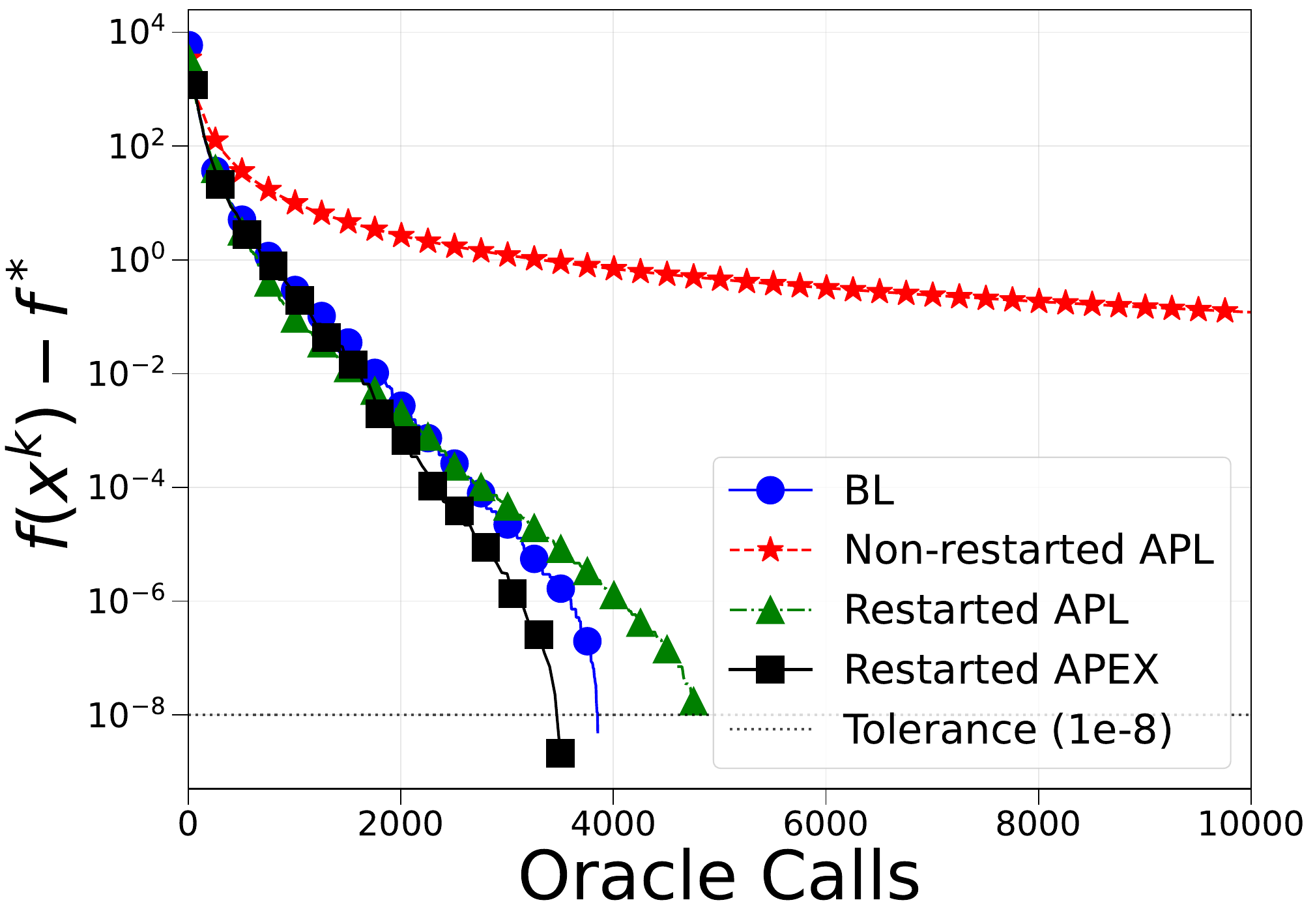}\caption{Convergence comparison between BL, APL, and Restarted APL on a randomly generated $\maxquad$ problem. \label{fig:The-illustration-demonstrates-2}}
\par\end{center}
\end{minipage}
\hfill
\begin{minipage}[t]{0.45\columnwidth}%
\begin{center}
\includegraphics[width=\textwidth]{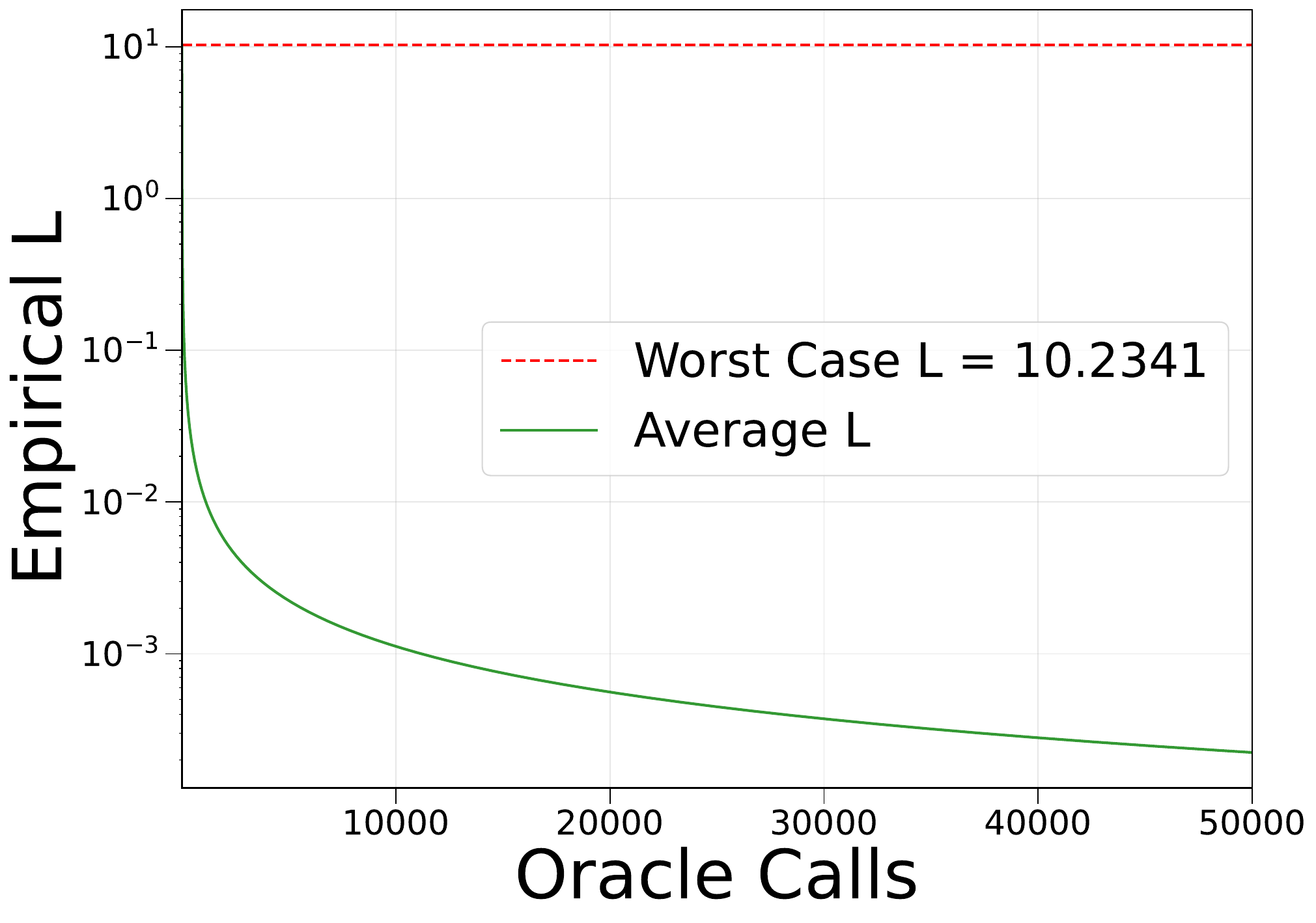}\caption{
    Worst and average empirical Lipschitz smoothness constant comparison when running Polyak-step in Example~\ref{ex:intro-core} with parameter $i=5$, $M=10^{10}$.
    \label{fig:worst_average_L}}
\par\end{center}%
\end{minipage}
\end{figure}

In light of these observations, the core challenges of unknown PWS optimization can be distilled into a single, fundamental question:
\[\ovalbox{\begin{minipage}{0.8\columnwidth - 2\fboxsep - 0.8pt}%
    \centering \it
    Can we design a parameter-free algorithm that achieves the accelerated convergence rate for unknown PWS optimization while simultaneously generating an accurate termination certificate?
\end{minipage}}
    \]

In this work, we answer this question in the affirmative by introducing \apex{}, a novel BL method that systematically addresses all three major gaps identified above.


\begin{table}[htbp]
    \centering
    \caption{
        Comparison of key properties between different methods for PWS optimization. "$\checkmark$": can handle unknown PWS optimization; "$\dagger$": using the proof technique developed in our paper, we obtain average empirical Lipschitz smoothness constant dependence for APL; however, the original analysis in~\cite{lan2015bundle} yields dependence on the worst-case empirical Lipschitz smoothness constant. QCQP: Quadratically Constrained Quadratic Programming; QP: quadratic programming; bounded LP: linear programming over a bounded region.
    }
      \begin{tabular}{llll}
      \toprule
      \multicolumn{1}{c}{\multirow{2}[2]{*}{Methods}} & \multicolumn{1}{c}{Unknown} & \multicolumn{1}{c}{Subproblem} & \multicolumn{1}{c}{First-order oracle}  \\
            & \multicolumn{1}{c}{ PWS} & \multicolumn{1}{c}{type} &     \multicolumn{1}{c}{complexity}   \\
      \midrule
      Survey Descent~\cite{osti_10410938} & $\checkmark$ & QCQP & local linear rate \\
      BL~\cite{zhang2025linearly}    & $\checkmark$ & QP & $O(1)\left(k^2L_{\mathrm{avg}}/\mu^* \cdot \log\brbra{{1}/{\varepsilon}}\right)$ \\
      APL~\cite{lan2015bundle}   & unknown & QP + bounded LP & $O(1)\left(\sqrt{L_{\mathrm{avg}}}/{\sqrt{\varepsilon}}\right)^{\dagger}$\\
      $\apex$ (ours)  & $\checkmark$ & QP & $O(1)\left( k\sqrt{L_{\mathrm{avg}}/\mu^*}\cdot \log\brbra{{1}/{\varepsilon}}\right)$\\
      \bottomrule
      \end{tabular}%
    \label{tab:complexity_comparison}%
  \end{table}%

Our contributions are organized as follows:
\begin{enumerate}
\item \textbf{Accelerated convergence rate for PWS optimization.}
We introduce an accelerated bundle-level method termed \textsc{APEX} (Accelerated Prox-Level method for Exploring piecewise-smooth structure) in Section~\ref{sec:Accelerated-Prox-level-method}. We prove that provided the bundle size is at least the number of relevant pieces encountered along the trajectory, \textsc{APEX} achieves an optimal iteration complexity of $O\left(k \sqrt{L/\mu^*} \log(1/\vep)\right)$ to compute an $\vep$-optimal solution for objectives satisfying $(k, L)$-PWS (Definition~\ref{def:piecewise_smooth}) and $\mu^*$-QG (Definition~\ref{def:QG}) conditions. This result significantly improves upon the $O\left(k^2 L/\mu^* \log(1/\vep)\right)$ complexity reported in \cite{zhang2025linearly}. Our theoretical advancements translate to superior practical performance, as demonstrated in Figure~\ref{fig:The-illustration-demonstrates-2}. Consequently, to the best of our knowledge, \apex{} is the first first-order method in this oracle model that achieve an accelerated convergence rate for unknown piecewise-smooth optimization.

\item \textbf{Accelerated certificate generation for PWS optimization.}
In Section~\ref{sec:Normalized-Wolfe-Certificate}, we introduce an accelerated bundle-level (BL) framework capable of constructing the $\Wcer$ (see Definition~\ref{def:W_certificate}) with significantly higher accuracy than previous methods~\cite{zhang2025linearly}. While prior certification techniques could only bound the optimality gap within a factor of $O(L/\mu)$, our approach improves this bound to a constant factor, $O(1)$. This enhancement allows the certificate to serve as a robust proxy for the true optimality gap. Consequently, the certificate becomes a critical tool for dynamically updating the target level (an estimate of the optimal objective value, $f^*$), an essential parameter for the bundle-level method. Leveraging this accurate dynamic update, we design a new parameter-free algorithm, as detailed in our next contribution.

\item \textbf{Adaptive and nearly parameter-free methods for PWS optimization.}
We develop $\rapex$ in Section~\ref{sec:Restarted_mu_unknown}, a restarted and nearly parameter-free variant of $\apex$. It only requires the number of cuts in the cutting-plane model is large enough to capture the relevant smooth pieces encountered along the trajectory. It adaptively exploits the problem structure and provably achieves the accelerated complexity with respect to the average empirical Lipschitz smoothness constant, making it both theoretically sound and practically efficient. To our knowledge, $\rapex$ is the first nearly parameter-free method for convex unknown PWS optimization under QG that attains the accelerated complexity. {A detailed comparison of the properties of different methods for PWS optimization is shown in Table~\ref{tab:complexity_comparison}.}

\item \textbf{Promising practical performance.}
We validate our approach through experiments on two representative PWS problems: \maxquad{} and two-stage stochastic linear programming problems. The \maxquad{} problem, a simulated setting, allows precise control over both the smoothness and the number of pieces in the PWS function, enabling rigorous verification of our theoretical claims. For two-stage stochastic linear programming, a widely studied real-world application, we assess the practical performance of our algorithm. Across both tasks, our results demonstrate that $\apex$ delivers strong practical effectiveness for unknown PWS optimization problems.
\end{enumerate}

\subsection{Related Work}

\paragraph{Algorithms for nonsmooth optimization: }
\begin{enumerate}[label=\arabic*)]
\item \textbf{Bundle methods}.  Bundle methods, rooted in Kelley's
cutting-plane scheme~\cite{kelley1960cutting},
iteratively refine a piecewise linear approximation of the objective.
Two main variants have been developed:
the proximal bundle method
and the bundle-level method,
distinguished by whether the linear model
enters the subproblem as part of the objective (proximal)
or as a constraint (level).
\begin{enumerate}
    \item \textbf{Proximal bundle method.} The proximal bundle method was introduced in the 1970s~\cite{lemarechal2009extension,mifflin1977algorithm}
    and usually comprises two
    parts: serious steps and null steps. Serious steps are used to ensure
    a decrease of the objective, and null steps are used to refine the information
    of the descent direction~\cite{bagirov2014introduction}.  The complexity
    analysis for the proximal bundle method has been widely studied in
    works~\cite{kiwiel2000efficiency,du2017rate,diaz2023optimal,liang2021proximal}.
    For convex functions admitting $\mcal{VU}$ structure~\cite{mifflin2000mathcalvu},
    a variant of bundle methods that incorporates second-order information
    relative to a smooth subspace enjoys superlinear convergence in terms
    of serious steps~\cite{mifflin2005algorithm}. Furthermore, the proximal
    bundle method has been extended to handle
    inexact oracles~\cite{de2014convex} and non-convex objectives~\cite{hare2010redistributed,de2019proximal}.
    \item \textbf{Bundle level method.} Introduced by Lemaréchal, Nesterov, and Nemirovski~\cite{lemarechal1995new},
    the bundle-level (BL) method achieves an oracle complexity of $O(1/\vep^{2})$
    for Lipschitz continuous nonsmooth convex problems.
    Since then, it has been extended to constrained convex programs,
    saddle-point problems, and variational inequalities~\cite{kiwiel1995proximal};
    generalized through non-Euclidean Bregman divergences~\cite{ben2005non};
    accelerated via Nesterov's techniques~\cite{lan2015bundle};
    and adapted to function-constrained settings~\cite{deng2024uniformly}.
    More recently, Zhang and Sra~\cite{zhang2025linearly} showed that the BL method
    attains linear convergence under the QG condition in PWS optimization.
    However, the resulting complexity's dependence
     on the condition number is suboptimal.
\end{enumerate}

\item \textbf{Goldstein method.} The seminal work of Goldstein~\cite{goldstein1977optimization} introduced the $\delta$-Goldstein subdifferential: $\partial_{\delta}f\brbra{x}:=\conv\brbra{\cup_{y\in \mcal B(x,\delta)}\partial f(y)}$ and defines a $\brbra{\delta,\vep}$-Goldstein stationary point as one satisfying $\min\bcbra{\norm{g}:g\in \partial_{\delta}f(x)}\leq \vep$. Goldstein~\cite{goldstein1977optimization} originally proposed a subgradient method using the minimal-norm element of $\partial_{\delta} f(x)$, which guarantees descent but is intractable to compute in practice. The gradient sampling algorithm~\cite{burke2002approximating,burke2020gradient,burke2005robust} circumvents this by approximating the $\delta$-Goldstein subdifferential through random sampling within $\mcal B(x,\delta)$. More recent work~\cite{pmlr-v119-zhang20p,davis2022gradient,tian2022finite} develops efficient randomized approximations for finding a $\brbra{\delta,\vep}$-Goldstein stationary point. Recently, Davis and Jiang~\cite{davis2024local} further proposed a randomized scheme of the Goldstein method that achieves nearly local linear convergence by exploring the underlying smooth substructures. Kong and Lewis~\cite{kong2025lipschitz} identify abstract properties that enable such near-linear convergence in Goldstein-type methods.

\item \textbf{Structured nonsmooth optimization.} To address the challenges of nonsmoothness, considerable effort has focused on exploiting problem structure. Sparse optimization is a classical example, where composite optimization methods~\cite{beck2009fast,nesterov2013gradient} handle objectives combining a smooth term with a prox-friendly nonsmooth regularizer. Problems involving the maximum of smooth functions can be addressed by prox-linear methods~\cite{drusvyatskiy2018error,drusvyatskiy2019efficiency}.  Various smoothing approaches~\cite{beck2017first,bohm2021variable,nesterov2005smooth,duchi2012randomized,chen2012smoothing} have also been widely applied to mitigate nonsmoothness. Nesterov~\cite{nesterov2005smooth} introduced smoothing techniques for structured nonsmooth problems. Moreau Envelope smoothing~\cite{davis2019stochastic,davis2018stochasticsubgradientmethodconverges} and randomized smoothing~\cite{duchi2012randomized} have received substantial attention.  More recently, Li and Cui~\cite{srd} proposed a subgradient regularization method for nonsmooth marginal functions, achieving notable progress.

\end{enumerate}

\paragraph{Certificate guarantees:}
In smooth optimization, the gradient norm is a widely accepted optimality certificate~\cite{lan2023optimal}. It offers two essential features. First, it is easily verifiable: given any candidate $\bar{y}$, one can directly compute $\norm{\nabla f(\bar{y})}$ without requiring additional problem parameters. Second, for an $L$-smooth $\mu$-strongly convex function, the gradient norm tightly characterizes the optimality gap:
\begin{equation}\label{eq:proportional}
    \frac{1}{2L}\norm{\nabla f\brbra{\bar{y}}}^2\leq f\brbra{\bar{y}}-f^*\leq \frac{1}{2\mu}\norm{\nabla f\brbra{\bar{y}}}^2\ .
\end{equation}
For nonsmooth but Lipschitz functions, verifiability is preserved, but the strong guarantees in~\eqref{eq:proportional} no longer hold. By Rademacher's theorem, Lipschitz functions are differentiable almost everywhere, which motivates defining the subdifferential as $\partial f(x):=\conv\bcbra{g:g=\lim_{t\to \infty}f^{\prime}(x^t),x^t\to x}$. With this definition, a point $x$ is stationary if  $0\in \partial f(x)$.
More generally,  $x$ is a $(\delta,\vep)$ near-approximate stationary point if $\norm{x-x^{\prime}}\leq\vep$ for some  $\vep$-stationary point $x^{\prime}$.
While Tian and So~\cite{tian2025testing} introduce the first oracle-polynomial-time algorithm for detecting near-approximate stationary points in piecewise affine functions, for general Lipschitz and bounded-below functions, finding such points cannot be guaranteed in polynomial time~\cite{kornowski2021oracle}.
As a relaxation, recent work~\cite{pmlr-v119-zhang20p,kornowski2021oracle} studies $(\delta,\varepsilon)$-Goldstein stationarity, which considers the convex hull of subgradients in a $\delta$-neighborhood.  Given a finite sample $\mathcal{P}=\bcbra{x^i}_{i=1}^m \subseteq \mathcal{B}(x;\delta)$, one can test $(\delta,\vep)$-Goldstein stationarity by: $\min_{\lambda\in \Delta_+}\norm{\sum_{i=1}^{m}\lambda_i f^{\prime}(x^i)}$, where $\Delta_+$ is the probability simplex. Approximate Goldstein stationarity is thus computationally tractable, though it does not imply bounds on the optimality gap. On the other hand,  for structured problems such as weakly convex functions, the Moreau stationary certificate is widely used. The Moreau envelope~\cite{moreau1965proximite} is defined as: $f_{\lambda}(x):=\min_{y}\bcbra{f(y)+\lambda \norm{y-x}^2/2}$, which is smooth when $\lambda$ exceeds the weak convexity modulus, with $\nabla f_{\lambda}(x)=\lambda\brbra{x-x_{\lambda}}$ and $x_{\lambda}=\argmin_{y}\bcbra{f(y)+\lambda \norm{y-x}^2/2}$. The main challenge is that computing $x_\lambda$ exactly is as hard as solving the original problem~\cite{davis2019stochastic}. Approximate Moreau-stationarity certificates~\cite{bohm2021variable} are possible, but they require an approximation $\hat{x}_{\lambda}$ with $\norm{\hat{x}_{\lambda}-{x}_{\lambda}}$ sufficiently small, which in turn demands an additional termination certificate for the proximal subproblem.
{Additionally, Nemirovski et al.~\cite{nemirovski2010accuracy} developed an accuracy certificate framework for convex optimization, variational inequalities with monotone operators, and convex Nash equilibrium problems, showing how such certificates can be constructed via ellipsoid or cutting-plane methods. Nevertheless, this framework has two crucial limitations: it only applies to bounded domains, and it does not accommodate error bounds such as the QG property. As a result, it produces only sublinear convergence rates for certificate generation.}

\subsection{Outline}

The remainder of this paper is organized as follows. We close this section by introducing the notation used throughout the paper. Section~\ref{sec:Preliminaries} provides key definitions and background on the bundle-level method. In Section~\ref{sec:Accelerated-Prox-level-method},
we propose $\apex$, which drives the objective down to a prescribed level
with a convergence rate that adapts to the {\pws}
structure. To remove the need to know the target level in advance,
Section~\ref{sec:Normalized-Wolfe-Certificate} develops an accelerated
Wolfe-certificate search procedure. Building on this certificate,
we then study {\pws} functions satisfying a $\mu^*$-QG
condition: first assuming $\mu^*$ is known (Section~\ref{sec:Restarted_mu_known}),
and then treating the case where $\mu^*$ is unknown (Section~\ref{sec:Restarted_mu_unknown}).
Section~\ref{sec:Numerical-Study} presents
numerical experiments that validate the practical performance of our
algorithms.

\subsection{Notation}

Throughout the paper, we use the following notation. Let $[n]:=\{1,\ldots,n\}$
for integer $n$. For integer $t$ and $m<t$, let $[t-m, t]:=\bcbra{t-m, t-m+1, \ldots, t}$. The $l_{q}$-norm of a vector $v\in\mbb R^{n}$
is defined as $\norm v_{q}=\brbra{\sum_{i=1}^{n}\abs{v_{(i)}}^{q}}^{1/q},$
where $v_{(i)}$ is the $i\text{-th}$ entry of vector $v$. For brevity,
$\norm{\cdot}$ denotes the $l_{2}$-norm. The inner product
of two vectors is defined as $\inner uv=\sum_{i=1}^{n}u_{(i)}v_{(i)}$.
We denote the distance from $x$ to a set $X$ as $\dist(x,X)=\min_{\bar{x}\in X}\norm{x-\bar{x}}.$
For brevity, we use $f^{\prime}(x)$ to denote a subgradient of $f$ at $x$,
which coincides with the gradient when $f\brbra x$ is differentiable.
The linearization of a function $f$ at point $\bar{x}$ is given
by $\ell_{f}\brbra{x;\bar{x}}:=f\brbra{\bar{x}}+\inner{f^{\prime}\brbra{\bar{x}}}{x-\bar{x}}.$
The indicator function is defined as $\onebf_{\bcbra A}={\renewcommand{\arraystretch}{0.7}\footnotesize\begin{cases}
1, & \text{if }A\text{ is true}\\
0, & \text{otherwise}
\end{cases}}$. The Euclidean ball centered at $\bar{x}$ with radius $\iota$ is
denoted as $\mcal B\brbra{\bar{x},\iota}=\bcbra{x:\norm{x-\bar{x}}\leq\iota}$.

\section{Preliminaries\label{sec:Preliminaries}}

In this section, we first provide several foundational definitions
and then introduce the basic BL method for convex nonsmooth
optimization. This background will help in understanding the algorithm
we propose later.

We begin with the key definitions of $\brbra{k,L}$-piecewise smoothness
and $\mu^*$-quadratic growth.
\begin{defn}
[$(k,L)$-Piecewise smoothness, $\brbra{k,L}$-\pws]\label{def:piecewise_smooth}We
say a function $f:X\to\mbb R$ is $(k,L)$-piecewise smooth if there
exists a covering of its domain $X$ by $k$ pieces $\bcbra{X_{i}}_{i=1}^{k}$
such that $X\subseteq \cup_{i=1}^{k} X_{i}$ and for each piece $X_{i}$, the restriction $f_{i}:=f_{\mid X_{i}}$
is $L$-smooth for some $L>0$. Specifically, we assume access to
a first-order oracle $f^{\prime}(x)$ such that the following inequality
holds for all $i\in\bsbra{k}$:
\begin{equation}\label{eq:pws_L_choice}
f\brbra x-\brbra{f(\bar{x}) + \inner{f^{\prime}(\bar{x})}{x-\bar{x}}}\leq L\norm{x-\bar{x}}^{2}/2\ ,\ \forall x,\bar{x}\in X_{i}\ .
\end{equation}
\end{defn}
Two remarks are given in order. First, the oracle $f^\prime$ in~\eqref{eq:pws_L_choice} is slightly stronger than a
standard convex subgradient oracle. At points in the interior of a smooth
piece, $f^\prime(\bar{x})$ coincides with the gradient. At boundary points,
however, the selected subgradient must satisfy~\eqref{eq:pws_L_choice} for all
$x$ in every piece $X_i$ containing $\bar{x}$. This boundary requirement is
mainly analytical: the practical algorithms below do not rely on repeatedly
querying such special boundary subgradients, and this assumption can be relaxed,
for example, by ensuring that iterates lie at differentiable points almost
surely. For simplicity, we assume throughout that the oracle
$f^\prime(\bar{x})$ satisfies~\eqref{eq:pws_L_choice}.
Second, this definition covers the familiar max-of-$k$-smooth-functions setting, but
also allows more general piecewise-smooth geometries. Figure~\ref{fig:maxquad_linear_example} shows that smooth
pieces can be disconnected. Figure~\ref{fig:trajector_polyak_bl} provides a simple illustration of the PWS structure in the two-dimensional case. Notably, optimizing such PWS objectives with the Polyak-step may lead to
zigzagging behavior~\cite{zhang2025linearly}.

\begin{figure}
    \raggedright{}%
    \begin{minipage}[t]{0.45\columnwidth}%
    \begin{center}
    \includegraphics[width=\textwidth]{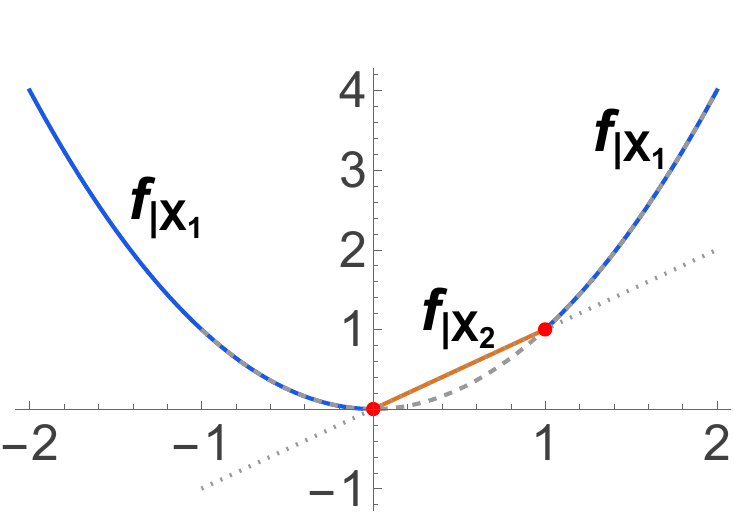}
    \caption{$\max\bcbra{{x}^2, x}$ is (2, 2)-$\pws$, where the region $X_1$ is disconnected.\label{fig:maxquad_linear_example}}
    \par\end{center}
    \end{minipage}
    \hfill
    \begin{minipage}[t]{0.50\columnwidth}%
    \begin{center}
    \includegraphics[width=0.85\textwidth]{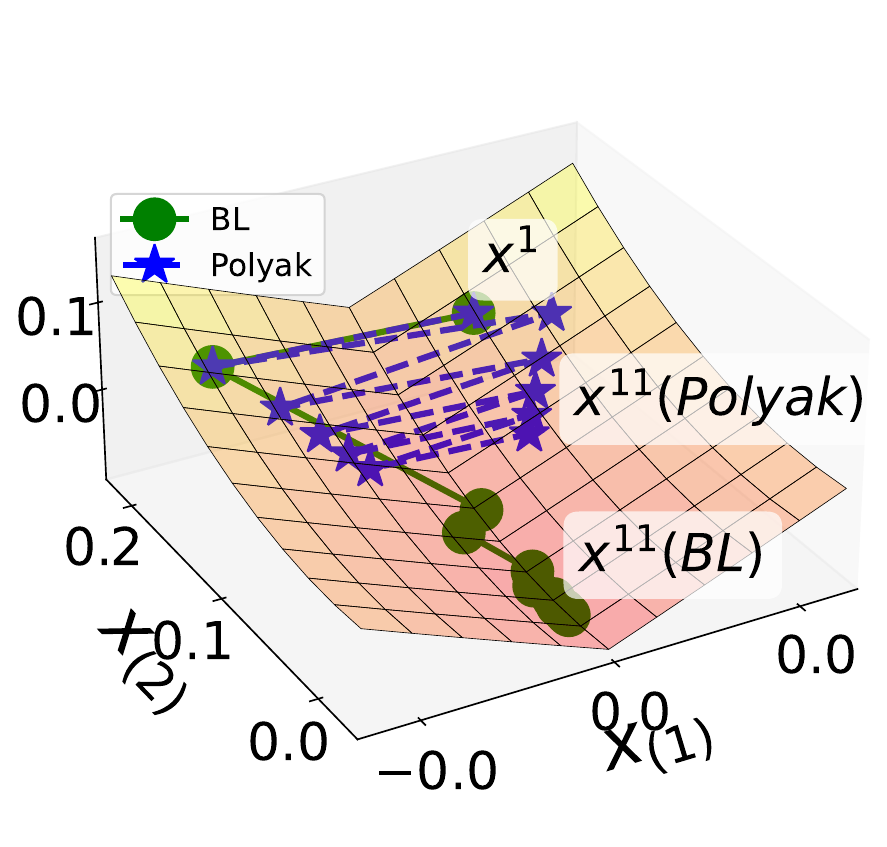}
    \caption{Trajectory comparison of Polyak-step and BL method on minimizing $\norm{x}^2+\abs{x_{(1)}}$ starting from $x^1=(0.01, 0.15)$.\label{fig:trajector_polyak_bl}}
    \par\end{center}%
    \end{minipage}
\end{figure}
\begin{defn}[Quadratic growth, QG]\label{def:QG}
 We say that $f:X\to\mbb R$
is $\mu$-quadratic growth if
$
f\brbra x-f^{*}\geq {\mu}\dist^{2}\brbra{x,X^{*}}/2,\forall x\in X
$, holds,
where $X^{*}$ is the nonempty minimizer set, i.e., $X^{*}=\argmin_{x\in X}f\brbra x$. Additionally, we define $\mu^*:=\sup\{\mu: f \text{ is } \mu\text{-QG}\}$ as the optimal QG modulus of $f$.
\end{defn}

{We briefly review the BL method, introduced in~\cite{lemarechal1995new} to improve the stability of the cutting-plane method.}
 At each iteration, the BL method
projects the current point onto a specific level set derived from the cutting
plane approximation. Specifically, the subproblem of BL is a diagonal quadratic
programming problem:
\[
x^{t}=\argmin_{x\in X}\norm{x-x^{t-1}}^{2}\ ,\ \st\ \ell_{f}\brbra{x;x^{i}}\leq\tilde{l}\ ,\ i \in [t-\bundleSize, t-1]\ ,
\]
where $\bundleSize$ is the prespecified number of cuts. If we set $\tilde{l}=f^*$ and $\bundleSize=1$, the BL method is equivalent to the Polyak-step~\cite{polyak1969minimization,polyak1987introduction}.
Figure~\ref{fig:trajector_polyak_bl} illustrates that Polyak steps may zigzag on
PWS objectives, while the BL method is more stable because it builds a local
model from multiple cutting planes. This distinction also appears in the
analysis: with $\bundleSize=1$, the Polyak-step admits a descent bound
involving only adjacent iterates:
\begin{equation}\label{eq:polyak_convergence}
\begin{aligned}
    f(x^{t+1})\leq &\ell_{f}(x^{t+1};x^{t})+{\hat{L}_t}\norm{x^{t+1}-x^{t}}^2/2\\
    \leq &f^*+{\hat{L}_t}(\norm{x^{t}-x^*}^2 - \norm{x^{t+1}-x^*}^2) /2
\end{aligned}\ ,
\end{equation}
where $\hat{L}_t$ represents the empirical Lipschitz smoothness constant estimated between $x^{t}$ and $x^{t+1}$. Hence, the convergence is strongly dependent on the empirical Lipschitz smoothness constant of adjacent iterates. When two adjacent iterates fall into different smooth pieces, $\hat{L}_t$ can be very large due to potential gradient discontinuities, leading to poorer theoretical and practical performance, as reflected in Polyak-step experiments (see Figure~\ref{fig:trajector_polyak_bl}).
In contrast, the BL method permits a refined analysis that leverages non-adjacent iterates. Specifically, its convergence can be established via inequalities of the form~\cite{zhang2025linearly}:
\begin{equation}\label{eq:BL_convergence}
\begin{aligned}
    f(x^{t+i})\leq& \ell_{f}(x^{t+i};x^{t}) + {\hat{L}_{t,i}} \norm{x^{t+i} - x^{t}}^2/2\\
    \leq& f^* + i \cdot {\hat{L}_{t,i}} \brbra{\norm{x^t-x^*}^2-\norm{x^{t+i}-x^*}^2}/2
\end{aligned}\ ,\forall i\in\bsbra{\bundleSize}\ ,
\end{equation}
where $\hat{L}_{t,i}$ is the empirical Lipschitz smoothness constant estimated between $x^{t}$ and $x^{t+i}$.
 Crucially, if the number of cuts exceeds the number of smooth pieces, the relevant Lipschitz dependence is determined by the worst constant among the individual pieces encountered, rather than the potentially much larger empirical Lipschitz smoothness constant across all iterates. This structural advantage explains the improved stability and convergence behavior observed for the BL method in $\pws$ optimization.



Building on the foundation of the BL method, Lan~\cite{lan2015bundle} introduced the Accelerated Bundle-Level (ABL) method, which enhances BL by incorporating Nesterov's acceleration technique.
The core
iteration includes three steps:
\begin{equation}
\begin{split}
\underline{x}^{t} & =\brbra{1-\alpha_{t}}\bar{x}^{t-1}+\alpha_{t}x^{t-1}\\
x^{t} & =\argmin_{x\in X}\norm{x-x^{t-1}}^{2}\ ,\ \st\ \ell_{f}\brbra{x;\underline{x}^{i}}\leq\tilde{l}\ ,\ i\in\bsbra{t}\ ,\\
\bar{x}^{t} & =\brbra{1-\alpha_{t}}\bar{x}^{t-1}+\alpha_{t}x^{t}
\end{split}\tag{ABL}
\end{equation}
A distinctive aspect of Nesterov-style acceleration is that the first-order model is evaluated at the weighted-average point $\underline{x}^t$, rather than merely at the most recent iterate. This strategy enables the algorithm to effectively incorporate momentum and leads to accelerated convergence rates.
ABL is uniformly optimal for smooth, weakly smooth, and nonsmooth functions. However,
one critical issue is that, as the algorithm proceeds, the number of cuts $\bundleSize=t$
increases, which makes the subproblem increasingly difficult to solve.
To address this issue, Lan~\cite{lan2015bundle} presents APL, which projects from
a fixed reference point $\bar{y}$ instead of the current point $x^{t-1}$ and introduces
an additional linear constraint (from~\cite{ben2001lectures,ben2005non,kiwiel1995proximal}) to avoid keeping all the cuts.
Specifically, it changes the second step of ABL to:
\[
\begin{split}
x^{t} =\argmin_{x\in X}\norm{x-\bar{y}}^{2}\ ,\ \st\  \inner{x-x^{t-1}}{x^{t-1}-\bar{y}}\geq0 \ ,\ \ell_{f}\brbra{x;\underline{x}^{i}}\leq\tilde{l}\ ,\ i\in\bsbra{t-\bundleSize,t-1}\ .
\end{split}
\]
This additional linear constraint is introduced to prevent the number of cutting planes from growing too large.
A key aspect of the convergence analysis for both ABL and APL lies in their reliance on a local Lipschitz-smoothness condition between adjacent points~\cite{lan2015bundle}, specifically,
\begin{equation}
\begin{aligned}
    f(\bar{x}^{t}) \leq & \ell_f(\bar{x}^t;\underline{x}^t) + {\hat{L}_{t}}\norm{\bar{x}^{t} - \underline{x}^t }^2/2\\
    \leq & \alpha_t f^* +(1-\alpha_t)f(\bar{x}^{t-1}) + {\hat{L}_{t}}\alpha_t^2\brbra{\norm{x^{t}-\bar{y}}^2-\norm{x^{t-1}-\bar{y}}^2}/2
\end{aligned}
\end{equation}
where $\hat{L}_{t}$ is the empirical Lipschitz smoothness constant determined between $\bar{x}^{t}$ and $\underline{x}^t$.
A key structural relationship, $\bar{x}^{t} - \underline{x}^t = \alpha_t (x^t - x^{t-1})$, forms the backbone of the telescoping arguments in the convergence proofs of ABL and APL. However, this identity is valid only for adjacent pairs of iterates and fails to generalize to non-adjacent sequences, such as those arising in the BL method discussed in~\eqref{eq:BL_convergence}. This limitation becomes especially pronounced when analyzing PWS functions, where telescoping across non-adjacent points is required.
\begin{algorithm}
\caption{$\protect\onestep\protect\brbra{\bar{y},\hat{x}^{t},x^{t,0},\tilde{l},\bundleSize}$}
\label{alg:One-step}
\begin{algorithmic}[1]
\State \textbf{Initialize:} $\flag=\false,x^{t+1,0}=x^{t,0},\hat{x}^{t+1} = \hat{x}^{t}$, $\bar{X}(t,0)=\bcbra{x\in X:\inner{x-x^{t,0}}{x^{t,0}-\bar{y}}\geq0}$
\For{$i=1,\ldots,\bundleSize$}
    \State $\underline{x}^{t,i}=\brbra{1-\alpha_{t}}\hat{x}^{t}+\alpha_{t}x^{t,i-1}$
    \State Set $\underline{X}(t,i)=\bcbra{x\in X:\ell_{f}\brbra{x;\underline{x}^{t,i}}\leq\tilde{l}}$
    \State Set $X(t,i):=\bcbra{\cap_{j\in\bsbra{1,i}}\underline{X}(t,j)}\cap\bar{X}\brbra{t,0}$
    \State Compute
    \begin{equation}
    x^{t,i}\leftarrow\argmin_{x\in X(t,i)}\norm{x-\bar{y}}^{2}\label{eq:subproblems}
    \end{equation}
    \If{the subproblem is infeasible}
        \State \label{enu:infeasible} \textbf{return} $\brbra{\hat{x}^{t+1},x^{t+1,0},\true,\infty,0}$
    \EndIf
    \State Compute $\bar{x}^{t,i}=\brbra{1-\alpha_{t}}\hat{x}^{t}+\alpha_{t}x^{t,i}$
    \State Choose $\hat{x}^{t+1}$ such that $f\brbra{\hat{x}^{t+1}}=\min\bcbra{f\brbra{\hat{x}^{t}},\min_{j\leq i}\bcbra{f\brbra{\bar{x}^{t,j}}}}$ and $x^{t+1,0}=x^{t,i}$
\EndFor
\State \label{enu:take_mini} Compute:
\begin{equation*}
\brbra{r_{t},l_{t}}=\argmin_{0\leq l<r\leq \bundleSize}L_{t}\brbra{r,l},
\end{equation*}
where $L_{t}\brbra{r,l}=\brbra{f(\hat{x}^{t+1}) - \tilde{l} - (1-\frac{3}{4}\alpha_t)(f(\hat{x}^{t}) - \tilde{l})}\big/\brbra{{\alpha_t^2}\|x^{t,r} - x^{t,l}\|^2/2}$ \Comment{Just for analysis}
\State \textbf{return} $\brbra{\hat{x}^{t+1},x^{t+1,0},\flag,L_{t}\brbra{r_{t},l_{t}},\norm{x^{t,r_{t}}-x^{t,l_{t}}}}$\label{enu:return}
\end{algorithmic}
\end{algorithm}
\section{\protect\label{sec:Accelerated-Prox-level-method}Accelerated Prox-level
method for piecewise smooth optimization}

We propose $\apex$ (\uline{A}ccelerated \uline{P}rox-level method for
\uline{Ex}ploring Piecewise Smooth), an accelerated method for convex {\pws}
optimization. Building on the APL framework~\cite{lan2015bundle}, $\apex$
uses a double-loop design: each outer iteration calls the subroutine \onestep
(Algorithm~\ref{alg:One-step}), which performs $\bundleSize$ inner projections.
Similar to APL, Algorithm~\ref{alg:One-step} computes gradients at the averaged points $\underline{x}^{t,i}$ and projects from a fixed reference point $\bar{y}$ onto the level set $X(t,i)$. Additionally, we ensure that the selected solution $\hat{x}^{t + 1}$ forms a non-increasing sequence of objective values across inner iterations.
A central distinction from APL is that $\apex$ updates the weighting factor $\alpha_t$ only in the outer loop, keeping the averaged point $\hat{x}^t$ fixed throughout each outer iteration. This approach enables piecewise-adaptive progress toward the target level while retaining the benefits of acceleration. Furthermore, $\apex$ incorporates quantities used for certificate generation and empirical Lipschitz estimate, features that are particularly valuable for developing parameter-free algorithms, as discussed in subsequent sections.

\begin{algorithm}
\caption{\uline{A}ccelerated \uline{P}rox-level Method for \uline{Ex}ploring Piecewise Smoothness ($\protect\apex$)}
\label{alg:apex}
\begin{algorithmic}[1]
\Require candidate solution $\bar{y}$, test level $\tilde{l}$, and the number of cuts $\bundleSize$
\State \textbf{Initialize:} $x^{1}=\bar{y},\hat{x}^{1}=\bar{y}$
\For{$t=1,\ldots,N$}
    \State $\brbra{\hat{x}^{t+1},x^{t+1},\ldots}\leftarrow\onestep\brbra{\bar{y},\hat{x}^{t},x^{t},\tilde{l},\bundleSize}$
\EndFor
\end{algorithmic}
\end{algorithm}

In contrast to existing analyses of APL~\cite{lan2015bundle}, which
depend on the Lipschitz continuity of the objective function across
successive iterates,
our proposed mechanism allows the derivation of optimality conditions
across non-adjacent iterates. This flexibility enables the
convergence proof to bypass certain intermediate points, leading to
tighter theoretical guarantees with improved dependence on the Lipschitz
constant.
We first define two Lipschitz constants:
\begin{equation}
    L_t = \min_{0\leq l<r\leq \bundleSize}L_t(r,l)\ \ \text{and}\ \ \bar{L}\brbra N=\frac{\sum_{t=1}^{N}\omega_{t}\alpha_{t}^{2}L_{t}\onebf_{\mcal E(t)}\norm{x^{t,r_{t}}-x^{t,l_{t}}}^{2}/2}{\sum_{t=1}^{N}\norm{x^{t,r_{t}}-x^{t,l_{t}}}^{2}}\ .\label{eq:L_N}
\end{equation}
Here, $L_t(r,l)$ is defined in Line~\ref{enu:take_mini} of Algorithm~\ref{alg:One-step} and event $\mcal E(t)$ is defined as follows:
\begin{equation}\label{eq:event_E_t}
\mcal E(t)=\bcbra{f\brbra{\hat{x}^{t+1}}-\tilde{l}>\brbra{1-\frac{\alpha_t}{2}}\brbra{f(\hat{x}^{t})-\tilde{l}}}.
\end{equation}
Furthermore, we adopt the standard convention that $0/0 = 0$.  {The event \(\mathcal E(t)\) marks the outer iterations in which the residual
\(f(\hat{x}^t)-\tilde{l}\) fails to contract at the benchmark rate
\(1-\alpha_t/2\). These are precisely the slow-descent iterations for which the proof must invoke an additional curvature correction, measured by \(L_t\).
When \(\mathcal E(t)\) does not occur, the benchmark contraction already holds, so the indicator $\onebf_{\mathcal E(t)}$ switches off that correction.
Consequently,
$\bar L(N)$ can be viewed as a weighted average of the effective curvature terms $(\omega_t\alpha_t^2/2)L_t$ over the outer iterations for which $\mcal E(t)$ occurs, where the averaging weights are proportional to $\|x^{t,r_t}-x^{t,l_t}\|^2$; if no such iteration occurs, we set $\bar L(N)=0$.}

We now state the following proposition, which formalizes the
convergence properties of Algorithm~\ref{alg:apex}.
\begin{proposition}
\label{prop:converge}
Suppose the sequence $\bcbra{\hat{x}^{t}}$ generated by Algorithm~\ref{alg:apex}
satisfies $f\brbra{\hat{x}^{t}}\geq\tilde{l}$ for all $t\geq1$ and
that every subproblem~\eqref{eq:subproblems} of Algorithm~\ref{alg:apex}
is feasible. Assume the weights $\bcbra{\omega_{t}}$
satisfy $\omega_{t}\brbra{1-\alpha_{t}/2}=\omega_{t-1}$ for $t\geq2$.
Then the following properties hold:
\begin{enumerate}
    \item \textbf{Monotonicity of proximal distances:} For every $t \geq 1$ and $0 \leq i < \bundleSize$, it holds that $\|x^{t,i} - \bar{y}\| \leq \|x^{t,i+1} - \bar{y}\|$.
    \item \textbf{Convergence bound:} The quantity $\bar{L}(t)$ (defined in~\eqref{eq:L_N}) satisfies the following inequality: 
    \begin{equation}
    \omega_t\left(f(\hat{x}^{t + 1}) - \tilde{l}\right) - \omega_1\left(1 - \frac{\alpha_1}{2}\right)\left(f(\hat{x}^{1}) - \tilde{l}\right) \leq \bar{L}(t)\|x^{t + 1} - \bar{y}\|^2\ ,\ \forall t\geq 2\ .
    \label{eq:upper_bound_con}
    \end{equation}
\end{enumerate}
\end{proposition}
\noindent\emph{Proof deferred.} The detailed proof is collected in Subsection~\ref{subsec:apex-detailed-proofs}.

Substituting $\tilde{l}=f^{*}$ into the bound of Proposition~\ref{prop:converge},
we obtain the following sublinear convergence bound.
\begin{thm}
\label{thm:f_star}
Set the weight sequence $\omega_{t}=\brbra{t+2}\brbra{t+3}/2$ and
$\alpha_{t}=4/\brbra{t+3}$ in $\apex$. Then for the sequence $\bcbra{\hat{x}^{t}}$
generated by $\apex$ with level $\tilde{l}=f^{*}$, the quantity $\bar{L}\brbra t$  (defined in~\eqref{eq:L_N}) satisfies
\[
f\brbra{\hat{x}^{t+1}}-f^{*}\leq\frac{6}{\brbra{t+2}\brbra{t+3}}\brbra{f\brbra{\hat{x}^{1}}-f^{*}}+\frac{2\bar{L}\brbra t}{\brbra{t+2}\brbra{t+3}}\dist^{2}\brbra{\bar{y}, X^*}\ .
\]
\end{thm}
\noindent\emph{Proof deferred.} The detailed proof is collected in Subsection~\ref{subsec:apex-detailed-proofs}.

We make several remarks regarding the result in Proposition~\ref{prop:converge} and Theorem~\ref{thm:f_star}.

First, Algorithm~\ref{alg:apex} is a double-loop method. By updating the weighting parameter
$\alpha_{t}$ and the reference iterate $\hat{x}^{t}$ only once per
outer iteration, the convergence analysis can be constructed on arbitrary pairs
of inner-loop iterates, rather than only successive ones. 
Specifically, Proposition~\ref{prop:converge} and Theorem~\ref{thm:f_star} express the complexity through $\bar{L}(t)$, defined in~\eqref{eq:L_N}, a weighted average of effective curvature terms built from the minimum over inner-iterate pairs in each outer iteration.
This flexibility enables the algorithm to exploit the $\pws$ structure.

Second, within each inner iteration, a cut-generation scheme is defined
as
\begin{equation}
\text{Option I}:\ \ X(t,i):=\bcbra{\cap_{j\in\bsbra{1,i}}\underline{X}(t,j)}\cap\bar{X}\brbra{t,0}\ .
\end{equation}
In this scheme, the algorithm starts with two initial cuts; at each
of the $\bundleSize$ inner steps, it appends one new cutting plane derived
from the updated point $\underline{x}^{t,i}$. A major limitation
of this strategy is that all cuts from the previous iteration must
be discarded when entering a new outer iteration. Alternatively, one
may consider
\begin{equation}
\text{Option II}:\ \ X\brbra{t,i}:=\bcbra{\cap_{j\in\bsbra{1,i}}\underline{X}(t,j)}\cap\bar{X}\brbra{t-1,0}\cap\bcbra{\cap_{j\in\bsbra{1,\bundleSize}}\underline{X}(t-1,j)}
\end{equation}
Here, we replace $\bar{X}\brbra{t,0}$ with $\bar{X}\brbra{t-1,0}\cap\bcbra{\cap_{j\in\bsbra{1,\bundleSize}}\underline{X}(t-1,j)}$.
This modification preserves all properties established in Proposition~\ref{prop:converge}
and does not affect the convergence result. The advantage is that
cuts from the most recent outer iteration are retained. However, if the
number of cuts is fixed, the number of cuts in constraints $\bcbra{\cap_{j\in\bsbra{1,i}}\underline{X}(t,j)}$
will be smaller than that in the Option I scheme. Consequently, the
Option II approach may reduce the algorithm's ability
to capture the $\pws$ structure.

Third, Proposition~\ref{prop:converge} contains two important properties.
The first is that the sequence $\bcbra{\norm{x^{t,i}-\bar{y}}^{2}}$ is non-decreasing,
which helps in constructing a lower bound on $f^{*}$ in the following
sections. The second property implies that if the domain $X$ is bounded,
i.e., there exists a constant $D_{X}>0$ such that $D_{X}=\max_{x_{1},x_{2}\in X}\norm{x_{1}-x_{2}}$,
then we have
\begin{equation}
\omega_{N}\Brbra{f\brbra{\hat{x}^{N+1}}-\tilde{l}}-\omega_{1}\brbra{1-\frac{\alpha_{1}}{2}}\brbra{f\brbra{\hat{x}^{1}}-\tilde{l}}\leq\bar{L}\brbra ND_{X}^{2}\ ,\label{eq:Dx_converge}
\end{equation}
which implies $f\brbra{\hat{x}^{N}}-\tilde{l}=\mcal O\brbra{\frac{\bar{L}\brbra ND_{X}^{2}}{\omega_{N}}}$
and this is also the core idea behind Lan's proof in~\cite{lan2015bundle}.

Fourth, in many applications one may choose the target level $\tilde{l}$
to be the true optimum $f^{*}$; this is exactly the strategy underlying
Polyak  step size~\cite{polyak1969minimization,polyak1987introduction}
and Polyak Minorant Method~\cite{devanathan2023polyak}. Furthermore, Dang et al.~\cite{dang2017linearly} have applied this idea to semidefinite programming.

A remaining question is that, even though $\bar{L}\brbra t$ can be computed in
practice, it is still unclear how to bound the quantity  $\bar{L}\brbra t$
as it appears in Proposition~\ref{prop:converge} and Theorem~\ref{thm:f_star}.
In the following Proposition~\ref{prop:upper_bound_LbarN}, we show that
value $\bar{L}\brbra t$ can be upper bounded by a weighted average
of empirical Lipschitz smoothness constants $\tilde{L}_{t}$, where $\tilde{L}_{t}$
is the smallest local Lipschitz constant between $\bar{x}^{t,r}$ and $\underline{x}^{t,l+1}$
with $r>l$ in $t$-th outer iteration.
\begin{proposition}
\label{prop:upper_bound_LbarN}Let $\bcbra{\hat{x}^{t},x^{t,i}}$ be
the iterates generated by Algorithm~\ref{alg:apex}, and 
for each $\onestep$ with $t\geq 1$,  define $\tilde{L}_t$ by 
\begin{equation}\label{eq:defn_tilde_L_t}
\tilde{L}_{t}:=\min_{0\leq l_{t}<r_{t}\leq B}\frac{2\brbra{\max\bcbra{f\brbra{\bar{x}^{t,r_{t}}}-\ell_{f}\brbra{\bar{x}^{t,r_{t}},\underline{x}^{t,l_{t}+1}}-\frac{\alpha_{t}}{4}\brbra{f\brbra{\hat{x}^{t}}-\tilde{l}},0}}}{\norm{\bar{x}^{t,r_{t}}-\underline{x}^{t,l_{t}+1}}^{2}}\ ,
\end{equation}
where all quotients in this proposition use the convention \(0/0=0\).
Then $L_{t}\leq\tilde{L}_{t}$ ($L_{t}$ is defined in~\eqref{eq:L_N}).
Furthermore, set the weighted sequence $\omega_{t}={(t+2)(t+3)/2}$ and $\alpha_{t}=4/\brbra{t+3}$, and suppose $f$ is $(k,L)$-$\pws$ and the selected bundle size $B\geq k$, then 
 $\bar{L}\brbra N\leq O(1)L$. 
\end{proposition}
\noindent\emph{Proof deferred.} The detailed proof is collected in Subsection~\ref{subsec:apex-detailed-proofs}.

We make a few remarks regarding this result. 
First, if the objective $f$ is globally $L$-smooth, then each local
surrogate $\tilde{L}_{t}$ is uniformly bounded above by $L$, and hence $\bar{L}\brbra N \leq O(1)L$.
The same conclusion holds more generally when $f$ is $\brbra{k,L}$-$\pws$ with at most $k$
smooth pieces, provided that $\bundleSize\geq k$ cuts are maintained in each outer iteration.
If instead $\bundleSize<k$, then under $M$-Lipschitz continuity of $f$ we use the following upper bound:
\[
\begin{split}\tilde{L}_{t}= & 2\brbra{f\brbra{\bar{x}^{t,r_{t}}}-\ell_{f}\brbra{\bar{x}^{t,r_{t}},\underline{x}^{t,l_{t}+1}}-\frac{\alpha_{t}}{4}\brbra{f\brbra{\hat{x}^{t}}-\tilde{l}}}\Big/\norm{\bar{x}^{t,r_{t}}-\underline{x}^{t,l_{t}+1}}^{2}\\
\leq & 2\brbra{M\norm{\bar{x}^{t,r_{t}}-\underline{x}^{t,l_{t}+1}}-\frac{\alpha_{t}}{4}\brbra{f\brbra{\hat{x}^{t}}-\tilde{l}}}\Big/\norm{\bar{x}^{t,r_{t}}-\underline{x}^{t,l_{t}+1}}^{2}\\
\aleq & 2\brbra{\frac{M^{2}\norm{\bar{x}^{t,r_{t}}-\underline{x}^{t,l_{t}+1}}^{2}}{\alpha_{t}\brbra{f\brbra{\hat{x}^{t}}-\tilde{l}}}}\Big/\norm{\bar{x}^{t,r_{t}}-\underline{x}^{t,l_{t}+1}}^{2}=\frac{2M^{2}}{\alpha_{t}\brbra{f\brbra{\hat{x}^{t}}-\tilde{l}}}
\end{split}
\ ,
\]
where $(a)$ holds by $\tau s\leq\frac{1}{2}\tau^{2}+\frac{1}{2}s^{2}$
by taking $s^2=\alpha_{t}\brbra{f\brbra{\hat{x}^{t}}-\tilde{l}}/2$
and $\tau=M\norm{\bar{x}^{t,r_{t}}-\underline{x}^{t,l_{t}+1}}\big/s$.
Consequently, although Algorithm~\ref{alg:apex} is tailored to {\pws}
problems, its analysis readily extends to general Lipschitz-continuous
objectives by substituting the bound $\tilde{L}_{t}\leq2M^{2}\alpha_{t}^{-1}\big/\brbra{f\brbra{\hat{x}^{t}}-\tilde{l}}$
into the harmonic-mean estimate. 
Second, $\bar{L}\brbra N$ is a weighted harmonic mean of local smoothness surrogates along the generated iterates. Thus, controlling $\bar{L}\brbra N$ does not require $\bundleSize$ to exceed the total number of pieces in the PWS objective; a sufficient condition for $\bar{L}\brbra N\leq O(1)L$ is that $\bundleSize$ covers the distinct pieces actually encountered along the iterates. This is also reflected in the experiments in Section~\ref{sec:Numerical-Study}, where we show that $\apex$ can achieve fast convergence even with $\bundleSize < k$.

To clarify the main idea behind our practical restarting procedure and the Lipschitz constant dependence, we first present the idealized restarting scheme for $\apex$ in Algorithm~\ref{alg:ideal_apex}, which isolates the core restarting mechanism used later in the practical algorithm.
\begin{algorithm}
\caption{Idealized $\apex$}
\label{alg:ideal_apex}
\begin{algorithmic}[1]
\Require initial candidate solution $\bar{y}^{1}$, optimal value $f^*$, number of cuts $\bundleSize$, shrink factor $\theta\in(0,1)$
\For{$s=1,\ldots,S$}
    \State Run $(\hat{x}^{t_s + 1}_{s},\ldots)\leftarrow \apex(f,\bar{y}^{s},f^*,\bundleSize)$ until
    $
    f(\hat{x}^{t_s + 1}_{s})-f^*
    \le \theta\brbra{f(\bar{y}^{s})-f^*}
    $
    \State Set $\bar{y}^{s + 1}=\hat{x}^{t_s + 1}_{s}$
\EndFor
\end{algorithmic}
\end{algorithm}

\begin{thm}\label{thm:complexity_idealized_apex}
Suppose that the assumptions of Theorem~\ref{thm:f_star} hold and that $f$ satisfies the $\mu^*$-QG condition.
For the $s$-th stage of Algorithm~\ref{alg:ideal_apex}, let $\bar{L}_{s}(t)$ be the stagewise counterpart of $\bar{L}(t)$ in Theorem~\ref{thm:f_star}. Let $\tilde{L}_{s,t}$ be the corresponding quantity $\tilde{L}_{t}$ in Proposition~\ref{prop:upper_bound_LbarN}, evaluated at the $t$-th iteration of $\apex$ with $\tilde{l}=f^*$, and let $\mcal E_s(t)$ be the corresponding event $\mcal E(t)$.
For any target accuracy $\vep\in\brbra{0,f\brbra{\bar y^{1}}-f^*}$, if Algorithm~\ref{alg:ideal_apex} stops after the first stage $S$ for which
$f(\bar y^{S+1})-f^*\leq\vep$, then the total number of gradient evaluations is at most
\[
O(1)\bundleSize\max\bcbra{\sqrt{\frac{{L}_{\text{avg}}}{\mu^*}}, 1}\,
\log\brbra{\frac{f\brbra{\bar y^{1}}-f^*}{\vep}},
\]
where ${L}_{\text{avg}}=\max_{s=1,\ldots,S}\bar{L}_{s}(t_s - 1)$ and we use the convention $\bar{L}_s(0) = 0$.
Moreover, for each $s=1,\ldots,S$ such that $t_s>1$,
\begin{equation}\label{eq:upper_L_s_t_s}
\bar{L}_{s}\brbra{t_s -1 }
\leq \frac{12}{\theta}\cdot
\frac{\sum_{t=1}^{t_{s}-1}\onebf_{\mcal E_{s}(t)}(t+3)}
{\sum_{t=1}^{t_{s}-1}\brbra{\onebf_{\mcal E_{s}(t)}(t+3)/\tilde{L}_{s,t}}}\ .
\end{equation}
Here, we adopt the standard convention that $0/0=0$.
\end{thm}

\noindent\emph{Proof deferred.} The detailed proof is collected in Subsection~\ref{subsec:apex-detailed-proofs}.

Two remarks are given with respect to result of Theorem~\ref{thm:complexity_idealized_apex}.  First,~\eqref{eq:upper_L_s_t_s} bounds each $\bar L_s(t_s)$ by a weighted harmonic average of the local surrogates $\tilde L_{s,t}$ over the indices where $\mcal E_s(t)$ occurs; because the weights are $t+3$, later such indices carry more weight. Second, the weighted harmonic average is insensitive to the presence of a few large values. Hence, the bound in~\eqref{eq:upper_L_s_t_s} is robust to occasional large local surrogates $\tilde{L}_{s,t}$.

The next proposition lower bounds, in each stage, the fraction of indices for which $\mcal E_s(t)$ occurs; consequently, the weighted harmonic average incorporates a major subset of the local Lipschitz constants.
\begin{proposition}\label{prop:lower_bound_event_time_ratio}
For any stage $s$ of Algorithm~\ref{alg:ideal_apex} with $t_{s}>1$, we have
\begin{equation}
\frac{\sum_{t=1}^{t_{s}}\onebf_{\mcal E_{s}\brbra t}}{t_{s}}\geq\frac{1}{t_{s}}\max\left\{\left\lfloor \frac{-1+\sqrt{1+4\theta\brbra{t_{s}+2}\brbra{t_{s}+3}}}{2}\right\rfloor-1,\ 0\right\}\ .
\end{equation}
\end{proposition}
\noindent\emph{Proof deferred.} The detailed proof is collected in Subsection~\ref{subsec:apex-detailed-proofs}.

Proposition~\ref{prop:lower_bound_event_time_ratio} shows that $\mcal E_s(t)$ occurs on a nontrivial portion of a long stage: the lower bound approaches $\sqrt{\theta}$ as $t_s\to\infty$. Thus, when $\theta=1/2$, roughly $70.7\%$ of a long stage contributes local surrogates to the weighted harmonic-mean bound in~\eqref{eq:upper_L_s_t_s}.


\subsection{Detailed proofs\label{subsec:apex-detailed-proofs}}

The statements above isolate the algorithmic consequences of the APEX construction. We now collect the full proofs in the same order as the corresponding results.

\begin{proof}[\underline{Proof of Proposition~\ref{prop:converge}}]
Part 1 (Monotonicity). We first establish that the distances $\norm{x^{t,i}-\bar{y}}$
are nondecreasing across the inner-loop updates. For each $1\leq i\leq \bundleSize$,
recall that $x^{t,i}\in\argmin_{x\in X(t,i)}\norm{x-\bar{y}}^{2}$.
By the three-point property for Euclidean projections, this implies
\[
\norm{x^{t,i}-\bar{y}}^{2}+\norm{x^{t,i}-x}^{2}\leq\norm{x-\bar{y}}^{2}\ ,\ \forall x\in X(t,i)\ .
\]
Because $x^{t,i+1}\in X(t,i+1)\subseteq X(t,i)$ and $x^{t,i+1}\in X\brbra{t,i+1}$,
substituting $x=x^{t,i+1}$ gives
\[
\norm{x^{t,i}-\bar{y}}^{2}+\norm{x^{t,i}-x^{t,i+1}}^{2}\leq\norm{x^{t,i+1}-\bar{y}}^{2}\ ,
\]
which implies $\norm{x^{t,i}-\bar{y}}\leq\norm{x^{t,i+1}-\bar{y}}$.

For the initial case $i=0$, we expand
\[
\begin{split}\norm{x^{t,1}-\bar{y}}^{2} & =\norm{x^{t,1}-x^{t,0}}^{2}+2\inner{x^{t,1}-x^{t,0}}{x^{t,0}-\bar{y}}+\norm{x^{t,0}-\bar{y}}^{2}\end{split}
\ .
\]
The linear constraint in $X(t,0)$ ensures the cross term is nonnegative,
so $\norm{x^{t,0}-\bar{y}}^{2}\leq\norm{x^{t,1}-\bar{y}}^{2}$. Hence
$\norm{x^{t,i}-\bar{y}}\leq\norm{x^{t,i+1}-\bar{y}}$, which completes
the proof of monotonicity.

Part 2 (Convergence bound). We now turn to the function-value recursion.


For each outer iteration $t$, let $\brbra{l_{t}, r_{t}}$ be a minimizing pair in the definition of $\tilde{L}_{t}$, where $0\leq l_{t}<r_{t}\leq \bundleSize$.
By the projection optimality condition for $x^{t,l_{t}}$, we
have
\[
\norm{x^{t,l_{t}}-\bar{y}}^{2}+\norm{x^{t,l_{t}}-x}^{2}\leq\norm{x-\bar{y}}^{2}\ ,\ \forall x\in X\brbra{t,l_{t}}\ .
\]
By construction,  $X(t,r_{t})\subseteq X(t,l_{t})$. Hence, taking $x=x^{t,r_{t}}\in X(t,r_{t})\subseteq X(t,l_{t})$
in above gives:
\begin{equation}
\norm{x^{t,l_{t}}-\bar{y}}^{2}+\norm{x^{t,l_{t}}-x^{t,r_{t}}}^{2}\leq\norm{x^{t,r_{t}}-\bar{y}}^{2}\ .\label{eq:away_bridged_three_point}
\end{equation}

Next, we give the descent analysis of the sequence based on~\eqref{eq:away_bridged_three_point}.
It follows from the non-decreasing property in Part~1 and~\eqref{eq:away_bridged_three_point}
that
\[
\norm{x^{t,0}-\bar{y}}^{2}+\norm{x^{t,l_{t}}-x^{t,r_{t}}}^{2}\leq\norm{x^{t,\bundleSize}-\bar{y}}^{2}\ .
\]
Summing from $t=1$ to $N$ gives
\begin{equation}
\sum_{t=1}^{N}\norm{x^{t,l_{t}}-x^{t,r_{t}}}^{2}\aeq\norm{x^{1,0}-\bar{y}}^{2}+\sum_{t=1}^{N}\norm{x^{t,l_{t}}-x^{t,r_{t}}}^{2}\leq\norm{x^{N,\bundleSize}-\bar{y}}^{2}\ ,\label{eq:sum_bounded}
\end{equation}
where $(a)$ follows because $x^{1,0}=\bar{y}$. From the definition
of $L_{t}$, we obtain
\[
\begin{split}f\brbra{\hat{x}^{t + 1}}-\tilde{l} & \leq\brbra{1-\frac{3}{4}\alpha_{t}}\brbra{f\brbra{\hat{x}^{t}}-\tilde{l}}+\frac{L_{t}\alpha_{t}^{2}}{2}\norm{x^{t,r_{t}}-x^{t,l_{t}}}^{2}\\
 & \leq\brbra{1-\frac{\alpha_{t}}{2}}\brbra{f\brbra{\hat{x}^{t}}-\tilde{l}}+\frac{L_{t}\alpha_{t}^{2}}{2}\norm{x^{t,r_{t}}-x^{t,l_{t}}}^{2}
\end{split}
\ .
\]
Then combining the above inequality and the definition of $\onebf_{\mcal E(t)}$
gives
\[
f\brbra{\hat{x}^{t + 1}}-\tilde{l}\leq\brbra{1-\frac{\alpha_{t}}{2}}\brbra{f\brbra{\hat{x}^{t}}-\tilde{l}}+\frac{L_{t}\onebf_{\mcal E(t)}\alpha_{t}^{2}}{2}\norm{x^{t,r_{t}}-x^{t,l_{t}}}^{2}.
\]
Multiplying both sides by $\omega_{t}$ and using $\omega_{t}\brbra{1-\frac{\alpha_{t}}{2}}=\omega_{t-1}$,
we obtain
\begin{equation}
\begin{split}\omega_{t}\Brbra{f\brbra{\hat{x}^{t+1}}-\tilde{l}} & \leq\omega_{t}\brbra{1-\frac{\alpha_{t}}{2}}\brbra{f\brbra{\hat{x}^{t}}-\tilde{l}}+\frac{\omega_{t}L_{t}\onebf_{\mcal E(t)}\alpha_{t}^{2}}{2}\norm{x^{t,r_{t}}-x^{t,l_{t}}}^{2}\\
 & \aeq\omega_{t-1}\brbra{f\brbra{\hat{x}^{t}}-\tilde{l}}+\frac{\omega_{t}L_{t}\onebf_{\mcal E(t)}\alpha_{t}^{2}}{2}\norm{x^{t,r_{t}}-x^{t,l_{t}}}^{2}\ ,\\
\omega_{1}\Brbra{f\brbra{\hat{x}^{2}}-\tilde{l}} & \leq\omega_{1}\brbra{1-\frac{\alpha_{1}}{2}}\brbra{f\brbra{\hat{x}^{1}}-\tilde{l}}+\frac{\omega_{1}L_{1}\onebf_{\mcal E(1)}\alpha_{1}^{2}}{2}\norm{x^{1,r_{1}}-x^{1,l_{1}}}^{2}\ ,
\end{split}
\label{eq:recusion}
\end{equation}
where $(a)$ holds by the definition of $\omega_{t}$.
Summing~\eqref{eq:recusion} from $t=1$ to $N$ yields
\begin{align*}
\omega_{N}\Brbra{f\brbra{\hat{x}^{N+1}}-\tilde{l}}-\omega_{1}\brbra{1-\frac{\alpha_{1}}{2}}\brbra{f\brbra{\hat{x}^{1}}-\tilde{l}} & \leq\bar{L}\brbra N\sum_{t=1}^{N}\norm{x^{t,r_{t}}-x^{t,l_{t}}}^{2}\aleq\bar{L}\brbra N\norm{x^{N,\bundleSize}-\bar{y}}^{2}\ ,
\end{align*}
where $(a)$ is by~\eqref{eq:sum_bounded}. This completes the proof.

\ifdefined\isarxiv
    \end{proof}
\else
    \qedsymbol\end{proof}
\fi

\begin{proof}[\underline{Proof of Theorem~\ref{thm:f_star}}]
The choice of $\omega_{t}$ and $\alpha_{t}$ satisfy $\omega_{t}\brbra{1-\frac{\alpha_{t}}{2}}=\omega_{t-1}$.
Hence, it follows from the second result of Proposition~\ref{prop:converge}
that
\[
\omega_{t}\Brbra{f\brbra{\hat{x}^{t+1}}-f^{*}}-\omega_{1}\brbra{1-\frac{\alpha_{1}}{2}}\brbra{f\brbra{\hat{x}^{1}}-f^{*}}\leq\bar{L}\brbra t\norm{x^{t+1}-\bar{y}}^{2}\ .
\]
Note that any optimal solution $x^{*}\in X^*$ is  feasible for the
subproblem if we set the level $\tilde{l}=f^{*}$. Hence, any feasible
solution to the subproblem $x^{t+1}=x^{t,\bundleSize}$ should satisfy $\norm{x^{t+1}-\bar{y}}\leq\dist\brbra{\bar{y},X^*}$.
\ifdefined\isarxiv
    \end{proof}
\else
    \qedsymbol\end{proof}
\fi

\begin{proof}[\underline{Proof of Proposition~\ref{prop:upper_bound_LbarN}}]
It follows from $f\brbra{\hat{x}^{t+1}}\leq f\brbra{\bar{x}^{t,r_{t}}}$
and the definition of $\bar{x}^{t,r_{t}}=\brbra{1-\alpha_{t}}\hat{x}^{t}+\alpha_{t}x^{t,r_{t}}$
that
\[
\begin{split}f\brbra{\hat{x}^{t+1}} & \leq f\brbra{\bar{x}^{t,r_{t}}}\aleq\ell_{f}\brbra{\bar{x}^{t,r_{t}};\underline{x}^{t,l_{t}+1}}+\frac{\tilde{L}_{t}}{2}\norm{\bar{x}^{t,r_{t}}-\underline{x}^{t,l_{t}+1}}^{2}+\frac{\alpha_{t}}{4}\brbra{f\brbra{\hat{x}^{t}}-\tilde{l}}\\
 & \beq\brbra{1-\alpha_{t}}\ell_{f}\brbra{\hat{x}^{t};\underline{x}^{t,l_{t}+1}}+\alpha_{t}\ell_{f}\brbra{x^{t,r_{t}};\underline{x}^{t,l_{t}+1}}+\frac{\tilde{L}_{t}}{2}\norm{\bar{x}^{t,r_{t}}-\underline{x}^{t,l_{t}+1}}^{2}+\frac{\alpha_{t}}{4}\brbra{f\brbra{\hat{x}^{t}}-\tilde{l}}\\
 & \cleq\brbra{1-\alpha_{t}}f\brbra{\hat{x}^{t}}+\alpha_{t}\tilde{l}+\frac{\tilde{L}_{t}}{2}\norm{\bar{x}^{t,r_{t}}-\underline{x}^{t,l_{t}+1}}^{2}+\frac{\alpha_{t}}{4}\brbra{f\brbra{\hat{x}^{t}}-\tilde{l}}
\end{split}
\ ,
\]
where $(a)$ holds by the definition of $\tilde{L}_{t}$ in~\eqref{eq:defn_tilde_L_t}, $(b)$ is
by the definition of $\bar{x}^{t,r_{t}}$ and $\ell_{f}$, and $(c)$
follows from the convexity of $f$ and $x^{t,r_{t}}\in X\brbra{t,r_{t}}$.
Subtracting $\tilde{l}$ on both sides gives
\[
f\brbra{\hat{x}^{t+1}}-\tilde{l}\leq\brbra{1-\frac{3}{4}\alpha_{t}}\brbra{f\brbra{\hat{x}^{t}}-\tilde{l}}+\frac{\tilde{L}_{t}\alpha_{t}^{2}}{2}\norm{x^{t,r_{t}}-x^{t,l_{t}}}^{2}\ ,
\]
which implies that $L_{t}\leq\tilde{L}_{t}$. Next, we consider the
upper bound of $\bar{L}\brbra N$. By the definition of $\bar{L}\brbra N$
and $\onebf_{\mcal E(t)}$ , the upper bound of $\bar{L}\brbra N$
is given as follows:
\[
\begin{split}\bar{L}\brbra N & {=}\frac{\sum_{t=1}^{N}\omega_{t}\alpha_{t}^{2}L_{t}\onebf_{\mcal E(t)}\norm{x^{t,r_{t}}-x^{t,l_{t}}}^{2}/2}{\sum_{t=1}^{N}\norm{x^{t,r_{t}}-x^{t,l_{t}}}^{2}}\leq \frac{4\sum_{t=1}^{N}\onebf_{\mcal E(t)}L_{t}\norm{x^{t,r_{t}}-x^{t,l_{t}}}^{2}}{\sum_{t=1}^{N}\norm{x^{t,r_{t}}-x^{t,l_{t}}}^{2}}
\end{split}
\ ,
\]
where the inequality follows from the parameter choice
$\omega_{t}=\brbra{t+2}\brbra{t+3}/2$ and $\alpha_{t}=4/\brbra{t+3}$, for which
$\frac{\omega_{t}\alpha_{t}^{2}}{2}=\frac{4(t+2)}{t+3}\leq 4.$
Finally, if $B\geq k$ and $f$ is $(k,L)$-$\pws$, then the pigeonhole
principle gives indices $0\leq l_t<r_t\leq B$ such that
$\bar{x}^{t,r_t}$ and $\underline{x}^{t,l_t+1}$ lie in the same smooth piece.
Hence $L_t \leq \tilde L_t\leq O(1)L$.
\ifdefined\isarxiv
    \end{proof}
\else
    \qedsymbol\end{proof}
\fi

\begin{proof}[\underline{Proof of Theorem~\ref{thm:complexity_idealized_apex}}]
If $f(\bar{y}^{s})=f^{*}$ at the beginning of some stage, the claims for that stage are immediate. Otherwise, applying Theorem~\ref{thm:f_star} to the $s$-th call to $\apex$ with $\hat{x}_{s}^{1}=\bar{y}^{s}$ and $\tilde{l}=f^{*}$ gives
\[
f\brbra{\hat{x}_{s}^{t+1}}-f^{*}\leq \frac{6}{\brbra{t+2}\brbra{t+3}}\brbra{f\brbra{\bar{y}^{s}}-f^{*}}+\frac{2\bar{L}_{s}\brbra t}{\brbra{t+2}\brbra{t+3}}\dist^{2}\brbra{\bar{y}^{s},X^{*}}\ .
\]
Since $f$ is $\mu^{*}$-QG, $\dist^{2}\brbra{\bar{y}^{s},X^{*}}\leq \frac{2}{\mu^{*}}\brbra{f\brbra{\bar{y}^{s}}-f^{*}}$, and therefore
\[
f\brbra{\hat{x}_{s}^{t+1}}-f^{*}\leq \frac{6+4\bar{L}_{s}\brbra t/\mu^{*}}{\brbra{t+2}\brbra{t+3}}\brbra{f\brbra{\bar{y}^{s}}-f^{*}}\ .
\]
Since $t_s+1$ is the first inner iteration satisfying the stage-$s$ stopping rule, the stage has not stopped at iteration $t_s$ whenever $t_s\geq1$. Hence, for $t_s\geq1$,
\[
\theta<\frac{f\brbra{\hat{x}_{s}^{t_s}}-f^{*}}{f\brbra{\bar{y}^{s}}-f^{*}}
\leq
\frac{6+4\bar{L}_{s}\brbra{t_s - 1}/\mu^{*}}{\brbra{t_s+2}\brbra{t_s+3}}\ .
\]
The geometric decrease follows directly from the stopping rule:
\[
f\brbra{\bar y^{s + 1}}-f^*
=f\brbra{\hat x_s^{t_s+1}}-f^*
\leq \theta\brbra{f\brbra{\bar y^{s}}-f^*}.
\]
Iterating this inequality over $s=1,\ldots,S$ gives the claimed bound on $f(\bar y^{S + 1})-f^*$.

Define $L_{\text{avg}}=\max_{s=1,\ldots,S}\bar L_s(t_s-1)$. The stopping-index inequality above implies
$t_s+1\leq O(1)\sqrt{\frac{L_{\text{avg}}}{\mu^*}}$. Since each inner iteration uses $O(\bundleSize)$ gradient evaluations, the cost through the first $S$ stages is bounded by $O(1)\bundleSize\sqrt{\frac{L_{\text{avg}}}{\mu^*}}\,S$.
If the algorithm is stopped at the first stage with $f(\bar y^{S+1})-f^*\leq\vep$, then the geometric decrease gives
$S\leq O(1){\log\brbra{\frac{f\brbra{\bar y^1}-f^*}{\vep}}}$. This proves the stated gradient-evaluation complexity.

It remains to prove~\eqref{eq:upper_L_s_t_s}. If no event $\mcal E_s(t)$ occurs, then $\bar L_s(t_s)=0$ by definition and the convention $0/0=0$ gives the result. Otherwise, let $L_{s,t}$ be the stage-$s$ counterpart of $L_t$ in~\eqref{eq:L_N}. The definition of $\bar L$ and the identity
$\omega_t\alpha_t^2/2=4(t+2)/(t+3)\leq4$ give
\begin{equation*}
\bar{L}_{s}\brbra{t_{s} - 1}
\leq
\frac{4\sum_{t=1}^{t_s - 1}\onebf_{\mcal E_s(t)}L_{s,t}\norm{x_{s}^{t,r_{t}}-x_{s}^{t,l_{t}}}^{2}}
{\sum_{t=1}^{t_s - 1}\norm{x_{s}^{t,r_{t}}-x_{s}^{t,l_{t}}}^{2}}\ .
\end{equation*}
For every $t$ such that $\mcal E_s(t)$ occurs, the definition of $L_{s,t}$ gives
\[
\frac{1}{2}\alpha_{s,t}^{2}L_{s,t}\norm{x_{s}^{t,r_t}-x_{s}^{t,l_t}}^{2}=f\brbra{\hat{x}_{s}^{t+1}}-f^{*}-\brbra{1-\frac{3}{4}\alpha_{s,t}}\brbra{f\brbra{\hat{x}_{s}^{t}}-f^{*}}.
\]
Since $\mcal E_{s}(t)$ means
$f\brbra{\hat{x}_{s}^{t+1}}-f^{*}>\brbra{1-\alpha_{s,t}/2}\brbra{f\brbra{\hat{x}_{s}^{t}}-f^{*}}$
and the update rule ensures
$f\brbra{\hat{x}_{s}^{t+1}}\leq f\brbra{\hat{x}_{s}^{t}}$, we have
\begin{equation}
\frac{\alpha_{s,t}}{4}\brbra{f\brbra{\hat{x}_{s}^{t}}-f^{*}}
\leq
\frac{1}{2}\alpha_{s,t}^{2}L_{s,t}\norm{x_{s}^{t,r_t}-x_{s}^{t,l_t}}^{2}
\leq
\frac{3\alpha_{s,t}}{4}\brbra{f\brbra{\hat{x}_{s}^{t}}-f^{*}}\ .
\end{equation}
Here and below the displayed inequalities are used only on iterations with $\mcal E_s(t)$ and $1\leq t\leq t_s$. Since $t_s+1$ is the first stopping index, while $f(\hat{x}_s^t)$ is nonincreasing,
\[
\theta\brbra{f\brbra{\hat{x}_{s}^{1}}-f^{*}}
<
f\brbra{\hat{x}_{s}^{t}}-f^{*}
\leq
f\brbra{\hat{x}_{s}^{1}}-f^{*}.
\]
Using $\alpha_{s,t}=4/(t+3)$, we obtain
\[
\frac{\theta(t+3)}{8}\brbra{f\brbra{\hat{x}_{s}^{1}}-f^{*}}
\leq
L_{s,t}\norm{x_{s}^{t,r_{t}}-x_{s}^{t,l_{t}}}^{2}
\leq
\frac{3(t+3)}{8}\brbra{f\brbra{\hat{x}_{s}^{1}}-f^{*}} .
\]
Substituting these two bounds into the preceding upper bound on
$\bar L_s(t_s - 1)$ yields
\[
\bar{L}_{s}\brbra{t_{s} - 1}
\leq
\frac{12}{\theta}\cdot
\frac{\sum_{t=1}^{t_{s}-1}\onebf_{\mcal E_s(t)}(t+3)}
{\sum_{t=1}^{t_{s}-1}\brbra{\onebf_{\mcal E_s(t)}(t+3)\Big/L_{s,t}}}\ .
\]
Finally, Proposition~\ref{prop:upper_bound_LbarN} gives $L_{s,t}\leq\tilde{L}_{s,t}$, so replacing $L_{s,t}$ by $\tilde L_{s,t}$ in the denominator only weakens the bound. This proves~\eqref{eq:upper_L_s_t_s}.
\ifdefined\isarxiv
    \end{proof}
\else
    \qedsymbol\end{proof}
\fi

\begin{proof}[\underline{Proof of Proposition~\ref{prop:lower_bound_event_time_ratio}}]
Let $\mcal I_{s}^{c}:=\bcbra{t\in[t_{s}]:\mcal E_{s}^{c}\brbra t}$. For each $t\in\mcal I_s^c$, we have
$$f\brbra{\hat{x}_{s}^{t}}-\tilde{l}\le \brbra{1-\frac{\alpha_{s,t}}{2}}\brbra{f\brbra{\hat{x}_{s}^{t-1}}-\tilde{l}}
=\frac{t+1}{t+3}\brbra{f\brbra{\hat{x}_{s}^{t-1}}-\tilde{l}},$$ while for $t\notin\mcal I_s^c$ we only use $f(\hat{x}_s^t)\le f(\hat{x}_s^{t-1})$. Since $t_s+1$ is the first index with $f\brbra{\hat{x}_{s}^{t_s+1}}-\tilde l\le \theta\brbra{f\brbra{\hat{x}_{s}^{1}}-\tilde l}$,
\[
\theta<\frac{f\brbra{\hat{x}_{s}^{t_s}}-\tilde l}{f\brbra{\hat{x}_{s}^{1}}-\tilde l}
\le \prod_{t\in\mcal I_s^c}\frac{t+1}{t+3}
\le \prod_{t=t_s-|\mcal I_s^c|+1}^{t_s}\frac{t+1}{t+3}
=\frac{(t_s-|\mcal I_s^c|+2)(t_s-|\mcal I_s^c|+3)}{(t_s+2)(t_s+3)},
\]
where the third inequality uses that $\frac{t+1}{t+3}$ is increasing in $t$. Hence
$t_s-|\mcal I_s^c|\ge \left\lfloor \frac{-1+\sqrt{1+4\theta(t_s+2)(t_s+3)}}{2}\right\rfloor-1$. Among the first $t_s$ indices, exactly $|\mcal I_s^c|$ belong to $\mcal I_s^c$, so
\[
\sum_{t=1}^{t_s}\onebf_{\mcal E_s(t)}= t_s-|\mcal I_s^c|\ge \left\lfloor \frac{-1+\sqrt{1+4\theta(t_s+2)(t_s+3)}}{2}\right\rfloor-1.
\]
Dividing by $t_s$ gives the result, and the extra $\max\{\cdot,0\}$ is only to keep the lower bound nonnegative.
\ifdefined\isarxiv
    \end{proof}
\else
    \qedsymbol\end{proof}
\fi

\section{\protect\label{sec:Normalized-Wolfe-Certificate}Normalized Wolfe
certificate}

To guarantee that the iterates $\bcbra{f\brbra{\hat{x}^{t}}}$ converge
to the true optimum $f^{*}$, Algorithm~\ref{alg:apex} requires
prior knowledge of $\tilde{l}=f^{*}$. While this assumption holds
in certain applications~\cite{dang2017linearly,devanathan2023polyak}, it is generally
impractical in broader optimization settings. To address this limitation,
we instead maintain an estimate of a lower bound on $f^{*}$. Since the current function value provides a natural
upper bound, we choose a target level between the current lower and upper bounds. As optimization
progresses, this estimate can be refined, ultimately leading to the
exact optimum.

In this section, we introduce the Normalized Wolfe certificate ($\Wcer$),
a procedure that constructs a reliable lower bound on $f^{*}$ by
exploiting the QG condition of the objective. Specifically, we (i)
formally define the $\Wcer$ (Definition~\ref{def:W_certificate}), (ii) prove that it yields a valid underestimate
of $f^{*}$ under the QG condition (Proposition~\ref{prop:W_certificate_lower_bound}), and (iii) describe how to apply Algorithm~\ref{alg:AWG}
to generate a $\Wcer$.
\begin{defn}
[Normalized Wolfe certificate, $\Wcer$\label{def:W_certificate}]Let $\iota>0$ and $\nu\geq0$.
$\bar{y}\in X$ is called an $\brbra{\iota,\nu}$-normalized Wolfe stationary point of
$f$ if there exists an evaluation point set $\mcal P=\bcbra{z^{0}=\bar{y},z^{1},\ldots}\subseteq\mcal B\brbra{\bar{y},\iota}$
such that
\begin{equation}
\mcal V_{\mcal P,f}\brbra{\iota;\bar{y}}=\frac{1}{\iota}\max_{x\in\mcal B\brbra{\bar{y},\iota}\cap X}\Brbra{\psi_{\mcal P,f}\brbra{\bar{y}}-\psi_{\mcal P,f}\brbra x}\leq\nu\ ,\ \text{where}\ \psi_{\mcal P,f}\brbra x=\max_{z\in\mcal P}\bcbra{\ell_{f}\brbra{x;z}}\ .\label{eq:defn_V_func}
\end{equation}
In this case, we call $\mcal P$ an $\brbra{\iota,\nu}$ normalized Wolfe certificate for $f$ at point $\bar{y}$.
\end{defn}
Zhang and Sra~\cite{zhang2025linearly} introduced the $\Wcer$ as a natural extension of
the classical gradient-norm criterion to {\pws} objectives
by quantifying the maximal normalized decrease of a cutting plane model. This model is constructed from multiple supporting
hyperplanes evaluated within an $\iota$-neighborhood of the current
iterate, and the resulting certificate can be computed using only
a first-order oracle.

Similar to the gradient norm in smooth optimization, the $\Wcer$ exhibits a comparable
property: $\Wcer$ provides a trustworthy lower bound
on the optimum $f^{*}$ under the QG condition, which is 
demonstrated in the following  Proposition~\ref{prop:W_certificate_lower_bound}.

\begin{proposition}
\label{prop:W_certificate_lower_bound}Let $f$ be a convex function
satisfying $\mu$-QG. If $\mcal P$ constitutes an $\brbra{\iota,\nu}$-$\Wcer$
for the point $\bar{y}$, then the function value at $\bar{y}$ obeys
the bound $f\brbra{\bar{y}}-f^{*}\leq\max\bcbra{\iota\nu,{2\nu^{2}}/{\mu}}$.
\end{proposition}
\noindent\emph{Proof deferred.} The detailed proof is collected in Subsection~\ref{subsec:wolfe-certificate-detailed-proofs}.

Proposition~\ref{prop:W_certificate_lower_bound} establishes that the $\Wcer$ provides a rigorous and practically implementable criterion for certifying the near-optimality of any candidate point $\bar{y}\in X$.
Building on the bundle-based iterative framework for constructing a $\Wcer$ for $\pws$ functions developed in \cite{zhang2025linearly}, we introduce Algorithm~\ref{alg:AWG} below. This algorithm incorporates $\apex$, substantially reducing the number of oracle calls needed to generate a certificate with quality that matches or even exceeds that of the conventional BL-based method.

Algorithm~\ref{alg:AWG} takes as input a candidate point $\bar{y}$, a gap
estimate $\Delta$, the bundle size $\bundleSize$, a maximum certification
radius $\iota_{\max}$, and a scaling parameter $\beta>0$. In each outer
iteration, the subroutine $\onestep$ builds supporting cuts at the accelerated
evaluation points $\bcbra{\underline{x}^{t,i}}$ and computes the corresponding
projection points $\bcbra{x^{t,i}}$ from the reference point $\bar{y}$. The
resulting $\Wcer$ is formed from the evaluation points
$\bcbra{\underline{x}^{t,i}}$, rather than from the projection points
$\bcbra{x^{t,i}}$ used in the non-accelerated construction
of~\cite{zhang2025linearly}.
By Proposition~\ref{prop:converge}, the proximal distance of the returned
projection point from $\bar{y}$ is nondecreasing over the iterations. Hence, if
the subroutine detects an infeasible projection subproblem or if this distance
exceeds $\iota_{\max}$, Algorithm~\ref{alg:AWG} reaches
Line~\ref{enu:return_W_cer} and returns a valid
$\brbra{\iota_{\max},\brbra{1+\beta}\Delta/\iota_{\max}}$-$\Wcer$ for
$\bar{y}$. Moreover, when the gap estimate is valid, i.e.,
$f\brbra{\bar{y}}-f^*\leq\Delta$, Proposition~\ref{prop:AWG_W_cer_gen}
guarantees that Algorithm~\ref{alg:AWG} reaches this certificate-return line.
Thus, any exit through Line~\ref{enu:WG_mu_wrong} certifies that $\Delta$
underestimates the true optimality gap.

\begin{algorithm}
\caption{Accelerated $\protect\Wcer$ Generation, $\protect\mcal{AWG}\protect\brbra{\bar{y},\Delta,\iota_{\max},\bundleSize,\beta}$}
\label{alg:AWG}
\begin{algorithmic}[1]
\State \textbf{Initialize:} $x^{1}=\bar{y},\hat{x}^{1}=\bar{y},\tilde{l}=f\brbra{\bar{y}}-\brbra{1+\beta}\Delta$
\For{$t=1,\ldots$}
    \State $\brbra{\hat{x}^{t+1},x^{t+1},\text{\textbf{Flag}},L_{t},\norm{x^{t,l_{t}}-x^{t,r_{t}}}}\leftarrow\onestep\brbra{\bar{y},\hat{x}^{t},x^{t},\tilde{l},\bundleSize}$
    \If{$\flag$ is $\true$ or $\norm{x^{t+1}-\bar{y}}>\iota_{\max}$}
        \State \label{enu:return_W_cer} \textbf{Return} $\true$
    \EndIf
    \State Compute $\bar{L}\brbra t$ (defined in~\eqref{eq:L_N}) based on $\norm{x^{t,l_{t}}-x^{t,r_{t}}}$ \Comment{Just for analysis}
    \If{$f\brbra{\hat{x}^{t + 1}}< f(\bar{y})-\Delta$}
        \State \label{enu:WG_mu_wrong} \textbf{Return} $\false$
    \EndIf
\EndFor
\end{algorithmic}
\end{algorithm}

We establish Proposition~\ref{prop:AWG_W_cer_gen},
which offers a rigorous characterization of the behavior
of Algorithm~\ref{alg:AWG}.

%

\begin{proposition}
\label{prop:AWG_W_cer_gen}
Let $\omega_{t}=(t+2)(t+3)/2$ and $\alpha_{t}=4/(t+3)$ in
Algorithm~\ref{alg:AWG}. Suppose Algorithm~\ref{alg:AWG} terminates at
iteration $t_{\mcal W}$. Then the weighted mean empirical Lipschitz constant
$\bar L(t_{\mcal W}-1)$ and the termination iteration $t_{\mcal W}$ satisfy
{\small
\begin{equation}\label{eq:weighted_harmonic_mean_AWG}
\bar{L}\brbra{t_{\mcal W}-1}\leq
\frac{12\brbra{1+\beta}}{\beta}
\frac{\sum_{t=1}^{t_{\mcal W}-1}\onebf_{\mcal E\brbra t}\brbra{t+3}}
{\sum_{t=1}^{t_{\mcal W}-1}\onebf_{\mcal E\brbra t}\brbra{t+3}/\tilde{L}_{t}}
\text{ and }
t_{\mcal W}\leq
\min\bcbra{t:t>\sqrt{\frac{2\iota_{\max}^{2}\bar{L}\brbra t+6\brbra{1+\beta}\Delta}{\beta\Delta}}}.
\end{equation}}
Here $\tilde{L}_{t}$ is the local empirical Lipschitz constant for
$\onestep$ defined in~\eqref{eq:defn_tilde_L_t}, and $\mcal E(t)$ is defined
in~\eqref{eq:event_E_t}. If, in addition,
$f(\bar{y})-f^{*}\leq\Delta$, then Algorithm~\ref{alg:AWG} must enter
Line~\ref{enu:return_W_cer}. Consequently, there exists a point set
$\mcal P_{t_{\mcal W}}$ that forms an
$\brbra{\iota_{\max},\brbra{1+\beta}\Delta/\iota_{\max}}$-$\Wcer$ for
$\bar{y}$.
\end{proposition}
\noindent\emph{Proof deferred.} The detailed proof is collected in Subsection~\ref{subsec:wolfe-certificate-detailed-proofs}.

We make a few remarks regarding the result. First, although algorithms
such as the bundle-level method proposed in~\cite{zhang2025linearly}
or $\mcal{AWG}$ are useful for constructing certificates, it is important
to emphasize that the $\Wcer$ itself is independent of any specific algorithm.
In other words, $\Wcer$ is an inherent property of a candidate point,
and an algorithm merely serves as a mechanism to generate a point
set that qualifies as a $\Wcer$ for that candidate. Second, $\Wcer$
is also independent of the $\mu$-QG condition. If the goal is to
construct an upper bound on the optimality gap, then a QG
assumption is indeed required (see Proposition~\ref{prop:W_certificate_lower_bound}). However, if $\Wcer$ is adopted purely
as a termination criterion (like the gradient norm), the QG condition is not necessary.

\subsection{Detailed proofs\label{subsec:wolfe-certificate-detailed-proofs}}
We now provide the detailed proofs of the certificate properties, the accelerated certificate-generation guarantee, and the certificate-transfer result stated above. To organize the arguments, we first establish two intermediate lemmas that will be used in the subsequent proofs.
\begin{lem}
\label{lem:V_non_increasing}$\mcal V_{\mcal P,f}\brbra{\iota;\bar{y}}$
(defined in~\eqref{eq:defn_V_func}) is non-increasing with respect
to $\iota$.
\end{lem}
\begin{proof}
Let some $\iota_{1}>\iota_{2}$ be given and let $x_{\iota_{1}}$
and $x_{\iota_{2}}$ denote the respective optimal solution of $\mcal V_{\mcal P,f}\brbra{\iota_{1};\bar{y}}$
and $\mcal V_{\mcal P,f}\brbra{\iota_{2};\bar{y}}$. By the convexity
of $X$, we have point $\tilde{x}:=\bar{y}+\frac{\iota_{2}}{\iota_{1}}\brbra{x_{\iota_{1}}-\bar{y}}\in X.$
Since $\psi_{\mcal P,f}\brbra x$ is convex and $\bar{y},\tilde{x}$
and $x_{\iota_{1}}$ lie on the same line, we get from the monotonicity
of the secant line:
\begin{equation}
\frac{\psi_{\mcal P,f}\brbra{\bar{y}}-\psi_{\mcal P,f}\brbra{\tilde{x}}}{\norm{\tilde{x}-\bar{y}}}\geq\frac{\psi_{\mcal P,f}\brbra{\bar{y}}-\psi_{\mcal P,f}\brbra{x_{\iota_{1}}}}{\norm{x_{\iota_{1}}-\bar{y}}}\ .\label{eq:secant}
\end{equation}
Combining~\eqref{eq:secant}, $\norm{\bar{y}-\tilde{x}}=\iota_{2}\norm{x_{\iota_{1}}-\bar{y}}/\iota_{1}$
gives $\brbra{\psi\brbra{\bar{y}}-\psi\brbra{\tilde{x}}}/\iota_{2}\geq\brbra{\psi\brbra{\bar{y}}-\psi\brbra{x_{\iota_{1}}}}/\iota_{1}$.
By the definition of $x_{\iota_{2}}$, we have $\brbra{\psi\brbra{\bar{y}}-\psi\brbra{x_{\iota_{1}}}}/\iota_{1}\leq\brbra{\psi\brbra{\bar{y}}-\psi\brbra{\tilde{x}}}/\iota_{2}\leq\brbra{\psi\brbra{\bar{y}}-\psi\brbra{x_{\iota_{2}}}}/\iota_{2}$.
 This completes the proof of that $\mcal V_{\mcal P,f}\brbra{\iota;\bar{y}}$
is non-increasing with respect to $\iota$.
\ifdefined\isarxiv
    \end{proof}
\else
    \qedsymbol\end{proof}
\fi

\begin{lem}
\label{lem:lower_bound_of_ball}
After the first $t$ calls to $\onestep$ in Algorithm~\ref{alg:AWG}, let
$
\mcal P_t:=\bcbra{\underline{x}^{i,j}:1\leq i\leq t,\ 1\leq j\leq \bundleSize}
$
be the set of generated evaluation points, and let $x^{t+1}$ be the projection
point returned at iteration $t$. For any
$\iota\in(0,\norm{x^{t+1}-\bar y})$, define
$\tilde{\mcal P}(\iota):= \mcal P_t\cap\mcal B(\bar y,\iota)$. Then
$\min_{x\in\mcal B(\bar y,\iota)\cap X}
\psi_{\tilde{\mcal P}(\iota),f}(x)>\tilde l .$
\end{lem}

\begin{proof}
Let $\Psi_{t,f}$ denote the cutting-plane feasible region generated by all cuts
constructed through iteration $t$:
\begin{equation}
\Psi_{t,f}=\bcbra{{x}\in X:\psi_{\mcal P_{t},f}\brbra{{x}}\leq\tilde{l}}.\label{eq:inducive_claim}
\end{equation}
    We first claim for any $\hat{x}\in X$ such that $\inner{\hat{x}-x^{t+1,0}}{x^{t+1,0}-\bar{y}}<0$, we have $\psi_{\mcal P_{t},f}\brbra{\hat{x}}>\tilde{l}$. This is equivalent
to for every $\hat{x}\in\Psi_{t,f}\cap X$, we claim $\inner{\hat{x}-x^{t+1,0}}{x^{t+1,0}-\bar{y}}\geq0$
holds for any $t\geq1$.

We prove the second claim by induction. For base case $t=1$, we have $x^{2,0}=x^{1,\bundleSize}$.
Since $x^{1,0}=\bar{y}$, then $\bar{X}\brbra{t,0}$ is
eliminated. Hence, $x^{1,\bundleSize}$ is the closest point to set $\bcbra{x:\psi_{\mcal P_{1},f}\brbra x\leq\tilde{l}}$,
which implies that for any point satisfies $\hat{x}\in\Psi_{1,f}\cap X$,
we have $\inner{\hat{x}-x^{2,0}}{x^{2,0}-\bar{y}}\geq0$. Suppose
the claim holds for $1,\ldots,t$, we
aim to prove if every $\hat{x}\in\Psi_{t,f}\cap X$, we have $\inner{\hat{x}-x^{t+1,0}}{x^{t+1,0}-\bar{y}}\geq0$.
Since $\hat{x}\in\Psi_{t-1,f}$ implies $\inner{\hat{x}-x^{t,0}}{x^{t,0}-\bar{y}}\geq0$
by induction. Hence, the goal of our proof is equivalent to for any
$\hat{x}\in X$, $\inner{\hat{x}-x^{t,0}}{x^{t,0}-\bar{y}}\geq0$
and $\psi_{\mcal P_{t}\backslash\mcal P_{t-1},f}\brbra{\hat{x}}\leq\tilde{l}$,
we have $\inner{\hat{x}-x^{t+1,0}}{x^{t+1,0}-\bar{y}}\geq0$. Note
that the region $$\bcbra{x\in X:\inner{\hat{x}-x^{t,0}}{x^{t,0}-\bar{y}}\geq0,\psi_{\mcal P_{t}\backslash\mcal P_{t-1},f}\brbra{\hat{x}}\leq\tilde{l}}$$
of $\hat{x}$ is exactly the feasible region $X(t,\bundleSize)$ of last bundle
subproblem at $t$-th iteration. Since $x^{t+1,0}=x^{t,\bundleSize}$ is the
closest point to the set $X(t,\bundleSize)$, then we have any $\hat{x}\in X(t,\bundleSize)$,
$\inner{\hat{x}-x^{t+1,0}}{x^{t+1,0}-\bar{y}}\geq0$ holds.

Now, we are ready to prove the main result of Lemma~\ref{lem:lower_bound_of_ball}.
 Since sequence $\bcbra{\norm{x^{t,i}-\bar{y}}}$ is non-decreasing
by Proposition~\ref{prop:converge} and $\norm{x^{1,0}-\bar{y}}=0$. Note
that $x^{i,\bundleSize}=x^{i+1,0}$ holds for any $i\in\bsbra{t-1}$. Hence
there exists an index pair $\brbra{\tilde{t},j}\in\bsbra t\times\bsbra{\bundleSize}$
such that $\iota\in\left(\norm{x^{\tilde{t},j-1}-\bar{y}},\norm{x^{\tilde{t},j}-\bar{y}}\right]$.
We first denote $\mcal P\brbra{\tilde{t},j}=\bcbra{\underline{x}^{\tilde{t},1},\ldots,\underline{x}^{\tilde{t},j}}$,
then the set $X\brbra{t,i}$ in Algorithm~\ref{alg:One-step} can
be rewritten as 
$$X\brbra{\tilde{t},j}=\bcbra{x\in X:\psi_{\mcal P\brbra{\tilde{t},j},f}\brbra x\leq\tilde{l}\ ,\ \inner{x-x^{\tilde{t},0}}{x^{\tilde{t},0}-\bar{y}}\geq0}.$$
Since $x^{\tilde{t},j}$ is the closest point to set $X\brbra{\tilde{t},j}$.
Hence, $\forall\hat{x}\in X\brbra{\tilde{t},j}$, we have $\norm{\hat{x}-\bar{y}}\geq\norm{x^{\tilde{t},j}-\bar{y}}\geq\iota$.
This implies that for any $\hat{x}\in\mcal B\brbra{\bar{y},\iota}\cap X$,
we have $\psi_{\mcal P\brbra{\tilde{t},j},f}\brbra{\hat{x}}>\tilde{l}$
or $\inner{\hat{x}-x^{\tilde{t},0}}{x^{\tilde{t},0}-\bar{y}}<0$.
It follows from the inductive claim that $\hat{x}\in X$ with $\inner{\hat{x}-x^{\tilde{t},0}}{x^{\tilde{t},0}-\bar{y}}<0$ implies $\psi_{\mcal P\brbra{\tilde{t}-1},f}\brbra{\hat{x}}>\tilde{l}$. Therefore, for any $\hat{x}\in\mcal B\brbra{\bar{y},\iota}\cap X$, we have $\psi_{\mcal P\brbra{\tilde{t},j},f}\brbra{\hat{x}}>\tilde{l}$ or $\psi_{\mcal P\brbra{\tilde{t}-1},f}\brbra{\hat{x}}>\tilde{l}$, which completes the proof.
\ifdefined\isarxiv
    \end{proof}
\else
    \qedsymbol\end{proof}
\fi

\begin{proof}[\underline{Proof of Proposition~\ref{prop:W_certificate_lower_bound}}]
Since the objective function $f$ is convex, then we have $\psi_{\mcal P,f}\brbra x\leq f\brbra x$.
Furthermore, condition $\bar{y}\in\mcal P$ implies $\psi_{\mcal P,f}\brbra{\bar{y}}\leq f\brbra{\bar{y}}$.
We consider two cases based on the distance between $x_{p}^{*}$ and
$\bar{y}$, where $x_{p}^{*}=\argmin_{x\in X^{*}}\norm{\bar{y}-x}$
and $X^{*}:=\bcbra{x\in X:f\brbra x=f^{*}}$.

Case 1. $x_{p}^{*}\in\mcal B\brbra{\bar{y},\iota}$. Then we
have
\[
\begin{split}f\brbra{\bar{y}}-f^{*}\leq f\brbra{\bar{y}}-\psi_{\mcal P,f}\brbra{x_{p}^{*}} & \aleq\max_{x\in\mcal B\brbra{\bar{y},\iota}\cap X}\psi_{\mcal P,f}\brbra{\bar{y}}-\psi_{\mcal P,f}\brbra x\end{split}
\leq \iota\nu\ ,
\]
where $(a)$ follows from the convexity of $f$ and $\bar{y}\in\mcal P$
and these facts implies that $f\brbra{\bar{y}}=\psi_{\mcal P,f}\brbra{\bar{y}}$.

Case 2. $x_{p}^{*}\in\mcal B^{c}\brbra{\bar{y},\iota}$. From
the radius argument that $\norm{x_{p}^{*}-\bar{y}}\geq\iota$ and
$\mcal V_{\mcal P,f}\brbra{\iota;\bar{y}}$ is non-increasing with
respect to $\iota$, we have
\begin{equation}
\begin{split}f\brbra{\bar{y}}-f^{*} & \leq\psi_{\mcal P,f}\brbra{\bar{y}}-\psi_{\mcal P,f}\brbra{x_{p}^{*}}\\
 & \leq\max_{x\in\mcal B\brbra{\bar{y},\norm{x_{p}^{*}-\bar{y}}}}\bcbra{\psi_{\mcal P,f}\brbra{\bar{y}}-\psi_{\mcal P,f}\brbra x}\\
 & =\norm{x_{p}^{*}-\bar{y}}\mcal V_{\mcal P,f}\brbra{\norm{x_{p}^{*}-\bar{y}};\bar{y}}\\
 & \leq\norm{x_{p}^{*}-\bar{y}}\mcal V_{\mcal P,f}\brbra{\iota;\bar{y}}\leq\norm{x_{p}^{*}-\bar{y}}\nu
\end{split}
\ .\label{eq:lower_f_star}
\end{equation}
Combining the above inequality and $f\brbra{\bar{y}}-f^{*}\geq\frac{\mu}{2}\norm{\bar{y}-x_{p}^{*}}^{2}$
yields $\frac{\mu}{2}\norm{\bar{y}-x_{p}^{*}}^{2}\leq\norm{x_{p}^{*}-\bar{y}}\nu,$
which implies that
\begin{equation}
\norm{x_{p}^{*}-\bar{y}}\leq\frac{2\nu}{\mu}\ .\label{eq:upper_x_star_y_bar}
\end{equation}
Putting~\eqref{eq:lower_f_star} and~\eqref{eq:upper_x_star_y_bar}
together, we have $f\brbra{\bar{y}}-f^{*}\leq {2}\nu^{2}/{\mu}$.

Summarizing two cases completes our proof.
\ifdefined\isarxiv
    \end{proof}
\else
    \qedsymbol\end{proof}
\fi

\begin{proof}[\underline{Proof of Proposition~\ref{prop:AWG_W_cer_gen}}]
Set $\tilde l=f(\bar y)-(1+\beta)\Delta$. Algorithm~\ref{alg:AWG} is the
certificate-generation version of $\apex$, with $\alpha_t=4/(t+3)$ and
$\omega_t=(t+2)(t+3)/2$. Before termination, neither return branch has been
triggered; hence, for $1\leq t\leq t_{\mcal W}-1$,
\[
\beta\Delta\leq f(\hat{x}^{t+1})-\tilde l\leq(1+\beta)\Delta .
\]
The first inequality in~\eqref{eq:weighted_harmonic_mean_AWG} follows from the
same weighted-average argument as in~\eqref{eq:upper_L_s_t_s}, with
$\theta=\beta/(1+\beta)$, and from Proposition~\ref{prop:upper_bound_LbarN},
which replaces $L_t$ by $\tilde L_t$ and gives the constant
$12(1+\beta)/\beta$.

We next prove the second inequality in~\eqref{eq:weighted_harmonic_mean_AWG}.
It suffices to show that any integer $t$ satisfying
\[
t>\sqrt{\frac{2\iota_{\max}^{2}\bar L\brbra t
    +6(1+\beta)\Delta}{\beta\Delta}}
\]
cannot be a completed nonterminating iteration. Suppose, to the contrary, that
Algorithm~\ref{alg:AWG} has not terminated by the end of such an iteration $t$.
Then all subproblems up to iteration $t$ are feasible, the radius test at
Line~\ref{enu:return_W_cer} gives $\norm{x^{t+1}-\bar y}\leq\iota_{\max}$, and
the false-decrease test at Line~\ref{enu:WG_mu_wrong} gives
\[
f(\hat x^{t+1})-\tilde l\geq\beta\Delta .
\]
Since $\beta>0$, such a $t$ is at least $2$. Applying
Proposition~\ref{prop:converge} with
$\omega_t=(t+2)(t+3)/2$, $\alpha_t=4/(t+3)$, and
$\omega_1(1-\alpha_1/2)=3$ yields
\[
\frac{1}{2}(t+2)(t+3)\brbra{f(\hat x^{t+1})-\tilde l}
-3(1+\beta)\Delta
\leq \bar L\brbra t\norm{x^{t+1}-\bar y}^{2}.
\]
Hence
\[
\frac{1}{2}(t+2)(t+3)\beta\Delta-3(1+\beta)\Delta
\leq \bar L\brbra t\,\iota_{\max}^{2}.
\]
The displayed condition on $t$ implies the reverse strict inequality, since
$(t+2)(t+3)>t^{2}$, a contradiction. Thus no iteration before
$t_{\mcal W}$ can satisfy the displayed condition, proving the second
inequality in~\eqref{eq:weighted_harmonic_mean_AWG}.

Suppose next that Algorithm~\ref{alg:AWG} reaches Line~\ref{enu:return_W_cer}.
Let $\mcal P_{t_{\mcal W}}$ be the set of evaluation points generated up to the
terminating call and lying in $\mcal B(\bar y,\iota_{\max})$. Then
$\mcal P_{t_{\mcal W}}\subseteq\mcal B(\bar y,\iota_{\max})$. If the return is
triggered by the radius test, Lemma~\ref{lem:lower_bound_of_ball} with
$\iota=\iota_{\max}$ gives
$\psi_{\mcal P_{t_{\mcal W}},f}(x)>\tilde l$ for all
$x\in\mcal B(\bar y,\iota_{\max})\cap X$. If the return is triggered by an
infeasible subproblem, the same projection-separation argument used in
Lemma~\ref{lem:lower_bound_of_ball} gives the same lower bound. Since convexity
gives $\psi_{\mcal P_{t_{\mcal W}},f}(\bar y)\leq f(\bar y)$, we obtain
\[
\mcal V_{\mcal P_{t_{\mcal W}},f}\brbra{\iota_{\max};\bar y}
<\frac{f(\bar y)-\tilde l}{\iota_{\max}}
=\frac{(1+\beta)\Delta}{\iota_{\max}}.
\]
Thus $\mcal P_{t_{\mcal W}}$ forms an
$\brbra{\iota_{\max},(1+\beta)\Delta/\iota_{\max}}$-$\Wcer$ for $\bar y$.

It remains to rule out false termination under the additional assumption
$f(\bar y)-f^*\leq\Delta$. Since $f(\hat x^{t+1})\geq f^*$, the decrease test in
Line~\ref{enu:WG_mu_wrong} cannot hold. Therefore the main-loop termination must
occur at Line~\ref{enu:return_W_cer}.
\ifdefined\isarxiv
    \end{proof}
\else
    \qedsymbol\end{proof}
\fi

\section{\protect\label{sec:Restarted_mu_known}Restarted $\protect\apex$
for known $\mu^{*}$}
By leveraging the $\Wcer$, we can now address convex {\pws}
optimization under a $\mu^{*}$-QG condition without requiring prior knowledge
of the optimal value $f^{*}$. We maintain two evolving estimates: one is a lower bound ($\bar{f}-\Delta$)
on $f^{*}$ obtained from $\Wcer$ and the other is an upper bound
($\bar{f}$) on $f^{*}$ tracked by the best objective value encountered
along the iterates. We then restrict our search to the level between
these two bounds. Algorithm~\ref{alg:rapex_mu_star_known} implements
this idea via a simple restart mechanism. Algorithm~\ref{alg:rapex_mu_star_known}
initializes a valid gap estimate $\Delta_{1}$ by using convexity and the QG
condition. Subsequently, Algorithm~\ref{alg:rapex_mu_star_known} sets
the level between the current upper and lower bounds and runs
the accelerated bundle-level procedure for $\pws$ functions until one of two
restart conditions is met:
\begin{itemize}
\item Upper bound improvement (Line~\ref{enu:upper_bound_shrink}): the
upper bound improves sufficiently; specifically, the gap between the
upper bound and the previous lower bound decreases by a factor $\theta$.
\item Lower bound improvement (Line~\ref{enu:lower_bound_shrink}): the method produces an
iterate $x_{s}^{t+1}$ with $\norm{x_{s}^{t+1}-\bar{y}^{s}}$ large enough
to generate a new $\Wcer$, which by Proposition~\ref{prop:W_certificate_lower_bound}
yields a strictly tighter lower bound $(\bar{f}-\theta\Delta)$ on
$f^{*}$.
\end{itemize}
By alternately tightening the upper and lower bounds through these
two triggers, Algorithm~\ref{alg:rapex_mu_star_known} converges
to $f^{*}$ without requiring explicit knowledge of $f^{*}$.
\begin{algorithm}
\caption{Restarted $\protect\apex$ ($\text{r}\protect\apex$) for known $\mu^{*}$}
\label{alg:rapex_mu_star_known}
\begin{algorithmic}[1]
\Require the initial point $\bar{y}^{1}$, QG modulus $\mu^{*}$, the number of cuts $\bundleSize$, shrink factor $\theta\in\brbra{\frac{1}{2},1}$
\State \textbf{Initialize:} upper bound $\bar{f}^{1}=f\brbra{\bar{y}^{1}}$, optimality gap $\Delta_{1}=\frac{2\norm{f^{\prime}\brbra{\bar{y}^{1}}}^{2}}{\mu^{*}}$, lower bound $\underline{f}^{1}=\bar{f}^{1}-\Delta_{1}$, $s=1$
\While{$\true$}
    \State \label{enu:set_level} $x_{s}^{1}=\bar{y}^{s},\hat{x}_{s}^{1}=\bar{y}^{s},\tilde{l}^{s}=\bar{f}^{s}-\theta\Delta_{s}$
    \For{$t=1,\ldots$}
        \State $\brbra{\hat{x}_{s}^{t+1},x_{s}^{t+1},\text{\ensuremath{\flag}},\ldots}\leftarrow\onestep\brbra{\bar{y}^{s},\hat{x}_{s}^{t},x_{s}^{t},\tilde{l}^{s},\bundleSize}$
        \If{$f\brbra{\hat{x}^{t+1}_{s}}-\underline{f}^{s}\leq\theta\Delta_{s}$}
            \State \label{enu:upper_bound_shrink} $\bar{y}^{s+1}=\hat{x}^{t+1}_{s},\bar{f}^{s+1}=f\brbra{\hat{x}^{t+1}_{s}},\Delta_{s+1}=\bar{f}^{s+1} - \underline{f}^{s},\underline{f}^{s+1}=\underline{f}^{s}$, $s=s+1$, and go to Line \ref{enu:set_level}
        \ElsIf{$\norm{x_{s}^{t+1}-\bar{y}^{s}}^{2}>\frac{2\theta\Delta_{s}}{\mu^*}$ or $\flag$ is $\true$}
            \State \label{enu:lower_bound_shrink} $\bar{y}^{s+1}=\hat{x}^{t+1}_s,\bar{f}^{s+1}=f\brbra{\hat{x}^{t+1}_s},\Delta_{s+1}=\theta\Delta_{s},\underline{f}^{s+1}=\bar{f}^{s}-\Delta_{s+1}$, $s=s+1$, and go to Line \ref{enu:set_level}
        \EndIf
    \EndFor
\EndWhile
\end{algorithmic}
\end{algorithm}

In the following Theorem~\ref{thm:mu_known}, we state the convergence
result of Algorithm~\ref{alg:rapex_mu_star_known} formally.
\begin{thm}
\label{thm:mu_known}
Suppose that $f$ satisfies the QG condition with modulus $\mu^*$.
Let $\omega_{t}=\brbra{t+2}\brbra{t+3}/2$ and $\alpha_{t}=4/\brbra{t+3}$ in Algorithm~\ref{alg:rapex_mu_star_known}.
Then the following statements hold:
\begin{enumerate}
\item If Algorithm~\ref{alg:rapex_mu_star_known} reaches Line~\ref{enu:lower_bound_shrink}, then there exists a point set $\mcal P_s$ that forms a
$\brbra{\sqrt{2\theta\Delta_{s}/\mu^{*}},\sqrt{\mu^{*}\theta\Delta_{s}/2}}$-$\Wcer$
for $\bar{y}^{s}$.
\item For any $\vep\in\brbra{0,{2\norm{f^{\prime}\brbra{\bar{y}^{1}}}^{2}}/{\mu^{*}}}$, to obtain a point $\bar{y}^{S+1}$ satisfying $f\brbra{\bar{y}^{S+1}}-f^{*}\leq\vep$, the total number of gradient evaluations is at most
\begin{equation}\label{eq:L_avg_complexity}
O(1)\bundleSize\max\bcbra{\sqrt{\frac{L_{\text{avg}}}{\mu^{*}}},1}\log\brbra{\frac{2\norm{f^{\prime}\brbra{\bar{y}^{1}}}^{2}}{\mu^{*}\vep}}\ ,
\end{equation}
where ${L}_{\text{avg}}=\max_{s=1,\ldots,S}\bar{L}_{s}(t_s - 1)$ and we adopt convention $\bar{L}_s{(0)} = 0$.
Here $t_s$ denotes the first inner-loop value whose call triggers a restart at stage $s$; the corresponding restart output is $\hat{x}_s^{t_s+1}$. Moreover, for every completed stage with $t_s>1$, $\bar{L}_s\brbra{t_s-1}$ admits the following upper bound: 
\begin{equation}\label{eq:upper_bound_mu_L_bar}
\bar{L}_s\brbra{t_s-1}
\leq
\frac{12\theta}{2\theta-1}\cdot
\frac{\sum_{t=1}^{t_s-1}\onebf_{\mcal E_{s}\brbra t}(t+3)}
{\sum_{t=1}^{t_s-1}\onebf_{\mcal E_{s}\brbra t}(t+3)/\tilde{L}_{s,t}}\ .
\end{equation}
Here $\bar{L}_s(t_s-1)$ is the weighted average of empirical Lipschitz constants defined in~\eqref{eq:L_N} over the pre-exit calls of stage $s$, $\tilde{L}_{s,t}$ is the local empirical Lipschitz constant in~\eqref{eq:defn_tilde_L_t}, $\mcal E_s(t)$ is the event in~\eqref{eq:event_E_t}, and we use the conventions $0/0=0$.
\item If $f$ is also a $\brbra{k,L}$-$\pws$ function with $\bundleSize\geq k$, then $L_{\text{avg}}\leq O(1)L$. Consequently, the total number of gradient evaluations needed to obtain $f(\bar{y}^{S+1})-f^{*}\leq\vep$ is at most
\begin{equation*}
O(1)\bundleSize\max\bcbra{\sqrt{\frac{L}{\mu^{*}}},1}\log\brbra{\frac{2\norm{f^{\prime}\brbra{\bar{y}^{1}}}^{2}}{\mu^{*}\vep}}\ .
\end{equation*}
\end{enumerate}
\end{thm}
\noindent\emph{Proof deferred.} The detailed proof is collected in Subsection~\ref{subsec:known-mu-detailed-proofs}.

A couple of remarks are in order regarding the convergence results
above. First, for a prescribed accuracy $\vep$, the number of oracle
evaluations required to obtain an $\vep$-optimal solution is shown in~\eqref{eq:L_avg_complexity}.
Once the bundle of supporting cuts fully captures
the underlying $\pws$ structure so that $L_{\text{avg}}\leq O(1)L$ (see Proposition~\ref{prop:upper_bound_LbarN} and the corresponding remark for details), this rate matches,
up to the multiplicative factor $\bundleSize$, the optimal complexity bound
for smooth convex optimization under the $\mu^{*}$-QG condition.
Moreover, since the number of active pieces may evolve as the algorithm
progresses, the method naturally transitions between a coarse piecewise
smooth regime and a smooth regime, adaptively accelerating convergence
once the structural decomposition is well resolved. Second, Algorithm~\ref{alg:rapex_mu_star_known}
requires only knowledge of the QG modulus $\mu^{*}$.
It does not depend on the global Lipschitz constant $L$, nor does
it require prior knowledge of the exact number of active pieces $k$.

\subsection{Detailed proofs\label{subsec:known-mu-detailed-proofs}}

We close this section with the proof of the known-QG convergence guarantee.

\begin{proof}[\underline{Proof of Theorem~\ref{thm:mu_known}}]
We first verify that the initial lower bound is valid. By convexity and the $\mu^*$-QG condition,
$f\brbra{\bar y^1}-f^*\leq\norm{f^{\prime}\brbra{\bar y^1}}\dist\brbra{\bar y^1,X^*}$ and
$\dist\brbra{\bar y^1,X^*}\leq\sqrt{2\brbra{f\brbra{\bar y^1}-f^*}/\mu^*}$; hence
$f(\bar y^1)-f^*\leq\Delta_1$. Therefore $\underline f^1=\bar f^1-\Delta_1\leq f^*$.

We next prove the certificate claim. When Line~\ref{enu:lower_bound_shrink} is reached, the certificate-generation condition in Proposition~\ref{prop:AWG_W_cer_gen} holds with level
$\tilde l^{s}=\bar f^{s}-\theta\Delta_{s}$ and radius
$\sqrt{2\theta\Delta_{s}/\mu^*}$. Hence by the similar proof process of Proposition~\ref{prop:AWG_W_cer_gen}, we have that there exists a point set $\mcal P_s$ forming a
$\brbra{\sqrt{2\theta\Delta_{s}/\mu^{*}},\sqrt{\mu^{*}\theta\Delta_{s}/2}}$-$\Wcer$
for $\bar y^{s}$.

This certificate also justifies the lower-bound update. Indeed, Proposition~\ref{prop:W_certificate_lower_bound} gives
$f\brbra{\bar y^{s}}-f^*\leq\theta\Delta_{s}$.
Therefore Line~\ref{enu:lower_bound_shrink} sets
$\underline f^{s+1}=\bar f^{s}-\theta\Delta_{s}\leq f^*$.
Line~\ref{enu:upper_bound_shrink} leaves the lower bound unchanged, so by induction $\underline f^s\leq f^*$ for every stage.
Moreover, either restart condition gives $\Delta_{s+1}\leq\theta\Delta_s$: this is immediate at Line~\ref{enu:lower_bound_shrink}.  Let $t_s$ be the first inner-loop value whose call triggers a restart at stage $s$. Then we have
$\Delta_{s+1}=f(\hat x_s^{t_s+1})-\underline f^{s}\leq\theta\Delta_{s}$ at Line~\ref{enu:upper_bound_shrink}. Since $\hat{x}_s^{t_s+1}$ is chosen with $f(\hat{x}_s^{t_s+1})\leq f(\bar y^s)$, the invariant $f(\bar y^{s+1})-f^*\leq\Delta_{s+1}$ is preserved. Hence
$f\brbra{\bar y^{S+1}}-f^*\leq\Delta_{S+1}\leq\theta^{S}\Delta_1$.

It remains to bound the cost of each stage. If $t_s=1$, the stage cost is already bounded by an absolute constant times $\bundleSize$. Otherwise, for every pre-exit call $t=1,\ldots,t_s-1$, neither restart condition has occurred; in particular,
\[
f\brbra{\hat x_s^{t+1}}-\underline f^{s}>\theta\Delta_{s},\qquad
\norm{x_s^{t+1}-\bar y^{s}}^2\leq \frac{2\theta\Delta_{s}}{\mu^*},\qquad
\flag=\false.
\]
These inequalities are only asserted for $t=1,\ldots,t_s-1$; at $t=t_s$, one of the restart conditions holds.
For any such pre-exit index $t$, since $\tilde l^{s}=\bar f^{s}-\theta\Delta_{s}$ and
$\underline f^{s}=\bar f^{s}-\Delta_{s}$, the first inequality implies
$f\brbra{\hat x_s^{t_s}}-\tilde l^{s}>\brbra{2\theta-1}\Delta_{s}$.
Applying Proposition~\ref{prop:converge} to this shifted inner sequence yields
\[
\omega_{t}\brbra{2\theta-1}\Delta_{s}
-\omega_1\brbra{1-\frac{\alpha_1}{2}}\theta\Delta_{s}
\leq
\bar L_s(t_s - 1)\norm{x_s^{t_s}-\bar y^{s}}^2.
\]
Using $\omega_1(1-\alpha_1/2)=3$ and the radius bound above, we get
\[
\brbra{\omega_{t}(2\theta-1)-3\theta}\Delta_{s}
\leq
\frac{2\theta\bar L_s(t_s - 1)}{\mu^*}\Delta_{s}.
\]
Using $\omega_t=(t+2)(t+3)/2$ yields
\[
t_s\leq O(1)\sqrt{\frac{\bar L_s(t_s-1)}{\mu^*}}\ .
\]
Since $\theta$ is fixed, each completed stage uses at most
$O(1)\sqrt{L_{\text{avg}}/\mu^*}$ inner iterations, and each inner iteration uses $O(\bundleSize)$ gradient evaluations. The number of completed stages needed to make $\Delta_{S+1}\leq\vep$ is at most $O(1)\log\brbra{2\norm{f^{\prime}\brbra{\bar{y}^{1}}}^{2}/(\mu^{*}\vep)}$, which proves the claimed total complexity.

Finally, the proof of~\eqref{eq:upper_bound_mu_L_bar} follows exactly as in~\eqref{eq:upper_L_s_t_s}. During the pre-exit calls of a stage,
$\brbra{2\theta-1}\Delta_{s}<f\brbra{\hat x_s^{t+1}}-\tilde l^{s}\leq\theta\Delta_{s}$.
Substituting this bound into the same weighted-average argument gives~\eqref{eq:upper_bound_mu_L_bar}.
If $f$ is $\brbra{k,L}$-$\pws$ and $\bundleSize\geq k$, then $\tilde L_{s,t}\leq L$ for all relevant calls to $\onestep$. Therefore $\bar L_s(t_s-1)\leq O(1)L$, so $L_{\text{avg}}\leq O(1)L$, and the last complexity bound follows.
\ifdefined\isarxiv
    \end{proof}
\else
    \qedsymbol\end{proof}
\fi

\section{\protect\label{sec:Restarted_mu_unknown}Restarted $\protect\apex$
for unknown $\mu^{*}$}

In this section, we propose an algorithm that eliminates the requirement for
a priori knowledge of the QG modulus $\mu^{*}$. Algorithm~\ref{alg:mu_unknown}
maintains a single estimate $\hat{\mu}_{s}>0$, which induces the gap estimate
$\tilde{\Delta}_{s}$, and checks this estimate through a guess-and-check mechanism.
Proposition~\ref{prop:generate_W_certificate} justifies the certificate
interpretation of the recorded pair $(\Delta_s,\mu_s)$: at an accepted stage,
it gives a
$\brbra{\sqrt{2\Delta_s/\mu_s},\sqrt{\mu_s\Delta_s/2}}$-$\Wcer$ for
$\bar y^s$. The unified refresh step at Line~\ref{enu:refresh_mu_unknown}
records $(\Delta_s,\mu_s)=(\theta\tilde{\Delta}_s,\hat{\mu}_s)$ when
$\lowerflag=\true$, and
$(\Delta_s,\mu_s)=((1+\beta)\tilde{\Delta}_s,\hat{\mu}_s)$ when
$\lowerflag=\false$. These values are used later when
Line~\ref{enu:mu_unknown_half_rho} recomputes $\tilde{\Delta}_s$ after rejecting
the current QG estimate.
\begin{algorithm}
\caption{Accelerated Gap Reduction, $\protect\mcal{AGR}\protect\brbra{\bar{y},\Delta,\mu,\bundleSize,\theta}$}
\label{alg:AGR}
\begin{algorithmic}[1]
\State \textbf{Initialize:} $\tilde{l}=f\brbra{\bar{y}}-\theta\Delta,x^{1}=\bar{y},\hat{x}^{1}=\bar{y},\underline{f} = f(\bar{y})-\Delta$
\For{$t=1,\ldots$}
    \State $\brbra{\hat{x}^{t
+1},x^{t+1},\text{\textbf{Flag}},L_{t},\norm{x^{t,l_{t}}-x^{t,r_{t}}}}\leftarrow\onestep\brbra{\bar{y},\hat{x}^{t},x^{t},\tilde{l}\mid\bundleSize}$
    \State Compute $\bar{L}\brbra t$ (defined in~\eqref{eq:L_N}) based on $\norm{x^{t,l_{t}}-x^{t,r_{t}}}$ \Comment{Just for analysis}\label{enu:compute_L_bar_t}
    \If{$\flag$ is $\true$ or $\norm{x^{t+1}-\bar{y}}^{2}>\frac{2\theta\Delta}{\mu}$}
        \State \label{enu:lower_improve} \textbf{Return} $\brbra{\hat{x}^{t+1},\true}$

    \ElsIf{$f\brbra{\hat{x}^{t+1}}-\underline{f}\leq\theta\Delta$}
        \State \label{enu:prox_center_change} \textbf{Return} $\brbra{\hat{x}^{t+1},\false}$
    \EndIf
\EndFor
\end{algorithmic}
\end{algorithm}

Each outer stage first calls $\mcal{AGR}$ (Algorithm~\ref{alg:AGR}) at the
current reference point $\bar y^{s}$ at
Line~\ref{enu:improve_gap_reduction_mu_unknown}. If $\lowerflag=\false$, the
algorithm enters the validation branch. It uses the returned point $\hat{x}^{s}$
to test the current lower bound: if
$f(\hat{x}^{s})<f(\bar y^{s})-\tilde{\Delta}_{s}$, then the current QG estimate
is rejected. Otherwise, Line~\ref{enu:W-cer-gen_mu_unknown} calls $\mcal{AWG}$
(Algorithm~\ref{alg:AWG}) to perform the certificate generation at $\bar y^{s}$ and
stores the result in $\quarflag$. Line~\ref{enu:mu_unknown_quartering_test}
then reduces the estimate if $\quarflag=\false$ or if the direct test has
already rejected the estimate. In that case, Line~\ref{enu:mu_unknown_half_rho}
quarters $\hat{\mu}_{s}$, recomputes $\tilde{\Delta}_{s}$ as the smaller gap
bound obtained from the previous accepted certificate and the local gradient
bound, and returns to Line~\ref{enu:improve_gap_reduction_mu_unknown}. If the
validation branch does not reject the estimate, or if $\lowerflag=\true$ from
the start, the unified refresh step at Line~\ref{enu:refresh_mu_unknown}
accepts $\hat{x}^{s}$ as the next reference point, keeps the same QG estimate,
and updates $\tilde{\Delta}_{s+1}$ and $(\Delta_s,\mu_s)$ by the displayed case
distinction.

\begin{algorithm}
\caption{$\protect\rapex$ for unknown $\mu^*$}
\label{alg:mu_unknown}
\begin{algorithmic}[1]
\State \textbf{Input:} objective function $f$, the initial point $\bar{y}^{1}$, initial QG-modulus estimate $\hat{\mu}_{1}$, the number of cuts $\bundleSize$, shrink factor $\theta\in\brbra{\frac{1}{2},1}$, extra factor $\beta>0$
\State \textbf{Initialize:} $\tilde{\Delta}_{1}=\frac{2\norm{f^{\prime}\brbra{\bar{y}^{1}}}^{2}}{\hat{\mu}_{1}},\Delta_0=+\infty,\mu_0=1$
\For{$s=1,\ldots$}
    \State \label{enu:improve_gap_reduction_mu_unknown} $\brbra{\hat{x}^{s},\lowerflag}\leftarrow\mcal{AGR}\brbra{\bar{y}^{s},\tilde{\Delta}_{s},\hat{\mu}_{s}\mid\bundleSize,\theta}$
    \If{$\lowerflag$ is $\false$}\vspace{-1ex}
        \Statex {\footnotesize\dotfill\ \textit{No certificate by $\mcal{AGR}$: run $\mcal{AWG}$ and validate $\hat{\mu}_s$}}\vspace{-1ex}
        \State $\quarflag=\true$
        \If{$f\brbra{\hat{x}^s}\geq f(\bar{y}^{s}) - \tilde{\Delta}_{s}$}\label{enu:AWG_candidate_check_mu_unknown}
            \State \label{enu:W-cer-gen_mu_unknown} ${\quarflag}\leftarrow\mcal{AWG}\brbra{\bar{y}^{s},\tilde{\Delta}_{s},\sqrt{\frac{2\brbra{1+\beta}\tilde{\Delta}_{s}}{\hat{\mu}_{s}}}\mid \bundleSize,\beta}$
        \EndIf
        \If{$\quarflag$ is $\false$ or $f\brbra{\hat{x}^s} < f(\bar{y}^{s}) - \tilde{\Delta}_{s}$}\label{enu:mu_unknown_quartering_test}
            \State \label{enu:mu_unknown_half_rho} $\hat{\mu}_{s}=\hat{\mu}_{s}/4\ ,\ \tilde{\Delta}_{s}=\min\bcbra{\frac{\mu_{s-1}}{\hat{\mu}_{s}}\Delta_{s-1},\frac{2}{\hat{\mu}_{s}}\norm{f^{\prime}\brbra{\bar{y}^{s}}}^{2}}$ and go to Line~\ref{enu:improve_gap_reduction_mu_unknown}
        \EndIf
    \EndIf \vspace{-0.8ex}
    \Statex {\footnotesize\dotfill\ \textit{Use current best solution as next candidate; update certificate}}\vspace{-0.5ex}
    \State \label{enu:refresh_mu_unknown}\label{enu:end_mu_unknown}\label{enu:lower_bound_improve} $\bar{y}^{s+1}=\hat{x}^{s}$, $\hat{\mu}_{s+1}=\hat{\mu}_{s}$ and set
    {\footnotesize\begin{equation*}
    \tilde{\Delta}_{s+1}=\protect\begin{cases}
    f\protect\brbra{\bar{y}^{s+1}}-\protect\brbra{f\protect\brbra{\bar{y}^{s}}-\theta\tilde{\Delta}_{s}} & \lowerflag\ \text{is}\ \true,\\
    f\protect\brbra{\bar{y}^{s+1}}-\protect\brbra{f\protect\brbra{\bar{y}^{s}}-\tilde{\Delta}_{s}} & \lowerflag\ \text{is}\ \false,
    \protect\end{cases}
    \end{equation*}}
    {\footnotesize\begin{equation*}
    \Delta_{s}=\protect\begin{cases}
    \theta\tilde{\Delta}_{s} & \lowerflag\ \text{is}\ \true,\\
    \brbra{1+\beta}\tilde{\Delta}_{s} & \lowerflag\ \text{is}\ \false,
    \protect\end{cases}
    \qquad
    \mu_{s}=\hat{\mu}_{s}.
    \end{equation*}}
\EndFor
\end{algorithmic}
\end{algorithm}

The reduction step also reuses past certificates. When
Algorithm~\ref{alg:mu_unknown} enters Line~\ref{enu:mu_unknown_half_rho} at
stage $s>1$, the most recent certificate for $\bar y^{s-1}$ yields, under the
new QG estimate $\hat{\mu}_s$, the lower bound
$f(\bar y^{s-1})-(\mu_{s-1}/\hat{\mu}_s)\Delta_{s-1}$ on $f^*$ by
Proposition~\ref{prop:W_certificate_lower_bound}. Since the accepted incumbent
satisfies $f(\bar y^s)\leq f(\bar y^{s-1})$, the same lower bound is valid for
the current reference point $\bar y^s$. This bound can be sharper than the
coarse gradient bound
$f\brbra{\bar{y}^{s}}-2\norm{f^{\prime}\brbra{\bar{y}^{s}}}^{2}/\hat{\mu}_{s}$,
which is why the algorithm can reduce $\hat{\mu}_{s}$ without restarting the
lower-bound estimate from scratch.

We first bound the number of iterations required by the subroutine $\mcal{AGR}$ in Proposition~\ref{prop:AGR_convergence}.
\begin{proposition}\label{prop:AGR_convergence}
Let the weighted sequence $\omega_{t}=(t+2)(t+3)/2$,
$\alpha_{t}=4/(t+3)$ and $\theta\in(1/2,1)$ in Algorithm~\ref{alg:AGR}.
Suppose Algorithm~\ref{alg:AGR} returns at $t_{\mcal G}$-th iteration.
Then the weighted mean of Lipschitz constants  $\bar{L}\brbra{t_{\mcal G}-1}$
and the iteration number $t_{\mcal G}$ admit the following bound:
\begin{equation}\label{eq:res_AGR_converge}
\bar{L}\brbra{t_{\mcal G} - 1}\leq\frac{12\theta}{2\theta-1}\cdot\frac{\sum_{t=1}^{t_{\mcal G} - 1}\onebf_{\mcal E\brbra t}(t+3)}{\sum_{t=1}^{t_{\mcal G}-1}\brbra{\onebf_{\mcal E\brbra t}(t+3)/\tilde{L}_{t}}}\,\,  \text{and}\,\,\  t_{\mcal G}\leq\min\bcbra{t:t>\left\lceil \sqrt{\frac{6\theta+4\theta\bar{L}\brbra t/\mu}{2\theta-1}}\right\rceil }\ ,
\end{equation}
where $\tilde{L}_t$ is local empirical Lipschitz constant defined in~\eqref{eq:defn_tilde_L_t} for $\onestep$ and $\mcal E(t)$ is defined in~\eqref{eq:event_E_t}.
\end{proposition}
\noindent\emph{Proof deferred.} The detailed proof is collected in Subsection~\ref{subsec:unknown-mu-detailed-proofs}.

Thus, Proposition~\ref{prop:AGR_convergence} shows that the iteration bound for $\mcal{AGR}$ depends on the effective condition number $\sqrt{\bar{L}\brbra{t}/\mu}$ in~\eqref{eq:res_AGR_converge}.
Now, we are ready to give the formal statement of convergence result
of Algorithm~\ref{alg:mu_unknown}.
\begin{thm}\label{thm:mu_unknown}Suppose $f$ satisfies $\mu^{*}$-QG.
    In Algorithm~\ref{alg:mu_unknown}, assume $\hat{\mu}_{1}\geq\mu^{*}$. Set $\omega_t=(t+2)(t+3)/2,\alpha_t=4/(t+3)$ in subroutine $\mcal{AWG}$ and $\mcal{AGR}$. To find a $\brbra{2\vep/\mu_{S},\vep}$-$\Wcer$
    for point $\bar{y}^{S}$ in function $f$, Algorithm~\ref{alg:mu_unknown}
    requires a total number of first-order oracle calls bounded by:
    \begin{equation}
    O(1)\cdot \bundleSize \max\bcbra{\sqrt{\frac{{L}_{\text{avg}}}{\mu^{*}}},1}\brbra{\log_{4}\frac{4\hat{\mu}_{1}}{\mu^{*}}+\log\frac{\norm{f^{\prime}\brbra{\bar{y}^{1}}}^{2}}{\vep^{2}}}\ ,
    \end{equation}
{where ${L}_{\text{avg}}$ is a uniform upper bound of weighted harmonic mean of empirical Lipschitz constants of $\mcal{AWG}$ (defined in~\eqref{eq:weighted_harmonic_mean_AWG}) and  $\mcal{AGR}$ (defined in~\eqref{eq:res_AGR_converge}) before obtaining the final certificate.}
Furthermore, if $f$ is $\brbra{k,L}-\pws$ and $\bundleSize \geq k$, then we have ${L}_{\text{avg}}\leq O(1)L$.
    \end{thm}
\noindent\emph{Proof deferred.} The detailed proof is collected in Subsection~\ref{subsec:unknown-mu-detailed-proofs}.

The complexity in Theorem~\ref{thm:mu_unknown} is established
for an $\brbra{2\vep/\mu_{s},\vep}$\textendash $\Wcer$. The $\Wcer$
is verifiable during the algorithmic process, at which stage $s$ the $\Wcer$
is guaranteed to be generated correctly. By applying the same proof
technique, we obtain the complexity result for the optimality gap,
stated in the following Theorem.
\begin{thm}
\label{thm:unknown_QG_opt_Gap}Suppose $f$ satisfies $\mu^{*}$-QG. In Algorithm~\ref{alg:mu_unknown},
assume $\hat{\mu}_{1}\geq\mu^{*}$. Set $\omega_t = (t+2)(t+3)/2,\alpha_t = 4/(t+3)$ in subroutine $\mcal{AWG}$ and $\mcal{AGR}$.
To find a solution $\bar{y}^{S+1}$ such that $f\brbra{\bar{y}^{S+1}}-f^{*}\leq\vep$,
Algorithm~\ref{alg:mu_unknown} requires a total number of gradient
oracle calls bounded by:
\begin{equation}
O(1)\cdot \bundleSize\max\bcbra{\sqrt{\frac{L_{\text{avg}}}{\mu^{*}}},1}\brbra{\log_{4}\frac{4\hat{\mu}_{1}}{\mu^{*}}+\log\frac{\norm{f^{\prime}\brbra{\bar{y}^{1}}}^{2}}{\mu^{*}\vep}}\ ,
\end{equation}
where $L_{\text{avg}}$ is a uniform upper bound of weighted harmonic mean of empirical Lipschitz constants of $\mcal{AWG}$ (defined in~\eqref{eq:weighted_harmonic_mean_AWG}) and  $\mcal{AGR}$ (defined in~\eqref{eq:res_AGR_converge}) before obtaining the $\vep$-optimal solution. Furthermore, if $f$ is $(k,L)-\pws$ and $B\geq k$, then we have $L_{\text{avg}}\leq O(1)L$.
\end{thm}
\noindent\emph{Proof deferred.} The detailed proof is collected in Subsection~\ref{subsec:unknown-mu-detailed-proofs}.

We make a few remarks on the complexity bounds in
Theorems~\ref{thm:mu_unknown} and~\ref{thm:unknown_QG_opt_Gap}. First,
both theorems assume $\hat{\mu}_{1}\geq\mu^{*}$. Such an initial upper
estimate can be obtained by choosing two points $x_1$ and $x_2$ with
$f(x_2)<f(x_1)$ and setting
$\hat{\mu}_1 = 2\norm{f^\prime(x_1)}^2/(f(x_1)-f(x_2))$, since
$f(x_1)-f(x_2)\leq f(x_1)-f^*\leq 2\norm{f^\prime(x_1)}^2/\mu^*$.
Second, if $\bundleSize\geq k$, then $L_{\text{avg}}\leq O(1)L$, and
the bounds inherit the accelerated dependence
$\max\bcbra{\sqrt{L/\mu^{*}},1}$. For the optimality-gap guarantee in
Theorem~\ref{thm:unknown_QG_opt_Gap}, the leading term is
\[
O(1)\cdot \bundleSize \max\bcbra{\sqrt{\frac{L}{\mu^{*}}},1}
\log\brbra{\frac{\norm{f^{\prime}\brbra{\bar{y}^{1}}}^{2}}{\mu^{*}\vep}}.
\]
The certificate guarantee in Theorem~\ref{thm:mu_unknown} has the same
condition-number dependence, with its logarithmic factor given in the
theorem statement. Consequently, up to additional logarithmic terms, these
bounds match the rates obtained by the known-$\mu^{*}$ restarted scheme in
Theorem~\ref{thm:mu_known} and the idealized known-$f^{*}$ scheme in
Theorem~\ref{thm:complexity_idealized_apex}, while requiring neither
$\mu^{*}$ nor $f^{*}$ as input. Finally, Algorithm~\ref{alg:mu_unknown} is
anytime: it does not require the target accuracy $\vep$ in advance.

\subsection{Detailed proofs\label{subsec:unknown-mu-detailed-proofs}}

We close this section by collecting the technical proofs for the unknown-QG
analysis in the order used by the argument. Proposition~\ref{prop:AGR_convergence}
first bounds the number of iterations required by $\mcal{AGR}$. We then prove
Proposition~\ref{prop:generate_W_certificate} and Lemma~\ref{lem:lower_bound_valid},
which establish certificate generation and the validity of the maintained lower
bound. Proposition~\ref{prop:guess_wrong_mu_s} bounds the number of
QG-estimate reductions. These ingredients yield the convergence guarantees in
Theorems~\ref{thm:mu_unknown} and~\ref{thm:unknown_QG_opt_Gap}.

\begin{proof}[\underline{Proof of Proposition~\ref{prop:AGR_convergence}}]
Algorithm~\ref{alg:AGR} can be viewed as a specific instance of APEX, where the parameters are set as $\tilde{l} = f\brbra{\bar{y}} - \theta \Delta$, $\alpha_t = 4/(t+3)$, and $\omega_t = (t+2)(t+3)/2$.
From Proposition~\ref{prop:converge}, we have
\[
\frac{1}{2}(t+2)(t+3)\brbra{f\brbra{\hat{x}^{t+1}}-\brbra{f\brbra{\bar{y}}-\theta\Delta}}-3\theta\Delta\leq\bar{L}\brbra t\norm{x^{t+1}-\bar{y}}^{2}\ .
\]
If Algorithm~\ref{alg:AGR} has not returned by iteration $t$, then
$f(\hat{x}^{t+1})>\theta\Delta+\underline f=f(\bar y)-(1-\theta)\Delta$ and
$\norm{x^{t+1}-\bar y}^{2}\leq2\theta\Delta/\mu$. Therefore,
\[
\frac{1}{2}(t+2)(t+3)(2\theta-1)\Delta - 3\theta\Delta \leq \bar{L}(t)\norm{x^{t+1}-\bar{y}}^2.
\]
It follows that, once
$t>\sqrt{\brbra{4\theta\,\bar{L}(t)/\mu+6\theta}/(2\theta-1)}$, the inequality
$\norm{x^{t+1}-\bar y}^{2}\leq2\theta\Delta/\mu$ cannot hold. Hence
Algorithm~\ref{alg:AGR} must have already returned or must return at iteration
$t$, establishing the second claimed result in~\eqref{eq:res_AGR_converge}. By
a similar procedure for~\eqref{eq:upper_bound_mu_L_bar}, we have the first
result in~\eqref{eq:res_AGR_converge}.
\ifdefined\isarxiv
\end{proof}
\else
\qedsymbol\end{proof}
\fi

\begin{proposition}
\label{prop:generate_W_certificate}
In Algorithm~\ref{alg:mu_unknown}, whenever Line~\ref{enu:refresh_mu_unknown}
is reached, a $\Wcer$ is available in either of the following two cases.
\begin{enumerate}
\item If $\lowerflag=\false$, then the call to $\mcal{AWG}$ at
Line~\ref{enu:W-cer-gen_mu_unknown} returns $\quarflag=\true$, and there
exists a point set $\mcal P_s$ that forms a $\Wcer$ for $\bar{y}^{s}$ with
parameters
\[
\brbra{\sqrt{2(1+\beta)\tilde{\Delta}_{s}/\hat{\mu}_{s}},
\sqrt{\hat{\mu}_{s}(1+\beta)\tilde{\Delta}_{s}/2}}.
\]
\item If $\lowerflag=\true$, then there exists a point set $\mcal P_s$ that
forms a $\Wcer$ for $\bar{y}^{s}$ with parameters
\[
\brbra{\sqrt{2\Delta_{s}/\mu_{s}},\sqrt{\mu_{s}\Delta_{s}/2}},
\]
where $\Delta_s=\theta\tilde{\Delta}_s$ and
$\mu_s=\hat{\mu}_s$ are the values assigned at Line~\ref{enu:refresh_mu_unknown}.
\end{enumerate}
\end{proposition}

\begin{proof}
We verify the two cases in the proposition. First, if $\lowerflag=\false$ and
Line~\ref{enu:refresh_mu_unknown} is reached, then
Line~\ref{enu:mu_unknown_quartering_test} did not trigger. Hence the call to
$\mcal{AWG}$ at Line~\ref{enu:W-cer-gen_mu_unknown} returned
$\quarflag=\true$. This call uses $\Delta=\tilde{\Delta}_s$ and
$\iota_{\max}=\sqrt{2(1+\beta)\tilde{\Delta}_s/\hat{\mu}_s}$, so
Algorithm~\ref{alg:AWG} has reached its certificate-return line. The certificate part of
Proposition~\ref{prop:AWG_W_cer_gen} therefore gives a point set $\mcal P_s$
that forms a
$\brbra{\sqrt{2(1+\beta)\tilde{\Delta}_{s}/\hat{\mu}_{s}},
\sqrt{\hat{\mu}_{s}(1+\beta)\tilde{\Delta}_{s}/2}}$-$\Wcer$
for $\bar y^s$.

Second, if $\lowerflag=\true$, then the lower-bound trigger in $\mcal{AGR}$ uses the
level $f(\bar y^s)-\theta\tilde{\Delta}_s$ and the radius
$\sqrt{2\theta\tilde{\Delta}_s/\hat{\mu}_s}$. The same
certificate-generation argument as in Proposition~\ref{prop:AWG_W_cer_gen}
then gives a point set $\mcal P_s$ that forms a
$\brbra{\sqrt{2\theta\tilde{\Delta}_{s}/\hat{\mu}_{s}},
\sqrt{\hat{\mu}_{s}\theta\tilde{\Delta}_{s}/2}}$-$\Wcer$
for $\bar y^s$. Since Line~\ref{enu:refresh_mu_unknown} sets
$\Delta_s=\theta\tilde{\Delta}_s$ and $\mu_s=\hat{\mu}_s$, this is exactly the
certificate stated in the proposition.
\ifdefined\isarxiv
    \end{proof}
\else
    \qedsymbol\end{proof}
\fi

\begin{lem}
\label{lem:lower_bound_valid}
At any stage $s$ of Algorithm~\ref{alg:mu_unknown}, suppose that $f$ satisfies
the $\hat{\mu}_{s}$-QG condition for the current estimate $\hat{\mu}_{s}$. Then
$\tilde{\Delta}_{s}$ satisfies the following inequality:
\begin{equation}\label{eq:lower_bound_effective}
    f\brbra{\bar y^s}-\tilde{\Delta}_s\leq f^*.
\end{equation}
\end{lem}

\begin{proof}
First consider the initialization. By convexity, for any
$x^*\in X^*$,
\[
f\brbra{\bar y^1}-f^*
\leq
\inner{f'\brbra{\bar y^1}}{\bar y^1-x^*}
\leq
\norm{f'\brbra{\bar y^1}}\norm{\bar y^1-x^*}.
\]
Taking the infimum over $x^*\in X^*$ gives
$f\brbra{\bar y^1}-f^*
\leq\norm{f'\brbra{\bar y^1}}\dist\brbra{\bar y^1,X^*}$.
Since $f$ satisfies the $\hat{\mu}_1$-QG condition,
\[
\dist\brbra{\bar y^1,X^*}
\leq
\sqrt{\frac{2\brbra{f\brbra{\bar y^1}-f^*}}{\hat{\mu}_1}}.
\]
Thus either $f\brbra{\bar y^1}=f^*$, or
\[
f\brbra{\bar y^1}-f^*
\leq
\frac{2\norm{f'\brbra{\bar y^1}}^2}{\hat{\mu}_1}
=\tilde{\Delta}_1.
\]

Assume that~\eqref{eq:lower_bound_effective} holds at the beginning of stage
$s$. If the algorithm reaches Line~\ref{enu:refresh_mu_unknown} with
$\lowerflag=\false$, then
\[
\tilde{\Delta}_{s+1}
=
f\brbra{\bar y^{s+1}}
-\brbra{f\brbra{\bar y^s}-\tilde{\Delta}_s}.
\]
The induction hypothesis gives
$f\brbra{\bar y^s}-\tilde{\Delta}_s\leq f^*$, and therefore
$f\brbra{\bar y^{s+1}}-f^*\leq\tilde{\Delta}_{s+1}$.

If the algorithm reaches Line~\ref{enu:refresh_mu_unknown} with
$\lowerflag=\true$, then Proposition~\ref{prop:generate_W_certificate} provides
a $\Wcer$ for $\bar y^s$ with parameters
\[
\brbra{\sqrt{2\Delta_s/\mu_s},\sqrt{\mu_s\Delta_s/2}},
\]
where
$\Delta_s=\theta\tilde{\Delta}_s$ and $\mu_s=\hat{\mu}_s$. Applying
Proposition~\ref{prop:W_certificate_lower_bound} yields
\[
f\brbra{\bar y^s}-f^*
\leq
\max\bcbra{
\Delta_s,\frac{2}{\hat{\mu}_s}\cdot\frac{\mu_s\Delta_s}{2}}
=\Delta_s.
\]
Thus $f\brbra{\bar y^s}-\theta\tilde{\Delta}_s\leq f^*$. Since
Line~\ref{enu:refresh_mu_unknown} sets
\[
\tilde{\Delta}_{s+1}
=
f\brbra{\bar y^{s+1}}
-\brbra{f\brbra{\bar y^s}-\theta\tilde{\Delta}_s},
\]
we again obtain
$f\brbra{\bar y^{s+1}}-f^*\leq\tilde{\Delta}_{s+1}$.

It remains to check Line~\ref{enu:mu_unknown_half_rho}. The reference point
$\bar y^s$ is unchanged, while the estimate $\hat{\mu}_s$ is quartered. Write
$\hat{\mu}_s$ for the new value after this update. The gradient candidate is
valid by the same convexity--QG argument used at initialization:
\[
f\brbra{\bar y^s}-f^*
\leq
\frac{2\norm{f'\brbra{\bar y^s}}^2}{\hat{\mu}_s}.
\]
For $s\geq2$, we also use the certificate generated at the previous accepted
stage. By Proposition~\ref{prop:generate_W_certificate}, there is an
$\brbra{\iota_{s-1},\nu_{s-1}}$-$\Wcer$ for $\bar y^{s-1}$ with
\[
\iota_{s-1}=\sqrt{\frac{2\Delta_{s-1}}{\mu_{s-1}}},\qquad
\nu_{s-1}=\sqrt{\frac{\mu_{s-1}\Delta_{s-1}}{2}}.
\]
Applying Proposition~\ref{prop:W_certificate_lower_bound} to this certificate
with the new QG estimate $\hat{\mu}_s$ gives
\[
f\brbra{\bar y^{s-1}}-f^*
\leq
\max\bcbra{\iota_{s-1}\nu_{s-1},
\frac{2\nu_{s-1}^2}{\hat{\mu}_s}}
=
\frac{\mu_{s-1}}{\hat{\mu}_s}\Delta_{s-1},
\]
where the equality uses $\iota_{s-1}\nu_{s-1}=\Delta_{s-1}$,
$2\nu_{s-1}^2/\hat{\mu}_s=(\mu_{s-1}/\hat{\mu}_s)\Delta_{s-1}$,
and $\hat{\mu}_s\leq\mu_{s-1}$ after quartering. Equivalently,
\[
f\brbra{\bar y^{s-1}}-\frac{\mu_{s-1}}{\hat{\mu}_s}\Delta_{s-1}\leq f^*.
\]
Since the incumbent-selection rule in Algorithm~\ref{alg:One-step} gives
$f\brbra{\bar y^s}\leq f\brbra{\bar y^{s-1}}$, this lower bound remains valid
for the current reference point:
\[
f\brbra{\bar y^s}-\frac{\mu_{s-1}}{\hat{\mu}_s}\Delta_{s-1}
\leq
f\brbra{\bar y^{s-1}}-\frac{\mu_{s-1}}{\hat{\mu}_s}\Delta_{s-1}
\leq f^*.
\]
Thus
\[
f\brbra{\bar y^s}-f^*
\leq
\frac{\mu_{s-1}}{\hat{\mu}_s}\Delta_{s-1}.
\]
Therefore both candidates in the minimum defining the new
$\tilde{\Delta}_s$ are valid upper bounds on the current optimality gap
(with $\Delta_0=+\infty$ making the first candidate inactive when $s=1$).
Their minimum is also a valid upper bound, so the invariant is preserved after
quartering.

All possible updates preserve
$f\brbra{\bar y^s}-f^*\leq\tilde{\Delta}_s$, which completes the induction.
\ifdefined\isarxiv
    \end{proof}
\else
    \qedsymbol\end{proof}
\fi

\begin{proposition}
\label{prop:guess_wrong_mu_s}
Suppose that $f$ is $\mu^*$-QG and $\hat{\mu}_1\geq\mu^*$. Whenever
Algorithm~\ref{alg:mu_unknown} reaches Line~\ref{enu:mu_unknown_half_rho} at
stage $s$, its pre-update estimate satisfies $\hat{\mu}_s>\mu^*$. Hence every
updated QG estimate is at least $\mu^*/4$, and
Line~\ref{enu:mu_unknown_half_rho} is reached at most
$\left\lceil\log_4\brbra{\hat{\mu}_1/\mu^*}\right\rceil$ times.
\end{proposition}

\begin{proof}
We prove the first claim by contraposition. Fix stage $s$ and suppose
$\hat{\mu}_s\leq\mu^*$ before Line~\ref{enu:mu_unknown_half_rho}. Then $f$ is
also $\hat{\mu}_s$-QG, so Lemma~\ref{lem:lower_bound_valid} gives
$f\brbra{\bar y^s}-\tilde{\Delta}_s\leq f^*$. Since $\hat{x}^s\in X$, we have
$f\brbra{\hat{x}^s}\geq f^*\geq f\brbra{\bar y^s}-\tilde{\Delta}_s$, and the
direct rejection test in Line~\ref{enu:mu_unknown_quartering_test} fails. Thus
Algorithm~\ref{alg:mu_unknown} calls $\mcal{AWG}$ at
Line~\ref{enu:W-cer-gen_mu_unknown}. By Proposition~\ref{prop:AWG_W_cer_gen}
with $\Delta=\tilde{\Delta}_s$ and
$\iota_{\max}=\sqrt{2(1+\beta)\tilde{\Delta}_s/\hat{\mu}_s}$, this call
returns through the certificate line of Algorithm~\ref{alg:AWG}; hence
$\quarflag=\true$. Both disjuncts in Line~\ref{enu:mu_unknown_quartering_test} are
therefore false, so the line cannot be reached. This proves the contrapositive.

Each time Line~\ref{enu:mu_unknown_half_rho} is reached, the pre-update estimate
is larger than $\mu^*$ and is divided by $4$; the updated estimate is therefore
larger than $\mu^*/4$. Starting from $\hat{\mu}_1\geq\mu^*$, this can occur at
most $\left\lceil\log_4\brbra{\hat{\mu}_1/\mu^*}\right\rceil$ times.

\ifdefined\isarxiv
    \end{proof}
\else
    \qedsymbol\end{proof}
\fi

\begin{proof}[\underline{Proof of Theorem~\ref{thm:mu_unknown}}]
    We first introduce notation for the QG estimate adjustment and derive
a sufficient condition for obtaining
the target certificate. We then bound the cost of a single call to
$\mcal{AWG}$ or $\mcal{AGR}$ and decompose the total complexity
into four terms: $S_{1}\brbra{\vep}$ for uncertified QG estimates,
$S_{2}\brbra{\vep}$ for the initial phases, $S_{3}\brbra{\vep}$
for the terminal phase, and $S_{4}\brbra{\vep}$ for certified intermediate
phases.

\textbf{Notation clarification.} To simplify the discussion, let $q$ index
the QG-modulus estimates $\hat{\mu}_{1}/4^q$ considered by the algorithm, where
$0\leq q\leq\log_{4}\frac{4\hat{\mu}_{1}}{\mu^{*}}$.
Let $\tilde{\Delta}_{(1)}^{(q)}$ and $\tilde{\Delta}_{(2)}^{(q)}$
denote the values of $\tilde{\Delta}_{s}$ immediately before and
after the $q$-th quartering of the QG estimates. Similar notations
are adopted for $\hat{\mu}_{(1)}^{(q)}$ and $\hat{\mu}_{(2)}^{(q)}$.

\textbf{Sufficient condition for target certificate. }By Proposition~\ref{prop:generate_W_certificate},
whenever Algorithm~\ref{alg:mu_unknown} reaches
Line~\ref{enu:refresh_mu_unknown}, there exists a
$\brbra{\sqrt{2\delta_s/\mu_s},\sqrt{\mu_s\delta_s/2}}$-$\Wcer$
for $\bar{y}^{s}$ with $\delta_s\leq(1+\beta)\tilde{\Delta}_{s}$.
Hence, to obtain an $\brbra{2\vep/\mu_{s},\vep}$-$\Wcer$
for $\bar{y}^{s}$, it suffices to enforce $\mu_{s}\delta_s\leq2\vep^{2}$.
Since $\delta_s\leq\brbra{1+\beta}\tilde{\Delta}_{s}$, a sufficient
condition for target certificate is
\begin{equation}
\tilde{\Delta}_{s}\leq\frac{2\vep^{2}}{\hat{\mu}_{s}\brbra{1+\beta}}=\frac{2\vep^{2}}{\mu_{s}\brbra{1+\beta}}\ .\label{eq:suff_eps}
\end{equation}

\textbf{Cost for single call to $\mcal{AWG}$ or $\mcal{AGR}$.} By
Proposition~\ref{prop:AWG_W_cer_gen} with $\iota_{\max}=\sqrt{\brbra{2\brbra{1+\beta}\tilde{\Delta}_{s}}/\hat{\mu}_{s}}$,
together with result in Proposition~\ref{prop:AGR_convergence}, we
have no matter $\mcal{AWG}$ or $\mcal{AGR}$, it will terminate at most $O(1)\bundleSize \sqrt{\bar{L}\brbra t/\hat{\mu}_{s}}$ gradient evaluations. Since $\bar{L}(t)$ is upper bounded by the weighted harmonic mean of empirical Lipschitz constants of $\mcal{AWG}$ and $\mcal{AGR}$ and $L_{\text{avg}}$ is the uniform upper bound of it in the whole run, then we have for $\mcal{AGR}$ or $\mcal{AWG}$, it will terminate at most $O(1)\bundleSize \sqrt{L_{\text{avg}}/\hat{\mu}_{s}}$ gradient evaluations.

\textbf{Cost for uncertified QG estimate $S_1(\vep)$.} For each QG estimation $\hat{\mu}_1/4^q,0\leq q\leq \log_{4}\frac{4\hat{\mu}_{1}}{\mu^{*}}$, $\mcal{AWG}$ uses at most $O(1)\bundleSize\sqrt{L_{\text{avg}}4^q/\hat{\mu}_1}$ gradient evaluations before either generating a certificate or triggering a QG adjustment. Hence, the total cost over all uncertified QG estimates is bounded by
\begin{equation}\label{eq:S_1_vep_upper_bound}
S_{1}\brbra{\vep}\leq\bundleSize\sum_{q=0}^{\log_{4}\frac{4\hat{\mu}_{1}}{\mu^{*}}}\sqrt{\frac{L_{\text{avg}}4^{q}}{\hat{\mu}_{1}}}\leq O(1)\bundleSize\sqrt{\frac{L_{\text{avg}}}{\mu^{*}}}\ .
\end{equation}

\textbf{Cost from initialization to the first QG adjustment $\brbra{S_{2}\brbra{\vep}}$.
}Suppose the algorithm enters Line~\ref{enu:mu_unknown_half_rho}
for the first time after reaching the gap value $\tilde{\Delta}_{(1)}^{(1)}$.
Up to that moment, it always works with the initial QG estimate
$\hat{\mu}_{1}$. Since $\tilde{\Delta}_{1}\leq2\norm{f^{\prime}\brbra{\bar{y}^{1}}}^{2}/\hat{\mu}_{1}$
and $\tilde{\Delta}^{(1)}_{(1)}\geq\frac{2\vep^{2}}{\hat{\mu}_{1}\brbra{1+\beta}}$
before sufficient condition holds, an upper bound of $S_{2}\brbra{\vep}$
is given by
\begin{equation}\label{eq:S_2_vep_upper_bound}
S_{2}\brbra{\vep}\leq O(1)\cdot\bundleSize\sqrt{\frac{L_{\text{avg}}}{\hat{\mu}_{1}}}\log_{\frac{1}{\theta}}\brbra{\frac{\norm{f^{\prime}\brbra{\bar{y}^{1}}}^{2}}{\vep^{2}}}\ .
\end{equation}
The logarithmic factor counts the number of successful stages with the initial
QG estimate: each such stage reduces $\tilde{\Delta}_{s}$ by the factor
$\theta$, until the sufficient condition~\eqref{eq:suff_eps} is met or the
first QG adjustment occurs. This count also bounds the corresponding calls to
$\mcal{AWG}$ and $\mcal{AGR}$ in this phase.

\textbf{Cost from the last QG estimate to the sufficient condition
$\brbra{S_{3}\brbra{\vep}}$.} Let $\hat{\mu}_{1}/4^{q^{*}}$ be the
last QG estimate before the sufficient condition~\eqref{eq:suff_eps}
is met. After this update, the QG estimate is fixed, and each successful
visit to Line~\ref{enu:refresh_mu_unknown} decreases $\tilde{\Delta}_{s}$ by at least the factor
$\theta$. Since $q^{*}\leq\log_{4}\frac{4\hat{\mu}_{1}}{\mu^{*}}$,
we obtain
\begin{equation}\label{eq:S_3_vep_upper_bound}
S_{3}\brbra{\vep}\leq O(1)B\sqrt{\frac{L_{\text{avg}}}{\mu^{*}}}\brbra{1+\log_{\frac{1}{\theta}}\frac{(1+\beta)\tilde{\Delta}_{(2)}^{(q^{*})}}{\vep^{2}/\hat{\mu}_{1}}\cdot4^{q^{*}}}=O(1)B\sqrt{\frac{L_{\text{avg}}}{\mu^{*}}}\brbra{1+\log_{\frac{1}{\theta}}\frac{\norm{f^{\prime}\brbra{\bar{y}^{1}}}^{2}}{\vep^{2}}}.
\end{equation}

\textbf{Cost for certified QG estimate $\brbra{S_{4}\brbra{\vep}}$.}
Let $q^{*}$ be the realized index of the last QG adjustment before the
sufficient condition~\eqref{eq:suff_eps} is met. By
Proposition~\ref{prop:guess_wrong_mu_s}, $q^{*}\leq
\log_{4}\brbra{4\hat{\mu}_{1}/\mu^{*}}$. Only the intermediate estimates
$\hat{\mu}_{1}/4^q$, $1\leq q\leq q^{*}-1$, contribute to $S_4(\vep)$; the
last estimate $\hat{\mu}_{1}/4^{q^*}$ is counted in $S_3(\vep)$. For each such
$q$, subroutine $\mcal{AGR}$ and $\mcal{AWG}$ are called at most
$2\log_{\frac{1}{\theta}}\brbra{\tilde{\Delta}_{(2)}^{(q)}/\tilde{\Delta}_{(1)}^{(q+1)}}$
times. Hence, with empty sums interpreted as zero, we have
\begin{equation}\label{eq:S_4_vep_raw}
\begin{split}S_{4}\brbra{\vep} & \leq O(1)\sum_{q=1}^{q^*-1}\sqrt{\frac{L_{\text{avg}}}{\hat{\mu}_{1}}\cdot4^{q}}\log_{1/\theta}\brbra{\frac{\tilde{\Delta}_{(2)}^{(q)}}{\tilde{\Delta}_{(1)}^{(q+1)}}}\leq O(1)\sqrt{\frac{L_{\text{avg}}}{\mu^{*}}}\cdot\sum_{q=1}^{q^*-1}\log_{1/\theta}\brbra{\frac{\tilde{\Delta}_{(2)}^{(q)}}{\tilde{\Delta}_{(1)}^{(q+1)}}}\\
 & =O(1)\sqrt{\frac{L_{\text{avg}}}{\mu^{*}}}\cdot\brbra{\log_{1/\theta}\frac{\tilde{\Delta}_{(2)}^{(1)}}{\tilde{\Delta}_{(1)}^{(q^*)}}+\log_{1/\theta}\brbra{\prod_{q=1}^{q^*-2}\frac{\tilde{\Delta}_{(2)}^{(q+1)}}{\tilde{\Delta}_{(1)}^{(q+1)}}}}
\end{split}
\end{equation}
By the update rule, we have
\begin{equation}
\tilde{\Delta}_{(2)}^{(q+1)}\leq\frac{\mu_{(1)}^{(q+1)}}{\mu_{(2)}^{(q+1)}}\brbra{1+\beta}\tilde{\Delta}_{(1)}^{(q+1)}\ ,\ q=1,\ldots,q^*-2
\end{equation}
and
\begin{equation}
\mu_{(1)}^{(q+1)}=\mu_{(2)}^{(q)}\ ,\ q=1,\ldots,q^*-2\ .
\end{equation}
Hence, we have
\begin{equation}\label{eq:upper_bound_product}
\prod_{q=1}^{q^*-2}\frac{\tilde{\Delta}_{(2)}^{(q+1)}}{\tilde{\Delta}_{(1)}^{(q+1)}}\leq\prod_{q=1}^{q^*-2}\frac{(1+\beta)\mu_{(1)}^{(q+1)}}{\mu_{(2)}^{(q+1)}}\leq\brbra{1+\beta}^{q^*}\cdot\frac{\hat{\mu}_{1}}{\mu^{*}}\ .
\end{equation}
By similar proof procedure for $S_{1}\brbra{\vep}$, we have an upper
bound of $\tilde{\Delta}_{(2)}^{(1)}/\tilde{\Delta}_{(1)}^{(q^*)}$:
\begin{equation}\label{eq:factor_upper_bound}
    \frac{\tilde{\Delta}_{(2)}^{(1)}}{\tilde{\Delta}_{(1)}^{(q^*)}}\leq O(1)\brbra{\frac{\norm{f^{\prime}\brbra{\bar{y}^{1}}}^{2}}{\vep^{2}}}\ .
\end{equation}
Combining~\eqref{eq:S_4_vep_raw},~\eqref{eq:upper_bound_product},~\eqref{eq:factor_upper_bound} and $q^* \leq \log_4 \frac{4\hat{\mu}_1}{\mu^*}$ gives
\begin{equation}\label{eq:S_4_vep_upper_bound}
\begin{split}S_{4}\brbra{\vep} & \leq O(1)\sqrt{\frac{L_{\text{avg}}}{\mu^{*}}}\cdot\brbra{\log_{\frac{1}{\theta}}\frac{\norm{f^{\prime}\brbra{\bar{y}^{1}}}^{2}}{\vep^{2}}+q^*+\log_{\frac{1}{\theta}}\brbra{\frac{\hat{\mu}_{1}}{\mu^{*}}}}\\
 & \leq O(1)\sqrt{\frac{L_{\text{avg}}}{\mu^{*}}}\cdot\brbra{\log_{\frac{1}{\theta}}\frac{\norm{f^{\prime}\brbra{\bar{y}^{1}}}^{2}}{\vep^{2}}+\log_{4}\frac{4\hat{\mu}_{1}}{\mu^{*}}}
\end{split}
\ .
\end{equation}

Summing the upper bound of $S_1(\vep), S_2(\vep), S_3(\vep), S_4(\vep)$ in~\eqref{eq:S_1_vep_upper_bound},~\eqref{eq:S_2_vep_upper_bound},~\eqref{eq:S_3_vep_upper_bound},~\eqref{eq:S_4_vep_upper_bound} gives the final result.

\ifdefined\isarxiv
    \end{proof}
\else
    \qedsymbol\end{proof}
\fi

\begin{proof}[\underline{Proof of Theorem~\ref{thm:unknown_QG_opt_Gap}}]
By Proposition~\ref{prop:generate_W_certificate}, at the terminal stage $S$,
Algorithm~\ref{alg:mu_unknown} generates a point set $\mcal P_{S}$
as an $\brbra{\iota_S,\nu_S}$-$\Wcer$ for $f$ at the point $\bar{y}^{S}$, where
\[
\iota_S=\sqrt{2\Delta_{S}/\mu_{S}},
\qquad
\nu_S=\sqrt{\mu_{S}\Delta_{S}/2}.
\]
By Proposition~\ref{prop:W_certificate_lower_bound},
the $\Wcer$ together with $\mu^{*}$-QG implies that $f\brbra{\bar{y}^{S}}-f^{*}\leq\max\bcbra{1,\mu_{S}/\mu^{*}}\Delta_{S}$.
To ensure $f\brbra{\bar{y}^{S}}-f^{*}\leq\vep$, we need $\mu_{S}\Delta_{S}\leq\min\bcbra{\mu_{S},\mu^{*}}\vep$.
Notice that Theorem~\ref{thm:mu_unknown} is to bound
the iteration number for condition $\mu_{S}\Delta_{S}\leq 2\vep^{2}$.
Hence, replacing $2\vep^{2}$ in the proof procedure of Theorem~\ref{thm:mu_unknown}
with ${\min\bcbra{\mu_{S},\mu^{*}}\vep}$ and noting that ${\mu}_{S}\geq\frac{1}{4}\mu^{*}$,
we have the complexity result.
\ifdefined\isarxiv
    \end{proof}
\else
    \qedsymbol\end{proof}
\fi

\section{\protect\label{sec:Numerical-Study}Numerical Study}

In this section, we evaluate our algorithms
on two tasks: minimizing a finite maximum of convex quadratic functions
(Section~\ref{subsec:max_quadratic}) and solving two-stage
stochastic linear programming (Section~\ref{subsec:two_sto_lp}).
All experiments were conducted on a Mac mini M2 Pro with 32 GB of
RAM.

\subsection{\protect\label{subsec:max_quadratic}Finite max of convex quadratic
functions ({\maxquad})}

We focus on the following optimization problem:
\begin{equation}
\min_{x\in\mbb R^{d}}\max_{i\in[k]}\ \ \frac{1}{2}x^{\top}A_{i}x+b_{i}^{\top}x+c_{i}\ ,
\tag{\maxquad}
\end{equation}
where each matrix $A_{i}\in\mbb R^{d\times d}$ is a randomly generated
positive definite matrix, and $\bcbra{b_{i},c_{i}}_{i=1}^{k}$ are
drawn independently from a normal distribution. 
The number of pieces can be determined manually by choosing $k$. Furthermore, we can control the QG modulus (strong convexity in $\maxquad$) by setting the eigenvalues of matrix $A_i$. 
For small-scale instances ($d=500$), each matrix $A_i$ is constructed with eigenvalues following an arithmetic sequence from $\mu$ to $L$ and distinct, randomly generated eigenvectors. For large-scale problems ($d=10{,}000$ or $d=20{,}000$), to save memory, all $A_i$ share a single set of eigenvectors, while the eigenvalues are generated independently for each matrix.

\paragraph{The performance of {\rapex} on {\maxquad}. } 
\ \ 

\textbf{Experiment setting.} We first evaluate the performance of {\rapex} on small-scale {\maxquad} with varying the number of pieces $k$, smoothness $L$ and initial QG estimates $\hat{\mu}_1$, as reported in Table~\ref{tab:com_rapex_max_quad}. In all experiments, we set the dimension ($d$) to $300$, number of cuts ($\bundleSize$) to $50$, the optimal QG modulus ($\mu^*$) to $1.0$, and terminate {\rapex} once {\rapex} improves the upper bound (Algorithm~\ref{alg:mu_unknown} enters Line~\ref{enu:end_mu_unknown}) with $\Delta_s\leq 10^{-6}$. Ipopt~\cite{wachter2006implementation} is used to compute a reference optimal objective value $(f^*)$. 

\textbf{Key observations.} 
Five observations can be drawn from Table~\ref{tab:com_rapex_max_quad}. 
First, the choice of initial $\hat{\mu}_1$ is robust. Across $\hat{\mu}_1$ in $\bcbra{1,10,100,1000}$, the final objective values are nearly identical to the reference result. 
Second, the algorithm does not always recover the true $\mu^*$. Even when the final estimated QG modulus $\hat{\mu}_s$ exceeds $\mu^*$, the objective value remains close to the optimum (the absolute error is on the order of $10^{-6}$).  For instance, when $k=50$, $L=10.0$ and $\hat{\mu}_1=100$, the final QG estimate is $\hat{\mu}_s=25$, much larger than $\mu^*=1.0$, yet the absolute error of the objective is small.  
Third, the iteration trends align with our intuition about how the initial QG estimate affects convergence. A larger initial $\hat{\mu}_1$ leads to more aggressive lower bound updates, which generally results in fewer iterations when no adjustment to the estimate is needed. This is clearly demonstrated when $(k,L)=(200,1000)$, where the final QG estimate $\hat{\mu}_s$ remains equal to the initial $\hat{\mu}_1$ throughout optimization. In these cases, larger initial values consistently reduce the iteration count through more aggressive updates. However, for the remaining cases, while this trend holds for moderate initial estimates ($\hat{\mu}_1\in \bcbra{1,10}$), it breaks down at 
$\hat{\mu}_1\in \bcbra{100,1000}$.
 With such an aggressive initial estimate, the algorithm must spend additional iterations adjusting $\hat{\mu}_1$ downward during optimization, ultimately leading to slower convergence. 
Fourth, a negative reported error of the objective value means that the objective value of {\rapex} is slightly below the reference value, which is attributable to numerical tolerances.
Finally, as highlighted in Proposition~\ref{prop:upper_bound_LbarN} and the corresponding remark, the convergence rate's dependence on the smoothness parameter is governed by the specific pieces traversed during the iterations. Thus, even when the number of cuts is less than the total number of pieces, {\rapex} can still effectively capture the underlying $\pws$ structure in many cases—such as $(k, L) = (200, 1000)$ and $(k, L) = (50, 100)$—and therefore still exhibits strong convergence behavior. In these instances, having fewer cuts than pieces does not hamper the overall convergence behavior. However, for $(k, L) = (200, 5.0)$ and $(k, L) = (200, 10.0)$, we observe that significantly more iterations are required. We suspect this is because the gap between $\mu$ and $L$ is small, resulting in many more active local pieces, and consequently, $m = 50$ cuts are insufficient to capture the local $\pws$ structure.

\begin{table}[htbp]
  \centering
  \caption{Comparison of {\rapex} on {\maxquad} in different initial QG modulus $\hat{\mu}_1$}
    \begin{tabular}{ccrrrrrr}
    \toprule
    \multicolumn{1}{l}{$k$} & \multicolumn{1}{l}{$L$} & \multicolumn{1}{l}{$\hat{\mu}_1$} & \multicolumn{1}{l}{$\hat{\mu}_S$} & \multicolumn{1}{l}{Oracle Num.} & \multicolumn{1}{l}{Obj($f(\bar{y}^{S})$)} & \multicolumn{1}{l}{Opt($f^*$)} & \multicolumn{1}{l}{$f(\bar{y}^{S})-f^*$} \\
    \midrule
    \multirow{16}[8]{*}{50} & \multirow{4}[2]{*}{5} & 1     & 1.00E+00 & 4043  & -6.91E-01 & -6.91E-01 & 2.08E-07 \\
          &       & 10    & 1.00E+01 & 2218  & -6.91E-01 & -6.91E-01 & -4.26E-08 \\
          &       & 100   & 6.25E+00 & 2494  & -6.91E-01 & -6.91E-01 & 1.36E-07 \\
          &       & 1000  & 3.91E+00 & 2542  & -6.91E-01 & -6.91E-01 & 6.06E-08 \\
\cmidrule{2-8}          & \multirow{4}[2]{*}{10} & 1     & 1.00E+00 & 2785  & -1.74E-01 & -1.74E-01 & 2.09E-07 \\
          &       & 10    & 1.00E+01 & 2003  & -1.74E-01 & -1.74E-01 & 1.19E-07 \\
          &       & 100   & 2.50E+01 & 1565  & -1.74E-01 & -1.74E-01 & 3.77E-07 \\
          &       & 1000  & 1.56E+01 & 1903  & -1.74E-01 & -1.74E-01 & 1.02E-07 \\
\cmidrule{2-8}          & \multirow{4}[2]{*}{100} & 1     & 1.00E+00 & 978   & 1.22E+00 & 1.22E+00 & 2.68E-07 \\
          &       & 10    & 1.00E+01 & 753   & 1.22E+00 & 1.22E+00 & 1.82E-07 \\
          &       & 100   & 1.00E+02 & 562   & 1.22E+00 & 1.22E+00 & 2.45E-07 \\
          &       & 1000  & 6.25E+01 & 671   & 1.22E+00 & 1.22E+00 & 3.83E-07 \\
\cmidrule{2-8}          & \multirow{4}[2]{*}{1000} & 1     & 1.00E+00 & 1152  & 2.34E+00 & 2.34E+00 & 2.25E-07 \\
          &       & 10    & 1.00E+01 & 810   & 2.34E+00 & 2.34E+00 & 1.92E-07 \\
          &       & 100   & 1.00E+02 & 482   & 2.34E+00 & 2.34E+00 & 1.94E-07 \\
          &       & 1000  & 2.50E+02 & 350   & 2.34E+00 & 2.34E+00 & 3.26E-07 \\
    \midrule
    \multirow{16}[8]{*}{200} & \multirow{4}[2]{*}{5} & 1     & 1.00E+00 & 195088 & 7.41E-01 & 7.41E-01 & 1.10E-05 \\
          &       & 10    & 1.00E+01 & 312249 & 7.41E-01 & 7.41E-01 & 8.04E-07 \\
          &       & 100   & 1.00E+02 & 136814 & 7.41E-01 & 7.41E-01 & 1.08E-06 \\
          &       & 1000  & 1.00E+03 & 128509 & 7.41E-01 & 7.41E-01 & 2.02E-06 \\
\cmidrule{2-8}          & \multirow{4}[2]{*}{10} & 1     & 1.00E+00 & 408943 & 1.06E+00 & 1.06E+00 & 3.55E-07 \\
          &       & 10    & 1.00E+01 & 33833 & 1.06E+00 & 1.06E+00 & 5.57E-07 \\
          &       & 100   & 2.50E+01 & 32431 & 1.06E+00 & 1.06E+00 & 4.49E-07 \\
          &       & 1000  & 1.56E+01 & 15204 & 1.06E+00 & 1.06E+00 & 5.01E-07 \\
\cmidrule{2-8}          & \multirow{4}[2]{*}{100} & 1     & 1.00E+00 & 1192  & 2.22E+00 & 2.22E+00 & 3.26E-07 \\
          &       & 10    & 1.00E+01 & 984   & 2.22E+00 & 2.22E+00 & 2.78E-07 \\
          &       & 100   & 1.00E+02 & 684   & 2.22E+00 & 2.22E+00 & 1.53E-07 \\
          &       & 1000  & 6.25E+01 & 761   & 2.22E+00 & 2.22E+00 & 4.02E-07 \\
\cmidrule{2-8}          & \multirow{4}[2]{*}{1000} & 1     & 1.00E+00 & 1080  & 2.96E+00 & 2.96E+00 & 1.56E-07 \\
          &       & 10    & 1.00E+01 & 679   & 2.96E+00 & 2.96E+00 & 1.65E-07 \\
          &       & 100   & 1.00E+02 & 329   & 2.96E+00 & 2.96E+00 & 2.22E-07 \\
          &       & 1000  & 1.00E+03 & 190   & 2.96E+00 & 2.96E+00 & 3.33E-07 \\
    \bottomrule
    \end{tabular}%
    \label{tab:com_rapex_max_quad}
\end{table}%

We further plot the convergence of the upper bound $(\bar{f}^s=f(\bar{y}^{s}))$ and lower bound $(\underline{f}^s=f(\bar{y}^s) - \tilde{\Delta}_s)$ in Figure~\ref{fig:rapex_convergence_300_50} and Figure~\ref{fig:rapex_convergence_300_200}, where the $y$-axis denotes the absolute error $\bar{f}^s-f^*$ or $\underline{f}^s-f^*$. 
Four observations follow. 
First, excluding the $\hat{\mu}_1 = 1000$ case in Figure~\ref{fig:rapex_convergence_300_50}, convergence is generally faster with larger initial $\hat{\mu}_1$, since a larger estimate implies a more aggressive lower bound update, which indirectly accelerates level adjustment and upper bound improvement. 
Second, inaccurate QG estimates can lead to incorrect lower bounds exceeding the optimal value, resulting in vertical segments in Figure~\ref{fig:rapex_convergence_300_50} where {\rapex} adjusts the QG and updates the lower bound estimates. 
Third,  linear convergence is evident in both datasets, which matches the theoretical result in Theorem~\ref{thm:mu_unknown} and Theorem~\ref{thm:unknown_QG_opt_Gap}. In Figure~\ref{fig:rapex_convergence_300_50}, the total number of pieces $(50)$ matches the number of cuts used by {\rapex}, while in the Figure~\ref{fig:rapex_convergence_300_200}, the number of cuts $(50)$ is smaller than the total PWS pieces $(200)$. Nevertheless, since {\rapex} relies on local pieces during iterations, it successfully captures the PWS structure and still demonstrates linear convergence.
{Finally, the lower- and upper-bound improvements are nearly symmetric, which is consistent that our certificate implies a constant factor of the optimality gap.}
\begin{figure}
\raggedright{}%
\begin{minipage}[t]{0.45\columnwidth}%
\begin{center}
\includegraphics[width=\textwidth,height=0.8\textwidth]{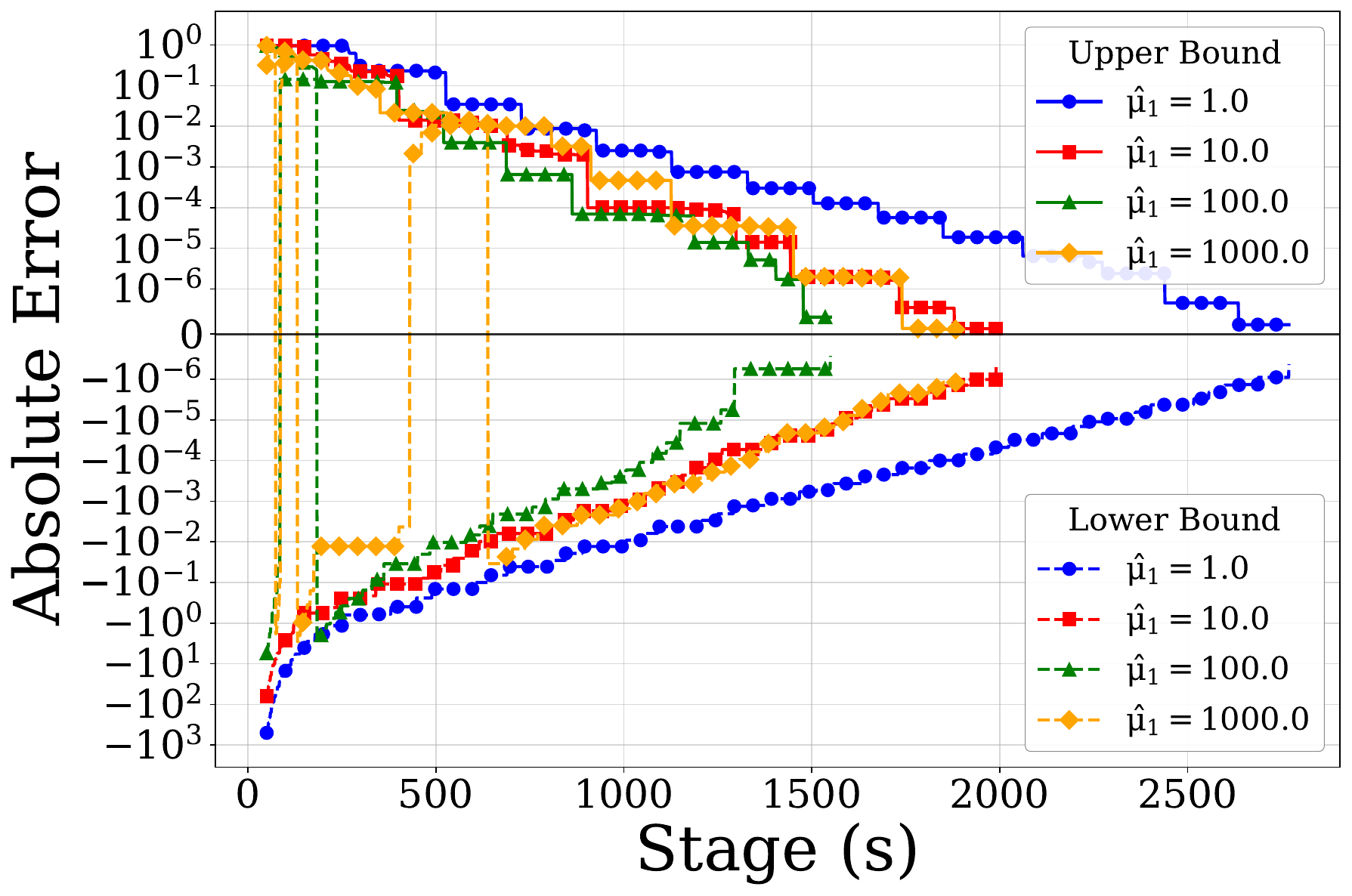}%
\caption{Absolute error of upper bound and lower bound of {\rapex} vs. Stage ($s$ in Algorithm~\ref{alg:mu_unknown}) on {\maxquad} with $k=50,L=10$ in Table~\ref{tab:com_rapex_max_quad}. \label{fig:rapex_convergence_300_50}}
\par\end{center}%
\end{minipage}\hfill%
\begin{minipage}[t]{0.45\columnwidth}%
\begin{center}
\includegraphics[width=\textwidth,height=0.8\textwidth]{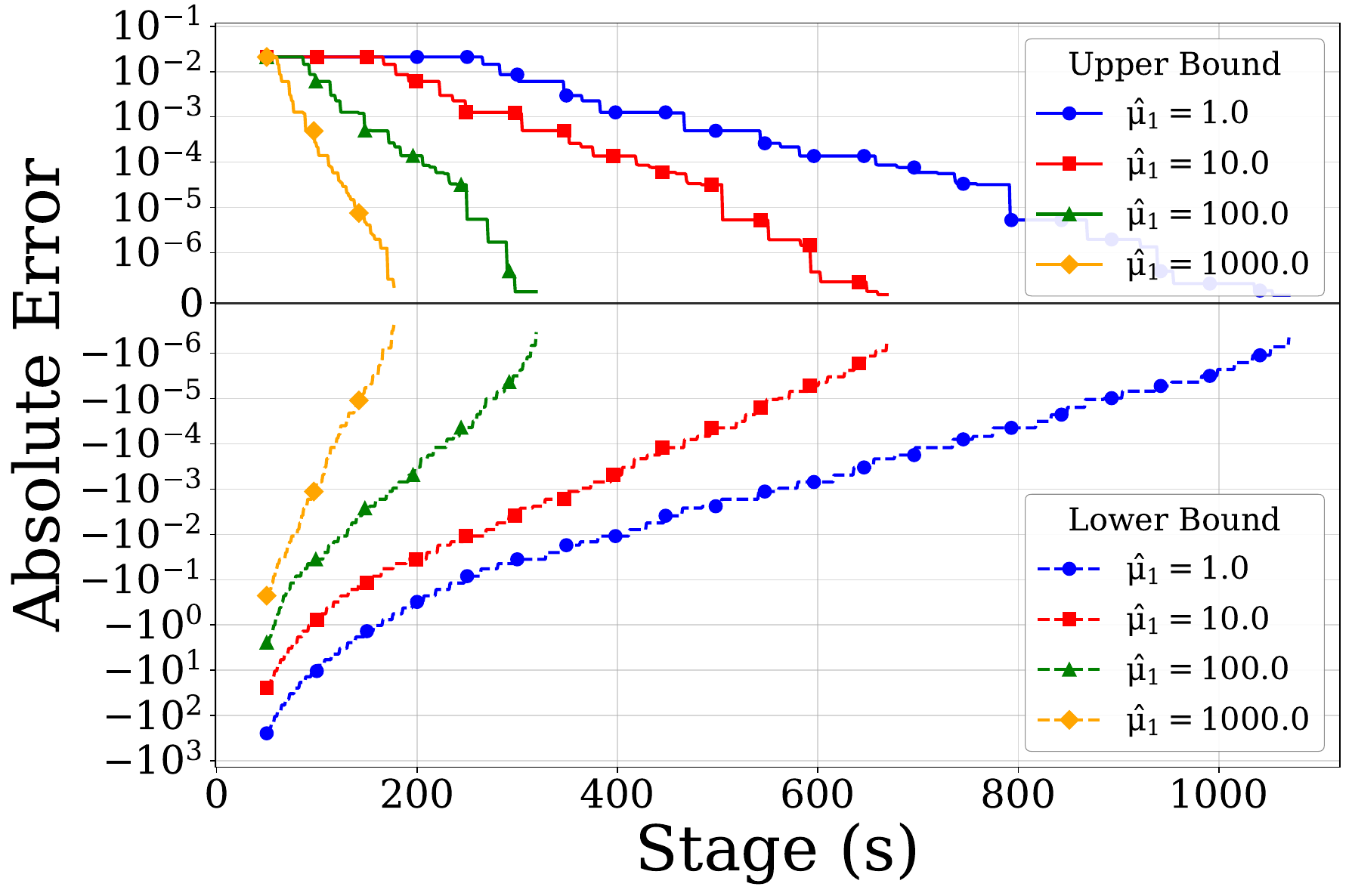}%
\caption{Absolute error of upper bound and lower bound of {\rapex} vs. Stage ($s$ in Algorithm~\ref{alg:mu_unknown}) on {\maxquad} with $k=200,L=1000$ in Table~\ref{tab:com_rapex_max_quad}. \label{fig:rapex_convergence_300_200}}
\par\end{center}%
\end{minipage}
\end{figure}

\paragraph{Comparison with other methods on {\maxquad}.}
We compare the following methods:
\begin{itemize}
\item $\rapex$ (Algorithm~\ref{alg:mu_unknown}): implemented in Julia, the subproblem is solved by ~\href{https://github.com/JuliaLinearAlgebra/NonNegLeastSquares.jl}{NonNegLeastSquares.jl}.
\item restarted APL~\cite{lan2015bundle}: implemented in Julia, the lower bound oracle is solved by Gurobi and the upper bound oracle is solved by ~\href{https://github.com/JuliaLinearAlgebra/NonNegLeastSquares.jl}{NonNegLeastSquares.jl}.
\item Proximal Bundle Method~\cite{BundleMethod.jl.0.4.0}: implemented in Julia, the subproblem is solved by Ipopt. 
\item Subgradient regularized descent (SRD) and SRD\_adapt~\cite{srd}: implemented in Matlab, see website~\footnote{\url{https://github.com/lhyoung99/subgradient-regularization}}.
\end{itemize}

\textbf{Experimental setup.} For all experiments summarized in Table~\ref{tab:compare_method_max_quad_summary}, we fix the number of pieces at $k=50$. For $d=500$, we set the strong convexity parameter $\mu = 1.0$ and vary the smoothness parameter $L$ over the set $\bcbra{5,10,100,1000}$. For large-scale instances, we do not explicitly control $\mu$ or $L$; instead, eigenvalues are sampled randomly and a shared set of eigenvectors is used for all matrices to reduce memory usage.

\textbf{Method-specific settings.} We impose $\hat{\mu}_1=100$ for $\rapex$ and a box constraint $-10^3 \leq x \leq 10^3$ for rAPL, which requires optimization over a bounded set. Since both SRD and SRD\_adapt compute the gradient of every function in the max structure at each iteration, we report their total number of gradient oracle calls as the product of the number of iterations and the number of pieces ($k=50$), enabling a fair comparison. Across all experiments, we impose a maximum time limit of 3 hours and a cap of $20,\!000$ iterations.

\textbf{Performance criteria.} For small-scale problems ($d=500$), we use the Ipopt solution as the reference solution; for large-scale instances ($d=10{,}000$ and $d=20{,}000$), where a high-accuracy external optimum is unavailable, we use the best objective value achieved by any tested method as a reference value. Table~\ref{tab:compare_method_max_quad_summary} reports the number of gradient oracle calls required by each method to reach $f(x)-f^* \leq 10^{-6}$ (small-scale) and $f(x)-f^* \leq 10^{-5}$ (large-scale).

\textbf{Key findings.}
We highlight three main observations from Table~\ref{tab:compare_method_max_quad_summary}. First, for small-scale problems, {\rapex} consistently requires fewer gradient oracle calls than SRD, SRD\_adapt, and rAPL across all tested smoothness regimes. The proximal bundle method achieves the lowest number of oracle calls only when $L/\mu=100/1.0$ and $1000/1.0$, but in all other settings, {\rapex} is competitive or superior. Since we set the number of cuts $\bundleSize=50$, {\rapex} can fully exploit the piecewise smoothness, yielding strong empirical performance in line with theoretical expectations.
Second, for large-scale problems, {\rapex} maintains a clear advantage, requiring dramatically fewer gradient oracle calls than SRD and SRD\_adapt. Competing approaches such as rAPL and the proximal bundle method are unable to handle these larger problem sizes, as indicated by the dashes in the table. 
Third, we observe an interesting phenomenon with the proximal bundle method: it fails to achieve the required accuracy for $L/\mu = 5/1.0$ and $10/1.0$ within the given time and iteration limits, but outperforms other methods when $L/\mu = 100/1.0$ and $1000/1.0$. One possible interpretation is that proximal bundle method may not efficiently exploit the $\pws$ structure on these instances. This potential limitation seems to be more pronounced when $L/\mu$ is small, possibly because the number of active local pieces increases in such regimes. Evidence consistent with this interpretation appears in the increasing number of gradient calls required by {\rapex} as $L/\mu$ decreases (for example, in the cases with $(k, L)\in \bcbra{(200, 5), (200, 10)}$ in Table~\ref{tab:com_rapex_max_quad}). 

\begin{table}[htbp]
  \centering
  \caption{Comparison of gradient oracle calls for different methods on {\maxquad} with $k=50$. Each cell reports the mean (top) and standard deviation (bottom) over 5 runs with different random seeds.}
    \begin{tabular}{ccrccrr}
    \toprule
    \multicolumn{1}{l}{$d$} & \multicolumn{1}{l}{$L/\mu$} & \multicolumn{1}{l}{\rapex} & \multicolumn{1}{l}{rAPL} & \multicolumn{1}{l}{prox-bundle} & \multicolumn{1}{l}{SRD} & \multicolumn{1}{l}{SRD\_adapt} \\
    \midrule
    \multirow{8}[8]{*}{500} & \multirow{2}[2]{*}{5/1.0} & \textbf{2.40E+03} & \multicolumn{1}{r}{3.28E+03} & \multirow{2}[2]{*}{-} & 2.34E+04 & 2.16E+04 \\
          &       & {$\pm \textbf{9.48E+01}$} & {$\pm$ 8.58E+02} &       & {$\pm$ 2.48E+03} & {$\pm$ 1.58E+03} \\
\cmidrule{2-7}          & \multirow{2}[2]{*}{10/1.0} & \textbf{2.22E+03} & \multicolumn{1}{r}{3.34E+03} & \multirow{2}[2]{*}{-} & 1.58E+05 & 1.81E+05 \\
          &       & {$\pm$ \textbf{9.46E+01}} & {$\pm$ 4.07E+02} &       & {$\pm$ 7.44E+03} & {$\pm$ 7.40E+03} \\
\cmidrule{2-7}          & \multirow{2}[2]{*}{100/1.0} & 7.40E+02 & \multicolumn{1}{r}{1.25E+03} & \multicolumn{1}{r}{\textbf{4.69E+02}} & 4.28E+04 & 4.33E+04 \\
          &       & {$\pm$ 1.51E+02} & {$\pm$ 7.07E+01} & {$\pm$ \textbf{5.43E+01}} & {$\pm$ 9.16E+03} & {$\pm$ 8.93E+03} \\
\cmidrule{2-7}          & \multirow{2}[2]{*}{1000/1.0} & 2.00E+02 & \multicolumn{1}{r}{9.05E+02} & \multicolumn{1}{r}{\textbf{1.02E+02}} & 1.17E+06 & 1.33E+06 \\
          &       & {$\pm$ 8.08E+00} & {$\pm$ 4.92E+01} & {$\pm$ \textbf{3.20E+00}} & {$\pm$ 5.24E+05} & {$\pm$ 5.63E+05} \\
    \midrule
    \multirow{2}[2]{*}{10000} & \multirow{2}[2]{*}{-} & \textbf{2.49E+03} & \multirow{2}[2]{*}{-} & \multirow{2}[2]{*}{-} & 2.69E+04 & 2.59E+04 \\
          &       & {$\pm$ \textbf{1.66E+02}} &       &       & {$\pm$ 1.04E+03} & {$\pm$ 6.35E+02} \\
    \midrule
    \multirow{2}[2]{*}{20000} & \multirow{2}[2]{*}{-} & \textbf{2.70E+03} & \multirow{2}[2]{*}{-} & \multirow{2}[2]{*}{-} & 2.77E+04 & 2.48E+04 \\
          &       & {$\pm$ \textbf{1.24E+02}} &       &       & {$\pm$ 5.08E+02} & {$\pm$ 8.37E+02} \\
    \bottomrule
    \end{tabular}%
  \label{tab:compare_method_max_quad_summary}%
\end{table}%

\subsection{\protect\label{subsec:two_sto_lp}Two-stage stochastic linear programming}

In this subsection, we consider the two-stage stochastic linear programming:
\begin{equation}
\min_{l_{x}\leq x\leq u_{x}, x\in \mbb R^{n_1}}\ c_{1}^{\top}x+\mbb E_{\xi\sim \Xi}\bsbra{Q\brbra{x;\xi}}\ \st\ A_{1}x\Join b_{1}\ ,\  A_1 \in \mbb R^{m_1 \times n_1}
\end{equation}
where 
\begin{equation}
  \begin{aligned}
Q\brbra{x;\xi}=\min_{l_{y}\leq y\leq u_{y}, y\in \mbb R^{n_2}}& c\brbra{\xi}^{\top}y\ \\
\st\ & A_{1}\brbra{\xi}x+A_{2}\brbra{\xi}y\Join b_{2}\brbra{\xi},A_1(\xi) \in \mbb R^{m_2\times n_1},A_2(\xi) \in \mbb R^{m_2 \times n_2}
\end{aligned}.
\end{equation}
Here, symbol $\Join\in \bcbra{\geq, =, \leq}$, depending on the specific problem instance.

We use benchmark datasets from~\cite{linderoth2006empirical}, which are available at website~\footnote{\url{https://pages.cs.wisc.edu/~swright/stochastic/sampling/}}.
We consider different sample sizes of the stochastic scenario, the specific problem dimension is described in the following Table~\ref{tab:2stslp_data_description}.

\begin{table}[htbp]
  \centering
  \caption{Problem dimensions and number of scenarios for two-stage stochastic linear programming test instances.}
    \begin{tabular}{lrrrrl}
    \toprule
    Problem & \multicolumn{1}{l}{$n_1$} & \multicolumn{1}{l}{$m_1$} & \multicolumn{1}{l}{$n_2$} & \multicolumn{1}{l}{$m_2$} & Total. Scen ($|\Xi|$) \\
    \midrule
    storm & 121   & 185   & 1259  & 528   & \{8,27,125,1000\} \\
    term20 & 63    & 3     & 764   & 124   & \{50,300\} \\
    \bottomrule
    \end{tabular}%
  \label{tab:2stslp_data_description}%
\end{table}%

For the two-stage stochastic linear programming experiments, we compare the performance of $\rapex$ and rAPL with the number of cutting planes set to $\bundleSize=50$.  We set the initial $\hat{\mu}_1=100$ for $\rapex$ and enforce box constraints in rAPL, setting the lower and upper bounds to $-10^3$ and $10^3$, respectively. All subproblems for both algorithms are solved using Gurobi~\cite{gurobi}. 
Table~\ref{tab:compare_two_slp} summarizes the number of gradient oracle calls required by each method to reach an optimality gap of $f(x) - f^* \leq \bcbra{10^{-3},10^{-4}}$. The reference value $f^*$ is computed by solving the corresponding deterministic linear program—incorporating all scenarios as constraints—using Gurobi.
Two key observations can be drawn from the results. First, $\rapex$ consistently outperforms rAPL across all datasets and tolerance levels, with the sole exception of the term20 instance with 300 scenarios at the $10^{-3}$ tolerance, where rAPL performs slightly better. Second, as the required tolerance becomes more stringent, the advantage of $\rapex$ becomes even more pronounced: rAPL often fails to find a satisfactory solution within the 3-hour time limit and 20,000-iteration cap, especially for a $10^{-4}$ tolerance. In contrast, the number of gradient oracle calls required by $\rapex$ to reach $10^{-4}$ accuracy remains nearly the same as that needed for $10^{-3}$ accuracy. For example, in the term20 instance with 50 scenarios, the optimality gap decreases directly below $10^{-4}$, resulting in an identical iteration count for both tolerance levels.


\begin{table}[htbp]
  \centering
  \caption{Comparison of gradient oracle calls for two methods on two-stage stochastic linear programming}
    \begin{tabular}{l|r|rr|rr}
    \toprule
    \multicolumn{1}{c|}{\multirow{2}[2]{*}{dataset}} & \multicolumn{1}{c|}{\multirow{2}[2]{*}{$f^*$}} & \multicolumn{2}{c|}{$10^{-3}$} & \multicolumn{2}{c}{$10^{-4}$} \\
          &       & \multicolumn{1}{l}{$\rapex$} & \multicolumn{1}{l|}{rAPL} & \multicolumn{1}{l}{$\rapex$} & \multicolumn{1}{l}{rAPL} \\
    \midrule
    storm8 & 1.500384E+07 & \textbf{282} & 446   & \textbf{332} & - \\
    storm27 & 1.500192E+07 & \textbf{510} & 620   & \textbf{527} & - \\
    storm125 & 1.500151E+07 & \textbf{398} & 704   & \textbf{452} & -  \\
    storm1000 & 1.522862E+07 & \textbf{384} & 555   & \textbf{2448} & - \\
    term20(50) & 2.559739E+05 & \textbf{1254} & 1258  & \textbf{1254} & 1696 \\
    term20(300) & 2.549250E+05 & 1341  & \textbf{1308}  & \textbf{1578} & 1747 \\
    \bottomrule
    \end{tabular}%
  \label{tab:compare_two_slp}%
\end{table}%


\renewcommand \thepart{}
\renewcommand \partname{}

\bibliographystyle{abbrvnat}
\bibliography{./ref.bib}


\newpage
\appendix

\addcontentsline{toc}{section}{Appendix}
\part{Appendix} 

\section{Solving the subproblem}

The bundle-level subproblem can be written as
\begin{equation}
\min_{x\in\mbb R^{n}}\frac{1}{2}\norm{x-\bar{y}}^{2}\ ,\ \st Ax\leq b\ ,\label{eq:bundle_subproblem_appendix}
\end{equation}
where each row of $A$ together with the corresponding element of
$b$ defines a cut constraint. The quality of our certificate depends
critically on the accuracy of this subproblem. Many existing solvers
do not meet our accuracy requirement, and the problem often becomes
harder as the procedure progresses, frequently due to strong linear
dependence among the rows of $A$. It is therefore important to provide
an effective solution method.

An important observation is that the dual of~\eqref{eq:bundle_subproblem_appendix}
is
\begin{equation}
\min_{\lambda\geq0}\ \ \frac{1}{2}\lambda^{\top}(AA^{\top})\lambda+\lambda^{\top}\brbra{b-A\bar{y}}\ ,
\end{equation}
which is a nonnegative quadratic program. Furthermore, if there exists
$\tilde{b}$ such that $A\tilde{b}=b$, then the problem reduces to
\begin{equation}
\min_{\lambda\geq0}\ \ \frac{1}{2}\norm{A^{\top}\lambda-(\bar{y}-\tilde{b})}^{2}\ ,
\end{equation}
which is a classical nonnegative least squares problem. There is extensive
literature on this topic, for example~\cite{lawson1995solving,bro1997fast,kim2011fast}.
In addition, an open-source Julia package implements related algorithms,
see~\href{https://github.com/JuliaLinearAlgebra/NonNegLeastSquares.jl}{NonNegLeastSquares.jl}.
In our setting, when $A$ has full row rank, which is typical in our small-bundle setting after removing linearly dependent cuts, a vector $\tilde{b}$ satisfying $A\tilde{b}=b$ exists.
On the other hand, when such a vector does not exist, the nonnegative
quadratic program can still be solved by a modified algorithm for nonnegative
least squares with high accuracy. After solving the dual to obtain
$\lambda^{*}$, we recover $x^{*}=\bar{y}-A^{\top}\lambda^*$ and evaluate
the constraint violation $\norm{\bsbra{Ax^{*}-b}_{+}}$. If this quantity
is below a prescribed tolerance, we accept $x^{*}$; otherwise, we
declare the subproblem infeasible.

\end{document}